\documentclass[12pt]{amsart}
\usepackage{amsthm}
\usepackage{amsmath}
\usepackage[margin=1.25in]{geometry}
\usepackage{amssymb}
\usepackage{verbatim}
\usepackage{cleveref}
\usepackage{longtable,enumitem,color}
\usepackage{stmaryrd}
\usepackage{graphicx}
\usepackage{url}
\usepackage{amscd}
\usepackage[all]{xy} 
\usepackage{enumitem}
\usepackage{tikz-cd}
	{\providecommand{\noopsort}[1]{}} 
\usepackage{chngcntr}\counterwithin{table}{section}

\numberwithin{figure}{section}

\theoremstyle{plain}
\newtheorem{theorem}{Theorem}[section]
\newtheorem{corollary}[theorem]{Corollary}  
\newtheorem{lemma}[theorem]{Lemma}

\newtheorem{proposition}[theorem]{Proposition}

\newtheorem{main}{Theorem}

\theoremstyle{definition}
\newtheorem{definition}[theorem]{Definition}
\newtheorem{thm_def}[theorem]{Theorem and Definition}

\theoremstyle{remark}
\newtheorem{remark}[theorem]{Remark}
\newtheorem{example}[theorem]{Example}
\numberwithin{equation}{section}

\newcommand{\Z}{\mathbb{Z}}\newcommand{\Q}{\mathbb{Q}}
\newcommand{\R}{\mathbb{R}}
\newcommand{\s}{\mathbb{S}}

\newcommand{\SO}{\mathrm{SO}}

\newcommand{\gO}{\mathsf{O}}

\DeclareMathOperator{\cod}{cod}

\DeclareMathOperator{\rank}{rank}

\DeclareMathOperator{\sys}{sys}
\DeclareMathOperator{\si}{si}

\newcommand{\of}[1]{\left(#1\right)}
\newcommand{\ceil}[1]{\left\lceil #1 \right\rceil}
\newcommand{\floor}[1]{\left\lfloor #1 \right\rfloor}

\newcommand{\inner}[1]{\left\langle #1 \right\rangle}
\newcommand{\st}{~|~}
\newcommand{\tensor}{\otimes}

\renewcommand{\subset}{\subseteq}

\newcommand{\Hom}{{\operatorname{Hom}}}                            

\newcommand{\gS}{\mathsf{S}}

\newcommand{\RP}{\mathbb{R\mkern1mu P}}
\newcommand{\CP}{\mathbb{C\mkern1mu P}}
\newcommand{\HP}{\mathbb{H\mkern1mu P}}
\newcommand{\CaP}{\mathrm{Ca}\mathbb{\mkern1mu P}^2}

\renewcommand{\R}{{\mathbb{R}}}
\renewcommand{\Z}{{\mathbb{Z}}}

\newcommand{\gF}{\ensuremath{\operatorname{\mathsf{F}}}}
\newcommand{\gT}{{\ensuremath{\operatorname{\mathsf{T}}}}}

\renewcommand{\SO}{\ensuremath{\operatorname{\mathsf{SO}}}}

\newcommand{\gU}{\ensuremath{\operatorname{\mathsf{U}}}}

\def\con#1=#2(#3){#1 \equiv #2 \bmod{#3}}

\renewcommand{\rank}{\ensuremath{\operatorname{rk}}}

\DeclareMathOperator{\pr}{pr}

\DeclareMathOperator{\Id}{Id}

\title{Positive curvature, torus symmetry, and matroids}
\author{Lee Kennard}
\address{\hspace*{-0.3em} Department of Mathematics, Syracuse University, Syracuse, NY 13244, U.S.A.}\email{ltkennar@syr.edu}
\author{Michael Wiemeler}
\address{\hspace*{-0.3em} Mathematisches Institut, WWU M\"unster, Einsteinstr. 62, 48149 M\"unster, Germany}\email{wiemelerm@uni-muenster.de}
\author{Burkhard Wilking}
\address{\hspace*{-0.3em} Mathematisches Institut, WWU M\"unster, Einsteinstr. 62, 48149 M\"unster, Germany} \email{ wilking@uni-muenster.de}
\date{\today}
\begin{document}


\begin{abstract}
We identify a link between regular matroids and torus representations all of whose isotropy groups have an odd number of components. Applying Seymour's 1980 classification of the former objects, we obtain a classification of the latter. In addition, we prove optimal upper bounds for the cogirth of regular matroids up to rank nine, and we apply this to prove the existence of fixed-point sets of circles with large dimension in a torus representation with this property up to rank nine. Finally, we apply these results to prove new obstructions to the existence of Riemannian metrics with positive sectional curvature and torus symmetry.
\end{abstract}

\maketitle

In \cite{KWW}, the authors computed the rational cohomology of closed, oriented, positively curved Riemannian manifolds with $\gT^7$ symmetry, under the additional assumption that the odd Betti numbers vanish. In \cite{Nienhaus-PhD}, Nienhaus proved that the symmetry assumption can be relaxed to $\gT^6$. Here, we provide similar computations in the case where the topological assumption on the Betti numbers is replaced by a geometric one on the torus action.

\begin{main}\label{thm:t9}
Let $M^n$ be a closed, oriented, positively curved Riemannian manifold. Assume $\gT^d$ acts effectively by isometries with a fixed point and the property that all isotropy groups in a neighborhood of the fixed point have an odd number of components. 
	\begin{enumerate}
	\item If $d = 9$, then $M$ is a rational cohomology \(\s^n,\CP^{\frac n 2}\), or \(\HP^{\frac n 4}\).
	\item If $d = 6$, then $M$ is a rational cohomology \(\s^n,\CP^{\frac n 2}\), or \(\HP^{\frac n 4}\) up to degree $\tfrac n 3$. 	\end{enumerate}
\end{main}

The conclusion in  (2) means that up to degree $\tfrac{n}{3}$ 
the rational  cohomology of $M$ is generated by one element $x$ with 
$\deg(x)\in \{0,2,4\}$.

The condition on isotropy groups is automatic when all isotropy groups are connected, and there are a few reasons to study the latter condition. First, it generalizes the notion of free or semi-free circle actions, which are the best behaved geometrically. 

Second, a reduction to the case of connected isotropy groups and an analysis of representations with this property led to the breakthrough in \cite{KWW}. We now know of at least three proofs of the $\gS^1$ Splitting Theorem from that paper, and the first one passed through this case. This paper extends this analysis. 

Third, when an isotropy group is connected or more generally has an odd number of components, the normal bundle of each of its fixed-point components has a complex structure. This condition played an important role in the Nienhaus' work \cite{Nienhaus-PhD}.

Finally, we note that any torus action on a manifold naturally gives rise to an induced action on a submanifold having only connected isotropy groups. The construction is as follows. If we have a smooth manifold $M^n$ with an effective action of $\gT^d$, then we can always look at a maximal finite isotropy group $\gF$ and pass to the fixed point component $N$ of $\gF$ endowed with the action of $\gT^d/\gF\cong \gT^d$. If moreover $M$ has positive sectional curvature and the action is isometric, then depending on the parity of $n$, either $\gT^d$ 
or a codimension 1 subtorus has a fixed point in $N$, and Theorem~\ref{thm:t9} thus determines the rational cohomology of $N$ if $d\ge 10$.

The main technical work here is to continue the analysis of special torus representations in \cite{KWW}. In that paper, the main new tool was a splitting theorem for torus representations, and an important step in the proof was to reduce to the case where all the isotropy groups of the representation are connected. Here we classify torus representations with connected isotropy groups and, more generally, with the property that every isotropy group has an odd number of components. This classification is accomplished by observing that the matrix of weight vectors satisfies a condition on the determinants of submatrices called {\it total unimodularity}. Up to equivalence, these matrices correspond to combinatorial objects called {\it regular matroids}. In a celebrated paper \cite{Seymour80}, Seymour classified regular matroids, so by retracing our steps we obtain a {\it classification of torus representations all of whose isotropy groups have an odd number of components} (see Section \ref{sec:cigClassification}).

To prove Theorem \ref{thm:t9}, we then need a detailed analysis of the codimensions of the fixed-point sets of circles inside a torus representation without finite isotropy groups of even order. Using Seymour's theorem, the most difficult case involves representations that arise from finite graphs as follows (see Section \ref{sec:cigClassification} for details): Given a finite graph with first Betti number $b$, via cellular homology each edge $e \in E(G)$ gives rise to a homomorphism
	\[e^*:H_1(G;\Z) \to \Z,\]
which can viewed as the weight of an irreducible representation of the torus 
	\[\gT^b \cong H_1(G;\R)/H_1(G;\Z).\]
Putting together the representations from each edge, we obtain what we call a cographic representation. To account for multiplicities, we assign weights to the edges using a non-negative function $\lambda:E(G) \to \R$. The codimensions of fixed-point sets of circles in $\gT^d$ then correspond to lengths of cycles in the weighted graph $(G,\lambda)$, so our problem turns into proving upper bounds for the {\it graph systole}, 
	\[\sys(G,\lambda) = \min_C \lambda(C)/\lambda(G),\]
where the minimum runs over cycles $C$ in the graph $G$ and where $\lambda(H) = \sum \lambda(e)$ is the sum of weights of edges in a subgraph $H$ of $G$. Notice that $\sys(G,\lambda)$ equals the classical girth of $G$ divided by the number of edges in $G$ when $\lambda$ is a constant function. Upper bounds on $\sys(G,\lambda)$ that depend only on the (first) Betti number $b$ exist by a routine compactness argument, and optimal upper bounds are known for $b \leq 6$. For our purposes, we need the values for $b \leq 9$.

\begin{main}[Theorem \ref{thm:bounds-cogr-matr}]\label{thm:systole}
For $b \leq 9$, the systole of a finite, weighted graph $(G,\lambda)$ with first Betti number \(b\) satisfies $\sys(G,\lambda) \leq s(b)$, where $s(b)$ is defined as follows:
  	\[
	\begin{array}{c|c|c|c|c|c|c|c|c|c}
	b		& 1 & 2 	& 3 & 4 		& 5 	& 6 	& 7 		& 8    & 9\\\hline
	s(b)	& 1 & 2/3 	& 1/2 & 4/9	& 3/8	& 1/3	& 3/10	& 2/7 & 1/4
	\end{array}.
      \]
Moreover, these bounds are optimal.
\end{main}

\begin{remark}
Asymptotic estimates have been shown by Bollob\'as and Szemer\'edi \cite{BollobasSzemeredi02}, and one consequence is that the optimal bound $s(b)$ satisfies $s(b) \leq 4 \ln b / b$ for all $b \geq 2$. An estimate of this form is used in Sabourau \cite{Sabourau08} to answer a question in Gromov \cite{Gromov96} on the separating systole of a surface of large genus. Sabourau's strategy is to embed a graph on the surface, relate the genus of the surface to the Betti number of the graph, and apply the Bollob\'as-Szemer\'edi bound to derive the existence of a short closed geodesic on the surface. In contrast to this work, we embed graphs into surfaces to derive bounds on the graph. 
\end{remark}

Returning to the general case of a torus representation all of whose isotropy groups have an odd number of components, we use Theorem \ref{thm:systole} together with an analysis of the other cases in Seymour's classification to prove the following:

\begin{main}[Theorem \ref{thm:cd}]\label{thm:cdIntro}
For a representation $\rho:\gT^d \to \SO(V)$ with $d \leq 9$ and the property that all isotropy groups have an odd number of components, the quantity
	\[c(\rho) = \min_{\gS^1 \subseteq \gT^d} \cod V^{\gS^1}/\dim V\]
is bounded from above by $c(d)$, where $c(d)$ is defined as follows:
 	\[
	\begin{array}{c|c|c|c|c|c|c|c|c|c}
	d	& 1 & 2 	& 3 & 4 	& 5 	& 6 	& 7 		& 8    & 9\\\hline
	c(d)	& 1 & 2/3 	& 1/2 & 4/9	& 2/5	& 1/3	& 3/10	& 2/7 & 1/4
	\end{array}.
      \]
Moreover, these bounds are optimal.
\end{main}

\begin{remark}
By the correspondence between torus representations as in Theorem \ref{thm:cdIntro} and regular matroids, this theorem implies that the {\it cogirth} of a regular matroid of rank $d \leq 9$ is bounded above by $c(d)$ (see Section \ref{sec:Optimization} for a definition). Similar bounds were proved already by Crenshaw and Oxley \cite{CrenshawOxley-pre} for the wider class of binary matroids, but their bounds are not sufficient for our geometric applications (see also \cite{ChoChenDing07}).
\end{remark}

Theorem \ref{thm:cdIntro} is the main new ingredient required for the proof of Theorem \ref{thm:t9} in the case of a $\gT^9$-action. For the case of a $\gT^6$-action, a refinement of Theorem \ref{thm:cdIntro} is required (see Theorem \ref{thm:6involutions}).

This paper is structured as follows. 
In Section \ref{sec:cigClassification}, we give examples and state the classification of torus representations with connected isotropy groups. 
In Section \ref{sec:systoles}, we prove Theorem \ref{thm:systole}. 
In Section \ref{sec:cogirth}, we introduce matroids, state Seymour's theorem, explain how to deduce the classification of torus representations with connected isotropy groups, and then prove Theorem \ref{thm:cdIntro}. 
In Section \ref{sec:t4orLess}, we review the Connectedness Lemma, related cohomological periodicity results, and Nienhaus' computation of the rational cohomology of fixed-point components of isometric $\gT^4$-actions on positively curved manifolds. 
In Section \ref{sec:t9}, we apply Theorem \ref{thm:cdIntro} to prove Theorem \ref{thm:t9}.

\setcounter{tocdepth}{1}
\tableofcontents

\subsection*{Acknowledgements} 
The authors wish to thank the referees for their attention to detail and suggestions to improve the exposition. The first author was partially supported by NSF Grants DMS-2005280 and DMS-2402129, and he is grateful to the Cluster of Excellence at WWU M\"unster for their support and hospitality during a visit in Summer 2022. The second and third authors were supported by the Deutsche Forschungsgemeinschaft (DFG, German Research Foundation) under Germany's Excellence Strategy EXC 2044--390685587, Mathematics M\"unster: Dynamics--Geometry--Structure and within CRC 1442 Geometry: Deformations and Rigidity at WWU M\"unster.

\bigskip
\section{Torus representations with connected isotropy groups}
\label{sec:cigClassification}

In this section, we describe how to build all torus representations with the property that all isotropy groups are connected (see Theorem \ref{thm:cigClassification}). The proof that this list is complete requires a deep theorem of Seymour and is given in Section \ref{sec:cogirth}. 

\subsection{Graphic and cographic representations}

Let $G$ be a finite, connected, directed graph with vertex set $V$ of size $d+1$ and edge set $E$ of size $n$. Fix any vertex \(v_0\in V\). 
The long exact (co)homology sequences for the triple \((G,V,v_0)\) reduce to short exact sequences as follows:
\[0\rightarrow H^0(V,v_0;\mathbb{Z})\rightarrow H^1(G,V;\mathbb{Z})\stackrel{\iota^*}{\rightarrow} H^1(G,v_0;\mathbb{Z})\rightarrow 0\]
and
\[0\leftarrow H_0(V,v_0;\mathbb{Z})\stackrel{\delta}{\leftarrow} H_1(G,V;\mathbb{Z})\leftarrow H_1(G,v_0;\mathbb{Z})\leftarrow 0.\]

Now \(H_1(G,V;\mathbb{Z})\) is the free \(\mathbb{Z}\)-module generated by the edges \(e\in E\) of \(G\). Moreover, for \(H^1(G,V;\mathbb{Z})=\Hom(H_1(G,V;\mathbb{Z}),\mathbb{Z})\) we can pick the dual basis consisting of elements \(e^*\) for \(e\in E\). 

For each edge $e \in E$, we get two irreducible torus representations as follows. 
	\begin{enumerate}
	\item 	The image \(\iota^*(e^*) \in H^1(G,v_0;\Z)\) of the basis element \(e^*\in H^1(G,V;\Z)\) gives rise to a linear map
          \[\iota^*(e^*):H_1(G,v_0;\mathbb{Z})\rightarrow \mathbb{Z}.\]
	We may view this as a weight of an irreducible representation
		\[\rho_e^*:\gT^b \cong H_1(G,v_0;\R)/H_1(G,v_0;\Z) \to \gU(1).\]
              \item The image $\delta(e) \in H_0(V,v_0;\Z)$ of the basis element $e \in H_1(G,V;\Z)$ gives rise to a linear map
          \[\delta(e):H^0(V,v_0;\mathbb{Z})\rightarrow \mathbb{Z}.\]
	We may view this as a weight of an irreducible representation
		\[\rho_e:\gT^d \cong  H^0(V,v_0;\R)/ H^0(V,v_0;\Z) \to \gU(1).\]
	\end{enumerate}

The direct sum over all edges gives rise to the {\it graphic representation}
	\[\rho_G = \bigoplus_{e \in E} \rho_e:\gT^d \to \gU(n)\]
of $G$ and the {\it cographic representation}
	\[\rho_G^* = \bigoplus_{e \in E} \rho_e^*:\gT^b \to \gU(n)\]
of $G$. Here $d = |V| - 1$, $n = |E|$, and $b$ is the first Betti number of the graph $G$.

\begin{proposition}\label{pro:GraphicCographicCIG}
For a finite, connected, directed graph $G$ with $d+1$ vertices, $n$ edges, and Betti number $b$, the graphic representation $\rho_G:\gT^d \to \gU(n)$ and cographic representation $\rho_G^*:\gT^b \to \gU(n)$ have the property that all isotropy groups are connected. 

In addition, if \(G\) is a simple graph, i.e. does not have loops and parallel edges, then $\rho_G$ is a subrepresentation of $\rho_{K_{d+1}}$, where $K_{d+1}$ is the complete graph on $d+1$ vertices. Moreover, $\rho_G^*$ is graphic if and only if $G$ is a planar graph, in which case $\rho_G^* \cong \rho_{G^*}$, where $G^*$ is the planar dual of $G$.

Finally, these representations are dual in the following sense: a collection $B \subseteq E$ of $d$ edges in $G$ gives rise to $d$ linearly independent weights of $\gT^d$ if and only if the complement $B^* = E \setminus B$ gives rise to $b$ linearly independent weights of $\gT^b$.
\end{proposition}

Note, in particular, that $d + b = n$ by the duality property. This is consistent with the formulas $\chi(G) = (d+1) - n$ and $\chi(G) = 1 - b$ for the Euler characteristic. Proposition \ref{pro:GraphicCographicCIG} follows by collecting known results for regular matroids and translating them into the language of torus representations with connected isotropy groups (see Section \ref{sec:cogirth}). Alternatively one can use basic algebraic topology to prove these properties.
Therefore the proof is omitted.

\subsection{The sporadic representation}
In the sense defined in the next section, there is exactly one indecomposable representation that is neither graphic nor cographic and that has the property that all isotropy groups are connected. It is the representation
	\[\rho_{R_{10}}:\gT^5 \to \gU(10)\]
defined by the weights
	\[E = \{e_i^* \st 1 \leq i \leq 5\} \cup \{e_{i-1}^* - e_{i}^* + e_{i+1}^* \st 1 \leq i \leq 5\},\]
where the $e_i$ denote the standard basis vectors of $\R^5$ and $e_i^*:\R^5 \to \R$ denote the elements in the dual basis. Here we consider the indices modulo five.

\begin{proposition}
The sporadic representation $\rho_{R_{10}}$ has connected isotropy groups, is neither graphic nor cographic, and is self dual in the following sense: A subset $B \subseteq E$ of five weights is linearly independent if and only if the five weights of $E \setminus B$ are linearly independent.
\end{proposition}

As with Proposition \ref{pro:GraphicCographicCIG}, the properties claimed are known for the sporadic matroid $R_{10}$, so we again omit the proof (see Section \ref{sec:cogirth}).

\subsection{The classification}
\label{sec:classification}

The class of representations with connected isotropy groups is closed under an operation called $k$-sums for $k \in \{1,2,3\}$. We define this construction, and then we state the classification.

For $i \in \{1,2\}$, let $\rho_i:\gT^{d_i} \to \gU(n_i)$ be a representation with connected isotropy groups. Write $\gT^{d_i} = V_i/\Gamma_i$ for some integral lattice $\Gamma_i \subseteq V_i \cong \R^{d_i}$, and denote the subset of weights of $\rho_i$ by $S_i \subseteq \Hom(\Gamma_i,\Z)$. 

The {\it $1$-sum} is simply the product representation 
	\[\rho_1 \oplus_1 \rho_2:\gT^{d_1 \times d_2} \to \gU(n_1+n_2)\]
defined as the composition of the maps 	
	\[\gT^{d_1+d_2} \cong \gT^{d_1} \times \gT^{d_2} \to \gU(n_1) \times \gU(n_2) \subseteq \gU(n_1+n_2),\]
where the middle map is given by $(z_1,z_2) \mapsto (\rho_1(z_1),\rho_2(z_2))$. Alternatively, we may view $\gT^{d_1 + d_2} = (V_1 \oplus V_2)/(\Gamma_1 \oplus \Gamma_2)$ and describe the weights as those lying in the set
	\[\of{S_1 \times 0} \cup \of{0 \times S_2} \subseteq \Hom(\Gamma_1,\Z) \oplus \Hom(\Gamma_2,\Z) \cong \Hom(\Gamma_1 \oplus \Gamma_2, \Z).\]

The {\it $2$-sum} is dependent on a choice of weights $w_i \in S_i$ for $i \in \{1,2\}$. The dependence on this choice is suppressed in the notation. It is a representation
	\[\rho_1 \oplus_2 \rho_2 : \gT^{d_1 + d_2 - 1} \to \gU(n_1 + n_2 - 2)\]
obtained by first passing to the subrepresentation of $\rho_1 \oplus_1 \rho_2$ obtained by removing the weights $(w_1,0)$ and $(0,w_2)$ and then by restricting to the subgroup
	\[\gT^{d_1 + d_2 - 1} = \ker(\rho_{w_1} \rho_{w_2}^{-1}),\]
where $\rho_{w_i}$ is the subrepresentation of $\rho_i$ corresponding to the weight $w_i$. Alternatively, we can define the $2$-sum by identifying $\gT^{d_1 + d_2 - 1} = V/\Gamma$ where
	\[\Gamma = \ker\of{w_1 - w_2} = V \cap (\Gamma_1 \oplus \Gamma_2),\]
	and
	\[V = \ker\of{w_1 \otimes \R - w_2 \otimes \R} \subseteq V_1 \oplus V_2\]
with $w_i\otimes \R$ being the natural extension of $w_i\colon \Gamma_i\to \Z$ to an $\R$-linear map  $V_i\to \R$, and by declaring the set of weights to be 
	\[S = \{(w,0)|_\Gamma \st w \in S_1 \setminus \{w_1\}\} \cup \{(0,w)|_\Gamma \st w \in S_2 \setminus\{w_2\}\}.\]

Finally the {\it $3$-sum} is dependent on a choice of $W_i = \{w_{i,j} \st 1 \leq j \leq 3\} \subseteq S_i$ for $i \in \{1,2\}$ where the $w_{i,j}$ are non-zero and satisfy $w_{i,1} + w_{i,2} + w_{i,3} = 0$ for $i \in \{1,2\}$. It is a representation
	\[\rho_1 \oplus_3 \rho_2:\gT^{d_1 + d_2 - 2} \to \gU(n_1 + n_2 - 6)\]
obtained by setting
	\[V = \bigcap_{j=1}^{3} \ker\of{w_{1,j} \tensor \R - w_{2,j} \tensor \R} \subseteq V_1 \oplus V_2,\]
	\[\Gamma = \bigcap_{j=1}^{3} \ker\of{w_{1,j} - w_{2,j}} = V \cap (\Gamma_1 \oplus \Gamma_2),\]
and
	\[S = \{(w,0)|_\Gamma \st w \in S_1 \setminus W_1\} \cup \{(0,w)|_\Gamma \st w \in S_2 \setminus W_2\}.\]
As with the $2$-sum, this may be viewed as a subrepresentation of 
	\[\gT^{d_1 + d_2 - 2} = \bigcap_{j=1}^{3} \ker\of{\rho_{w_{1,j}} \rho_{w_{2,j}}^{-1}}.\]

\begin{remark}\label{rem:sums}\begin{enumerate}
\item[a)] The above definition of $2$-sum and $3$-sum 
 is consistent  with the corresponding definition  on matroids
in the literature. However, instead of taking the representation 
in $\gU(n_1+n_2-2)$ and $\gU(n_1+n_2-6)$, respectively, one 
could also just work with the induced representation in $\gU(n_1+n_2)$. 
In practice, there is no big difference because the weights of matroids and
representations occur with multiplicities and thus the above two potential definitions only distinguish themselves by the multiplicities of the six  involved weights.
\item[b)] The $1$-, $2$-, and $3$-sum of two graphic representations 
is again graphic. Similarly the $1$- and $2$-sum of two 
cographic representations is cographic. For the $2$-sum, one can 
see this as follows. If we have two graphs $G_1$ and $G_2$ and an (directed) edge $e_i\in G_i$, then we define $\hat G$ as the graph being  obtained from 
the disjoint union $G_1\cup G_2$ by removing the edge $e_i$ from $G_i$ and then connecting the two vertices in $G_1$ to the corresponding vertices in $G_2$ by a new edge. The cographic representation of  $\hat G$ then corresponds to alternative definition of the $2$-sum from a). The $2$-sum $\rho_{G_1}^* \oplus_2 \rho_{G_2}^*$ is the subrepresentation of the cographic representation associated to $\hat G$ obtained by contracting each of the two additional edges.
\item[c)] In an important special case, the $3$-sum of two 
cographic representations is cographic as well: 
Suppose $G_1$ and $G_2$ are graphs and we have two three-valent 
vertices $v_i\in G_i$. Assume that the three weights needed to 
define the $3$-sum are given by the three edges emanating from $v_i$ for $i=1,2$. By assumption we have also an identification of the three 
edges in $G_1$ with the corresponding ones in $G_2$.  
One then defines a graph $\hat G$  obtained from the disjoint union 
$G_1\cup G_2$ by removing the vertex $v_i$ from $G_i$, 
then add a vertex to each end of the three edges  in $G_1$ 
and join the vertices with the corresponding edges in $G_2$. 
The graph $\hat G$ then corresponds to the alternative definition 
of $3$-sum from a). The graph $G$ for the  
$3$-sum can be obtained by contracting each of the six involved edges 
in $\hat G$ to a point. 
Notice that the first Betti number $G$ is bounded above by 
the first Betti number  of $\hat G$ which in turn 
is given by $b_1(G_1)+b_1(G_2)-2$.  If the inequality $b_1(G)\le b_1(\hat G)$ is strict, then the corresponding representation of the $3$-sum has a kernel 
of positive dimension.  In this case the representation corresponding to 
the $3$-sum is just given as an ineffective subrepresentation of the cographic representation of $\hat G$.
\end{enumerate}
\end{remark}

Seymour's theorem on regular matroids and the proof of Tutte's theorem on regular matroids then imply the following (see Section \ref{sec:cogirth} and \cite[Proof of Theorem 6.6.3]{Oxley_matroids}).

\begin{theorem}[Classification of torus representations with connected isotropy groups]\label{thm:cigClassification}

A torus representation $\rho$ has the property that all isotropy groups are connected if and only if $\rho$ can be constructed from iterated $1$-, $2$-, and $3$-sums of graphic, cographic, and sporadic representations.

More generally, if $\rho:T^d\rightarrow U(n)$ has the property that all isotropy groups have an odd number of components, then \(\rho\) can be constructed from a representation \(\rho':T^d\rightarrow U(n)\) with connected isotropy groups as follows:
\begin{enumerate}
\item First pull back the weights of \(\rho'\) along a finite covering \(T^d\rightarrow T^d\) with an odd number of sheets.
\item Then multiply the weights with odd integers.
\item Then push the weights forward along a finite covering \(T^d\rightarrow T^d\) with an odd number of sheets.
\item Finally divide the weights by odd integers. The weights obtained in this way are then the weights of \(\rho\).
\end{enumerate}
\end{theorem}

We illustrate this classification by describing all $\gT^d$-representations with $d \leq 5$ and the property that all isotropy groups are connected. To do so, it suffices to enumerate {\it simple} representations, which means $\gT^d$-representations $W$ with the following properties:
\begin{enumerate}
\item \(W^{\gT^d}=0\)
\item The multiplicity of every weight of \(W\) is one.
\item The isotropy groups of all points in \(W\) are connected.
\end{enumerate}
Moreover, it suffices to enumerate {\it maximal} simple representations, where we say $\rho \leq \hat \rho$ for two simple representations $\rho$ and $\hat \rho$ of $\gT^d$ if the set of weights for $\rho$ is a subset of the set of weights of $\hat \rho$. Given a list of maximal simple representations, all others are obtained by passing to subrepresentations, adding multiplicities, and adding copies of the trivial representation.

There is a natural order-preserving map
	\[\mathrm{si}:R\rightarrow R_0\]
with \(\mathrm{si}|_{R_0}=\Id_{R_0}\), where \(R\) denotes the set of isomorphism types of \(\gT^d\)-representations with connected isotropy groups and \(R_0\) denotes the corresponding set of simple \(T^d\)-representations. 
For a representation \(W\in R\), \(\mathrm{si}(W)\) is the simple representation whose weights are given by the non-trivial weights of \(W\) without repetitions.

Note that a simple graphic representation is always a subrepresentation of \(\rho_{K_{d+1}}\) where \(K_{d+1}\) is the complete graph on \((d+1)\)-vertices.
This is because the graph corresponding to a simple graphic representation is a simple graph, that is, does not have loops or multiple edges. Therefore, there is a unique maximal graphic representation, $\rho_{K_{d+1}}$ of $\gT^d$.

Next, we consider the cographic case, which is more involved. First we need a definition we will use later as well.

\begin{definition}\label{def:splitting}
A {\it splitting} of a simple graph $G$ at a vertex $v$ is a new simple graph $G_+$ obtained from partitioning the set of vertices adjacent to $v$ into two subsets, $A$ and $B$, removing $v$ and all edges incident to $v$, adding new vertices $a$ and $b$, connecting $a$ to every vertex in $A$ and $b$ to every vertex in $B$, and finally adding a new edge $e_+$ between $a$ and $b$.
\end{definition}

\begin{proposition}\label{pro:ReductionToCubic-reps}
If $\rho_G^*$ is both a cographic representation of $\gT^b$ and a maximal simple representation, then one of the following holds:
	\begin{enumerate}
	\item $b = 1$ and $G$ consists of a single vertex and a single loop,
	\item $b = 2$ and $G$ consists of two vertices that are connected by three edges, or
	\item $b \geq 3$ and $G$ is a $3$-valent, $3$-connected graph of girth at least $3$.
	\end{enumerate}
\end{proposition}

Recall that $b$ is the Betti number of the graph $G$. Additionally, the condition $3$-valent is also called cubic and means that all vertices have degree three. Three-connected implies that $G$ remains connected if any two edges are removed. Also the girth is the number of edges in a minimal length cycle of $G$.

\begin{proof}
First, we claim $\si(\rho_G^*) = \si(\rho_{G_1}^*)$ for some $1$-connected graph $G_1$. If $G$ is not connected, we may connect two components with an edge $e$ in such a way that the resulting graph, $G \cup e$, has one fewer component than $G$. Since $e$ is not part of any cycle, its induced representation $\rho_e^*$ is trivial. Therefore $\si(\rho_{G}^*) = \si(\rho_{G \cup e}^*)$. Iterating this process implies the claim.

Second, we claim $\si(\rho_{G}^*) = \si(\rho_{G_2}^*)$ for some $2$-connected graph $G_2$. To see this, we may assume $G$ is $1$-connected but has a bridge, that is, an edge $e$ whose removal results in a disconnected graph $G \setminus e$. Such an edge is not part of any simple cycle, so $\rho_e^* = 0$. In particular, if $G/e$ is the graph obtained by contracting the edge $e$ to a point, we have $\si(\rho_{G}^*) = \si(\rho_{G/e}^*)$ under the natural identification $H_1(G;\Z) \cong H_1(G/e;\Z)$. Since $G/e$ has one fewer bridge, the claim follows by iterating this process.

Third, we claim $\si(\rho_G^*) = \si(\rho_{G_3}^*)$ for some $3$-connected graph $G_3$. We may assume $G$ is $2$-connected but contains a pair of edges, $e_1$ and $e_2$, whose removal disconnects $G$. Any simple cycle containing $e_1$ also contains $e_2$ and vice versa, therefore $\rho_{e_1}^* = \rho_{e_2}^*$. It follows that the quotient map $G \to G/e_1$ obtained by contracting the edge $e_1$ to a point induces an isomorphism $\si(\rho_{G/e_1}^*) = \si(\rho_G^*)$. Once again we have reduced the number of cutsets of a certain size, so arguing iteratively implies the claim.

Fourth, for $b \geq 2$, we claim $\si(\rho_G^*) = \si(\rho_{G_c}^*)$ for some $3$-connected, cubic graph $G_c$. We may assume $G$ is $3$-connected. Since $b \geq 2$, this implies that all vertices have degree at least three. Suppose some vertex $v$ has degree $k \geq 4$. It is an elementary property of graphs that there is a splitting $G_+$ of $G$ at $v$ that maintains $3$-connectivity (see Definition \ref{def:splitting} and Lemma \ref{lem:ThreeConnectedSplitting}). Under the identification $H_1(G;\Z) \cong H_1(G_+/e_+;\Z)$, where $e_+$ is the new edge formed in the splitting, we see that $\si(\rho_G^*) \leq \si(\rho_{G_+}^*)$. By maximality, equality holds, and now iteration yields the claim.

To conclude the proof, we note that a $3$-connected graph with Betti number $b = 1$ or $b = 2$ is as claimed. For $b \geq 3$, the girth is at least three because it is $3$-connected and cubic, so the proof is complete.
\end{proof}

These observations enable us to classify maximal simple representations $\rho$ of $\gT^d$ for \(d\leq 5\).

For $d = 1$, we have no decomposable or sporadic representations for rank reasons. Moreover, if $\rho$ is cographic, then Proposition \ref{pro:ReductionToCubic-reps} implies that we may assume $\rho = \rho_G^*$ where $G$ is the the graph with a single vertex and a single loop. It is then easy to see that $\rho_G^*$ is the graphic representation $\rho_{K_2}$.

For $2 \leq d \leq 4$, we again have no sporadic representations. Moreover, by induction, any decomposable representation is a $k$-sum of graphic representations, and it is known that these are again graphic. As for cographic representations $\rho_G^*$, these are also graphic unless $G$ is non-planar. Since we may assume moreover that $G$ is cubic, we find that either $\rho_G^*$ is graphic or that $G$ is the Kuratowski graph $K_{3,3}$, which has Betti number $b = 4$. 

For $d = 5$, the sporadic representation $\rho_{R_{10}}$ is one maximal simple representation. The graphic representation $\rho_{K_6}$ is another. The cographic case reduces to $\rho_G^*$ for some cubic, $3$-connected graph $G$ with girth at least three. Moreover, we may assume $G$ is non-planar. It follows that $G$ is obtained from $K_{3,3}$ by attaching an edge to a pair of distinct edges in $K_{3,3}$. By the symmetries of $K_{3,3}$, only two graphs can arise in this way. 
One has girth three and the other has girth four, and we denote them by $G_{5,3}$ and $G_{5,4}$. (We give an alternative description below.) Finally, we claim the decomposable case leads to nothing new. Indeed, one of the summands has rank at most three and hence is both graphic and cographic, and the other has rank at most four and hence is graphic or cographic. Using Remark~\ref{rem:sums}
it then easy to check that the resulting sum is graphic or cographic, respectively.

\begin{example}\label{exa:cigUpTo5}
For $1 \leq d \leq 5$, a representation of $\gT^d$ with connected isotropy groups is, up to precomposing with an automorphism of $\gT^d$, a subrepresentation of one of the following:
	\begin{enumerate}
	\item The graphic representation $\rho_{K_{d+1}}$. 
	\item The cographic representation $\rho_{K_{3,3}}^*$ for $d = 4$.
	\item The cographic representation $\rho_{G_{5,3}}^*$ for $d = 5$.
	\item The cographic representation $\rho_{G_{5,4}}^*$ for $d = 5$.
	\item The sporadic representation $\rho_{R_{10}}$.
	\end{enumerate}
Here, $G_{5,3}$ is the graph obtained from the disjoint union of a $K_{2,3}$ and a triangle $C_3$ by choosing a one-to-one correspondence between the degree two vertices of $K_{2,3}$ and the vertices of $C_3$ and attaching corresponding vertices by an edge. In addition, $G_{5,4}$ is the M\"obius ladder on four rungs. 
\end{example}

\begin{remark} If we view the simple representations as real representations, then we can disregard signs of weights and in the above situation the number of weights 
is $\tfrac{1}{2}d(d+1)$ in the  case of $\rho_{K_{d+1}}$, nine for $\rho_{K_{3,3}}^*$, $12$ for $\rho_{G_{5,3}}^*$ and for $\rho_{G_{5,4}}^*$, and $10$ in the case of the sporadic representation.
These upper bounds remain valid for the nonzero weights of the representation $\Z_2^d\subset \gT^d$, if we only require for the $\gT^d$-representation that all  isotropy groups have an odd number of  components. 
\end{remark}

We close this section with a proof of an elementary result for $3$-connected graphs. It was used in Proposition \ref{pro:ReductionToCubic-reps} (see Definition \ref{def:splitting}).

\begin{lemma}
\label{lem:ThreeConnectedSplitting}
Let $G$ be a finite, three-connected graph. If $v$ is a vertex of degree $k \geq 4$, then there exists a $3$-connected splitting $G_+$ of $G$ at $v$. In particular, $G$ can be transformed into a three-connected, cubic graph by a sequence of splittings.
\end{lemma}

\begin{proof}
Label the edges at $v$ by $e_1,\ldots,e_k$ and label the endpoint of $e_i$ not at $v$ by $v_i$ for $1 \leq i \leq k$. Consider the splitting $G_+$ for which the edges attached at one end of the new edge $e_+$ are $\{e_1,e_2\}$ and those attached at the other end are $\{e_3,\ldots,e_k\}$. Since $G$ is $3$-connected, $\{e_1,e_2\}$ is not a cutset of $G$, so there exists a path $p_1$ in $G \setminus \{e_1,e_2\}$ connecting $v_1$ and $v_2$. If $p_1$ passes through $v$, then the splitting $G_+$ has the property that the endpoints of $e_+$ are connected by three disjoint paths. This property implies that $e_+$ is not a bridge and, moreover, that $e_+$ is not part of a cutset of size two. Since any other cutset with two edges in $G_+$ would be a cutset of $G$, there is no such cutset. Hence $G_+$ is $3$-connected in this case. 

Now suppose $p_1$ does not pass through $v$. We repeat this argument with a splitting $G_+'$ based on the partition $\{e_3,e_4\}$ and $\{e_1,e_2,e_5,\ldots,e_k\}$. This gives us a path in $G$ from $v_3$ to $v_4$ that, without loss of generality avoids $v$. By an extension of the same argument, we may assume moreover that $q$ avoids $p$. But now we look at the splitting $G_+''$ based on the partition $\{e_2,e_3\}$ and $\{e_1,e_4,\ldots,e_k\}$ and note that the endpoints of the new edge $e_+''$ are connected by three disjoint paths, namely, one given by $e_+''$, one involving $p$, and one involving $q$. As before, this implies that $e_+''$ is not part of any cutset of size at most two and that $G_+''$ is $3$-connected. 
\end{proof}

\subsection{An optimization problem on fixed-point sets of circles}

For our purposes, we use the classification of torus representations $\gT^d \to \SO(V)$ with connected isotropy groups to prove upper bounds on
\begin{equation}
  \label{eq:opti}
c(\rho) = \min_{\gS^1 \subseteq \gT^d} \cod V^{\gS^1} / \dim V  
\end{equation}
that depend only on the rank of the torus and are strictly less than one. Such bounds do not exist if one removes the assumption on finite isotropy, since generically a sequence of representations $\rho_i:\gT^d \to \SO(V_i)$ with $\dim V_i \to \infty$ will satisfy $\cod V_i^{\gS^1} \geq \dim V_i - 2(d-1)$ for all circles and hence satisfy $c(\rho_i) \to 1$ as $i \to \infty$. 

In Section \ref{sec:systoles}, we prove Theorem \ref{thm:cdIntro} in the case of cographic representations. Let $\rho_G^*:\gT^b \to \gT^n$ be a cographic representation associated to a graph $G$ with Betti number $b$ and with $n$ edges. The irreducible subrepresentations are indexed by the edges $e$ of $G$, and they are allowed now to have multiplicities $\lambda(e)$. For an embedded circle $\gS^1 \subseteq \gT^b$, we get a cycle
	\[v = \sum_{e \in E(G)} a_e e \in \iota_* H_1(G;\Z)\subset H_1(G,V;\Z)\]
for some $a_i \in \Z$. The (complex) codimension $V^{\gS^1} \subseteq V$ is then
	\[\sum_{e^*(v) \neq 0} \lambda(e) = \sum_{a_e \neq 0} \lambda(e) \geq \lambda(C),\]
where the two sums are over edges $e \in E(G)$ satisfying $e^*(v) \neq 0$ and $a_e \neq 0$, respectively, and where $C \subseteq E(G)$ is a (simple) cycle in $G$ consisting of edges $e$ for which $a_e \neq 0$. 

Minimizing over $\gS^1 \subseteq \gT^d$ results in minimizing over cycles $C$ in $G$, and we find that
	\[\min_{\gS^1 \subseteq \gT^d} \frac{\cod V^{\gS^1}}{\dim V} = \min_{C} \frac{\lambda(C)}{\lambda(G)}.\]
Since we seek upper bounds on this quantity, we furthermore consider the maximum over edge weights $\lambda:E(G) \to \R$, restricting without loss of generality to those with total weight $\lambda(G) = 1$. This implies
	\[c(\rho) = \max_\lambda \min_C \lambda(C),\]
where $C$ runs over cycles of the graph $G$ and where $\lambda$ runs over functions $\lambda:E(G) \to \R^{\geq 0}$ with $\lambda(G) = 1$. The right-hand side is called the {\it systole} of the graph $G$ and is denoted by $\sys(G)$. 

In Section \ref{sec:cogirth}, we complete the proof of Theorem \ref{thm:cd} assuming the result in the cographic case. This is not difficult once we have established preliminaries about matroids and stated Seymour's classification theorem. Indeed, the decomposable case gives rise to recursive bounds and follows by induction on the rank of the torus, the graphic case reduces to one case (namely, when $G$ is the complete graph), and the sporadic case similarly is just one representation.

\bigskip
\section{Systole bounds for graphs of small Betti number}
\label{sec:systoles}

For a finite, connected graph $G$, we let 
	\[\sys(G) = \max_\lambda \min_C \lambda(C)\]
be the systole of $G$, where the maximum runs over edge weight functions $\lambda:E(G) \to \R^{\geq 0}$ with $\sum_{e \in E(G)} \lambda(e) = 1$ and where $C$ runs over (simple) cycles $C$ of $G$. The main result is Theorem \ref{thm:systole} from the introduction, rephrased slightly here for convenience:

\begin{theorem}[Theorem \ref{thm:systole}]\label{thm:bounds-cogr-matr}
For any $b \leq 9$, we have $\max_G \sys(G) = s(b)$ where the maximum is taken over finite graphs $G$ with Betti number $b$ and where $s(b)$ is defined as in the following table:
	\[
	\begin{array}{c|c|c|c|c|c|c|c|c|c}
	b 		& 1 & 2 	& 3 & 4 		& 5 		& 6 & 7 		& 8    & 9\\\hline
	s(b)^{-1}	& 1 & 1.5 	& 2 & 2.25		& \displaystyle 2.\bar 6	& 3 & 3.\bar 3 	& 3.5 & 4
	\end{array}
	\]
Here $2.\bar 6$ and $3.\bar 3$ denote $8/3$ and $10/3$, respectively. 
\end{theorem}

The proof requires the rest of this section. 

\subsection{Reductions and recursive bounds}
\label{sec:cographic_estimates}

In this section, we assume the Betti number $b$ of the graph $G$ satisfies $b \geq 2$. The first step is to reduce the problem to the case of trivalent, or cubic, graphs. 

\begin{lemma}[Reduction to three-connected, cubic graphs]
\label{lem:ReductionToCubic-systole}
For any finite graph $G$ with Betti number $b \geq 2$, there exists a finite, 3-connected, cubic graph $G_c$ with Betti number $b$ such that $\sys(G) \leq \sys(G_c)$. Moreover, $G_c$ has girth at least two for $b = 2$ and at least three for $b \geq 3$.
\end{lemma}

\begin{proof}
The proof follows a strategy similar to that for Proposition \ref{pro:ReductionToCubic-reps}. The first step is to show $\sys(G) \leq \sys(G_1)$ for some $1$-connected graph $G_1$. If $G$ is disconnected, we may add an edge $e$ connecting an edge in one component to either an edge in another component or an isolated vertex. For any weight function $\lambda$ on $E(G)$, we can extend it to a weight function $\tilde \lambda$ on the resulting graph $G \cup e$ by declaring $\tilde \lambda(e) = 0$. Since $e$ is not involved in any cycles,
	\[\sys(G,\lambda) = \sys(G \cup e, \tilde \lambda) \leq \sys(G \cup e).\]
Maximizing over $\lambda$ and iterating, we conclude the claim.

Second, we claim $\sys(G) \leq \sys(G_2)$ for some $2$-connected graph $G_2$. We may assume $G$ is connected but has an edge $e$ whose removal results in a disconnected graph $G\setminus e$. Note that $e$ is not part of any cycle. We may contract $e$ to obtain a graph $G/e$, and we define a weight function $\bar \lambda = (1 - \lambda(e))^{-1} \lambda$ using the natural identification of $E(G/e)$ with $E(G) \setminus \{e\}$. Note also we may assume $\lambda(e) \neq 1$, since otherwise there is a cycle of zero length in $G$, in which case the upper bound we prove below holds trivially. Hence
	\[\sys(G,\lambda) \leq (1 - \lambda(e)) \sys(G/e, \bar \lambda) \leq \sys(G/e).\]
Now maximizing over $\lambda$ and iterating as needed, the claim follows.

Third, we claim $\sys(G) \leq \sys(G_3)$ for some $3$-connected graph $G_3$. We may assume $G$ is $2$-connected, and we may assume there exists a pair of edges $e_1$ and $e_2$ such that $G \setminus (e_1 \cup e_2)$ is disconnected. Note that any cycle passing through $e_1$ also passes through $e_2$. Contracting $e_1$ and replacing the value of $\lambda(e_2)$ by $\lambda(e_1) + \lambda(e_2)$ for any weight function $\lambda$ on $E(G)$, we find that $\sys(G) \leq \sys(G/e_1)$. This process reduces the number of pairs of edges whose removal disconnects the graph, so the claim hold by iteration.

Fourth, we claim $\sys(G) \leq \sys(G_c)$ for some cubic graph $G_c$. We may assume $G$ is $3$-connected, and we may assume $v$ is a vertex in $G$ with degree $k \geq 4$. Moreover, we may choose a $3$-connected splitting $G_+$ of $G$ at $v$ (see Lemma \ref{lem:ThreeConnectedSplitting}). Let $e_+$ denote the new edge formed in the splitting. For any $\lambda:E(G) \to \R^{\geq 0}$, we can extend the definition to $\lambda_+:E(G_+) \to \R^{\geq 0}$ by setting $\lambda_+(e_+) = 0$. This shows
	\[\sys(G,\lambda) = \sys(G_+, \lambda_+) \leq \sys(G_+).\]
Taking the maximum over $\lambda$ yields $\sys(G) \leq \sys(G_+)$. Note that $G_+$ has the same Betti number as $G$ and that it either has smaller maximum degree or a smaller number of vertices with maximum degree. By iterating, the claim holds.

Finally, a $3$-connected, cubic graph is easily seen to have girth at least two for $b = 2$ and girth at least three for $b \geq 3$.
\end{proof}

By Lemma \ref{lem:ReductionToCubic-systole}, it suffices for $b \geq 2$ to prove our systole bounds for $3$-connected, cubic graphs $G$. Our proof strategy will be by induction over the Betti number, and we implicitly assume throughout this section that the theorem holds for all smaller Betti numbers. The following two lemmas give strong recursive estimates on the systole as a function of the Betti number. The first considers the case of small girth.

\begin{lemma}[Small cycle estimate]
Let $G$ be a three-connected graph with Betti number $b \geq 2$. If $G$ contains a cycle $C$ of length $g \geq 1$, then
	\[\sys(G)^{-1} \geq \frac{h}{g} + s(b-h)^{-1}\]
for all $1 \leq h \leq \min(g,b-1)$.
\end{lemma}

\begin{proof}
Fix a $g$-cycle $C$, and fix a weight function $\lambda$. Let $e$ be an edge in $C$ with maximal weight $\lambda(e)$. By considering the cycle $C$ on one hand and all cycles not containing the edge $e$ on the other, we get two estimates on the systole:
	\[\sys(G,\lambda) \leq \lambda(C) \leq g \lambda(e)\]
and
	\[\sys(G,\lambda) \leq (1 - \lambda(e)) \sys(G\setminus e,\lambda') \leq  \of{1-\lambda(e)} s(b-1),\]
where $\lambda' = (1 - \lambda(e))^{-1} \lambda|_{E(G) \setminus \{e\}}$, and where the last inequality follows by induction because $G \setminus e$ has the same number of components as $G$ and hence has Betti number $b-1$. Combining these inequalities to eliminate $\lambda(e)$ implies
	\[s(b-1) \sys(G, \lambda) + g \sys(G, \lambda) \leq g s(b-1).\]
Dividing by $g$, $s(b-1)$, and $\sys(G, \lambda)$ and then maximizing over $\lambda$ implies the bound for $h = 1$.
        
Generalizing this idea, for any $1 \leq h \leq g$, we may remove $h$ of the edges in $C$ with the largest weights and consider only cycles that do not share an edge with these edges. The only additional observation needed for this case is that $G$ remains connected after removing the $h$ edges. Indeed, since $G$ is three-connected, any two vertices are connected by three disjoint paths in $G$, so even removing the entire cycle $C$ does not disconnect $G$. 
\end{proof}

This lemma is strongest when $G$ has a cycle of small length (i.e., $G$ has small girth). The following lemma provides bounds in the complementary case.

\begin{lemma}[Large girth estimates]\label{lem:largegirth}
Suppose $G$ is a cubic graph with Betti number $b$ such that every cycle has length at least $g$. The following hold:
	\begin{enumerate}
	\item If $g \geq 2$ and $b \geq 3$, then $\sys(G)^{-1} \geq \frac{b-1}{b-2} s(b-2)^{-1}$.
	\item If $g \geq 3$ and $b \geq 4$, then $\sys(G)^{-1} \geq \frac{3b-3}{3b-8} s(b-3)^{-1}$.
	\item If $g \geq 4$ and $b \geq 6$, then $\sys(G)^{-1} \geq \frac{b-1}{b-4} s(b-5)^{-1}$.
	\end{enumerate}
\end{lemma}

As the proof shows, this lemma has straightforward generalizations, but we do not require them for the proof so we omit them for simplicity.

\begin{proof}
Fix $G$ as in the lemma, and fix a weight function $\lambda:E(G) \to \R^{\geq 0}$. To prove (1), fix a vertex $v$ and let $B_1(v)$ denote the vertex together with the edges incident at $v$. Note that $g \geq 2$ implies that $B_1(v)$ contains three distinct edges. Let $\lambda(B_1(v))$ denote the sum of the weights on the three edges in $B_1(v)$. By considering cycles in $G$ that do not contain $v$, we find that
	\[\sys(G,\lambda) \leq (1 - \lambda(B_1(v)))\sys(G \setminus B_1(v)) \leq (1 - \lambda(B_1(v))) s(b-2),\]
where for the second inequality we used that $G \setminus B_1(v)$ has Betti number at least $b-2$. Summing over all vertices $v$, we have
	\[(2b-2) \sys(G, \lambda) \leq \of{2b-4} s(b-2),\]
where we have used the fact that each edge appears in exactly two $B_1(v)$. Maximizing over $\lambda$ implies the first claim of the lemma.

To prove (2), we fix an edge $e$ and consider $G \setminus B_1(e)$, the graph that results by removing $e$ and all edges that touch $e$. Note that this results in removing five edges in all since $g \geq 3$. Note that $G \setminus B_1(e)$ has rank at least $b-3$, so estimating as above gives
	\[\sys(G,\lambda) \leq  (1 - \lambda(B_1(e))) s(b-3).\]
Summing over edges and using the fact that each edge appears in exactly five of the $B_1(e)$, we have
	\[(3b-3) \sys(G, \lambda) \leq (3b - 8) s(b-3),\]
and the claim follows.

The proof of (3) is similar to (1), except that we remove subsets $B_2(v)$ consisting of $v$, the three edges and vertices adjacent to $v$, and the six edges adjacent to one of these three vertices. Note that $G \setminus B_2(v)$ has rank $b-5$ and that each edge in $G$ appears in exactly six of the $B_2(v)$. The claim follows.
\end{proof}

\subsection{Optimal systole estimates for $b \not\in \{7,9\}$}
\label{sec:cographic_not7or9}

In this section, we prove Theorem \ref{thm:bounds-cogr-matr} for $b \leq 6$ and $b = 8$. In addition, we prove $s(10) \leq 1/4$, which is sufficient to prove Theorem \ref{thm:t9} if the $\gT^9$-symmetry assumption is strengthened to a $\gT^{10}$-symmetry assumption. To finish the proof, we need the values $s(7)$ and $s(9)$. These are difficult computations that are postponed until the next two sections.

Fix a finite graph $G$ with Betti number $b$. By Lemma \ref{lem:ReductionToCubic-systole}, we may assume without loss of generality that $G$ is $3$-connected and moreover cubic if $b \geq 2$ and girth at least three if $b \geq 3$. We use that connected, cubic graphs with Betti number $b$ have $2(b-1)$ vertices and $3(b-1)$ edges. This follows from the Handshaking Lemma and the Euler characteristic formula $|V| - |E| = 1 - b$.

For $b = 1$, there is only one such graph, a cycle with one vertex. Clearly $\sys(G) = 1$ in this case, so \[s(1)^{-1} = 1.\]

For $b = 2$, $G$ has two vertices, three edges, and girth at least two. Therefore $G$ is uniquely determined and is called the theta graph $\Theta$. Either computing explicitly or applying the small cycle estimate with $h = 1$, we find $s(G)^{-1} \geq 1.5$. Moreover, equality holds for this graph by taking all weights equal to $1/3$. Hence 
	\[s(2)^{-1} = \sys(\Theta)^{-1} = 1.5.\]

For $b = 3$, $G$ has four vertices, six edges, and girth at least three. Therefore $G = K_4$, the complete graph on four vertices. Given an edge weight function $\lambda$, we sum over all four $3$-cycles to obtain $\sys(G,\lambda) \leq 1/2$. Moreover, equality holds if all edges have weight $1/4$, so 
	\[s(3)^{-1} = \sys(K_4)^{-1} = 2.\]

For $b = 4$, Lemma \ref{lem:largegirth}.(1) implies that
	$\sys(G)^{-1} \geq \tfrac{4-1}{4-2} s(2)^{-1} = \frac 9 4.$
Moreover, equality holds for $G = K_{3,3}$ by putting equal weights on each of the nine edges. The smallest cycles have four edges and hence have weight $4/9$. Hence 
	\[s(4)^{-1} = \sys(K_{3,3})^{-1} = 2.25.\]

For $b = 5$, Lemma \ref{lem:largegirth}.(1) implies that
	$\sys(G)^{-1} \geq \tfrac{5-1}{5-2}s(3)^{-1} = \frac 8 3.$ 
Moreover, equality is attained if we consider a M\"obius ladder $G_{5,4}$. Putting weights \(1/16\) on each edge on the side of the ladder and weights \(2/16\) on the rungs of the ladder, we find that both the four- and five-cycles have length 6/16 = 3/8. It follows easily that 
	\[s(5)^{-1} = \sys(G_{5,4})^{-1} = 2.\bar 6.\]

For $b = 6$, Lemma \ref{lem:largegirth}.(2) implies that
	$\sys(G)^{-1} \geq \tfrac{3(6)-3}{3(6)-8}s(3)^{-1} = 3.$ 
For equality, we consider the Petersen graph $P$, which is the unique cubic graph on ten vertices with girth five. Putting equal weights on all $15$ of its edges shows $\sys(P) \geq \frac{5}{15}$. This proves that 
	\[s(6)^{-1} = \sys(P)^{-1} = 3.\]

For $b = 8$, Lemma \ref{lem:largegirth}.(2) implies that
	$\sys(G)^{-1} \geq \tfrac{3(8)-3}{3(8)-8} s(5)^{-1} = \frac 7 2.$ 
To prove equality, we put equal weights on the $21$ edges of the Heawood graph \(H\), which is the unique cubic graph on $14$ vertices with girth six. It follows that
	\[s(8)^{-1} = \sys(H)^{-1} =  3.5.\]

For $b = 10$, we use the small cycle estimate with $h = 2$ if $g = 3$ to conclude $s(10)^{-1} \geq \tfrac 2 3 + s(8)^{-1} > 4$, and we use Lemma \ref{lem:largegirth}.(3) if $g \geq 4$ to conclude $s(10)^{-1} \geq \tfrac{(10) - 1}{(10)-4} s(5)^{-1} = 4$. Together these imply
	\[s(10)^{-1} \geq 4.\]
As we will see, $s(9)^{-1} = 4$. In addition $s(b)^{-1}$ is strictly increasing in $b$ by the small cycle lemma, so the estimate on $s(10)$ is not sharp. On the other hand, Theorem~\ref{thm:t9} can now be proved if one is willing to replace the $\gT^9$ by a $\gT^{10}$ in the assumption.

The value $s(8) = \tfrac 2 7 > \tfrac 1 4$ also shows that we cannot replace the $\gT^9$ by $\gT^8$ in Theorem~\ref{thm:t9}. Indeed, our proof requires that a $\gT^b$ action as in Theorem~\ref{thm:t9} has $b$ large enough so that $s(b) \leq \frac 1 4$.

\subsection{Optimal systole estimates for $b = 7$}
\label{sec:cographic_7}

In this section we prove Theorem \ref{thm:bounds-cogr-matr} in the case \(b=7\).
The proof requires an additional estimate.

\begin{proposition}
  \label{pro:embeds}
  Let \(G\) be a connected graph with Betti number $b$. If $G$ embeds into a closed connected surface \(\Sigma\) with Euler characteristic \(\chi(\Sigma)\), then
  \[\sys(G)^{-1} \geq \frac{b - 1 + \chi(\Sigma)}{2}.\]
\end{proposition}

\begin{remark}
Note that a graph $G$ is planar if and only if it embeds into the $2$-sphere, which has Euler characteristic two, so this bound implies $\sys(G)^{-1} \geq \tfrac{b+1}{2}$ for planar graphs $G$. This is consistent with the fact that the cographic matroid $M^*[G]$ of a planar graph $G$ is isomorphic to the graphic matroid $M[G^*]$ of the dual graph $G^*$ to $G$ obtained by interchanging the roles of vertices and faces in an embedding of $G$ in $\R^2$. Combined with the estimate on the cogirth of graphic matroids given in Section \ref{sec:cogirth}, we similarly obtain the bound $\tfrac{b+1}{2}$ in the planar case.
\end{remark}

\begin{proof}
  First note that we can assume that the complement of \(G\) in \(\Sigma\) is a disjoint union of open discs $\bigcup_{i=1}^k D_i$, since otherwise we can do surgery on \(\Sigma\) to embed \(G\) in a surface with higher Euler-characteristic. Moreover, note that the boundaries $C_i = \partial D_i$ of the components of the complement are closed paths in $G$. Since each edge $e$ in $G$ is contained in at most two of these boundaries, we have
  	\[\sum_{i=1}^k \lambda(C_i) \leq 2 \sum \lambda(e) = 2,\]
where the sum is over the edges $e$ in $G$. Hence $\lambda(C_i) \leq 2/k$ for some $i$.

Now $C_i$ need not be a cycle. However by removing double points in $C_i$, if necessary, we get a cycle $C$ with $\lambda(C) \leq 2/k$. Hence $\sys(G) \leq 2/k$, and the claim follows by Euler's formula since
	\[\chi(\Sigma) = |V| - |E| + k = 1 - b + k,\]
where $|V|$, $|E|$, and $b$ are the number of vertices, the number of edges, and the Betti number of $G$, respectively.
\end{proof}

As it turns out, most graphs with Betti number $b = 7$ embed into the real projective plane $\RP^2$. Therefore Proposition \ref{pro:embeds} covers most cases of the following lemma. Note that we do not assume $3$-connectedness, since our proof of $s(9)^{-1} = 4$ in the following section requires the more detailed statement proved here.

\begin{lemma}[Calculation of $s(7)$]\label{sec:d=7-not-sufficient}
  Let \(G\) be a cubic graph with Betti number \(7\). Then one of the following two cases holds:
  \begin{enumerate}
  \item \(\sys(G)^{-1} \geq 3.5\).
  \item \(G\) is $F_{13}$ or $F_{14}$, respectively, $\sys(G)^{-1}$ equals $3.375$ or $3.\bar 3$.
  \end{enumerate}
\end{lemma}

The graphs in the second statement coincide with the set of cubic graphs with Betti number $7$ that are connected and have girth five. There are no such graphs of girth larger than five.

\begin{proof}
First, if $G$ is not connected, we argue as in the proof of the Small Cycle Estimate. By restricting to cycles in one component $H$ or its complement $G \setminus H$, we get the estimates $\sys(G,\lambda) \leq \lambda(H) \sys(H)$ and $\sys(G,\lambda) \leq \lambda(G\setminus H) \sys(G\setminus H)$. Since $\lambda(H) + \lambda(G\setminus H) = 1$, these estimates imply
	\[\sys(H)^{-1} + \sys(G\setminus H)^{-1} \leq \sys(G, \lambda)^{-1}.\]
Letting $b_1 = b(H)$ and $b_2 = b(G\setminus H)$, we have by induction that
	\[s(b_1)^{-1} + s(b_2)^{-1} \leq \sys(G, \lambda)^{-1}.\]
Finally since $b_1 + b_2 = b$, we have
	\[\sys(G,\lambda) \geq \min(s(1)^{-1} + s(6)^{-1}, s(2)^{-1} + s(5)^{-1}, s(3)^{-1} + s(4)^{-1}) = 4.\] 

Second, if $G$ is connected and embeds into real projective space $\RP^2$, then Proposition \ref{pro:embeds} implies that $\sys(G)^{-1} \geq 3.5$.

We may assume that $G$ is a connected, cubic graph that does not embed into $\RP^2$. Note that property of embeddability in $\RP^2$ is preserved under removing edges in the graph. The full set of subgraph-minimal cubic graphs that are not embeddable in $\RP^2$ has been classified and consists of six graphs (see \cite{GloverHuneke75,Milgram73}). These graphs are the six cubic graphs appearing in the list of $103$ graphs (see \cite{Glover_et_al,archdeacon} or \cite[Appendix A]{MoharThomassen-book}) that are similarly not embeddable in $\RP^2$ and minimally so with respect to passing to subgraphs and graph minors. In the notation of \cite{Glover_et_al}, the six cubic graphs are
	\[E_{42}, F_{11}, F_{12}, F_{13}, F_{14}, \mathrm{and~} G_1.\]
Moreover $E_{42}$ has rank 8, so it cannot be a subgraph of $G$. The other graphs have rank $7$ for the $F_i$ and $6$ for $G_1$. Therefore we have to check the cases $G = F_i$ and the graphs $G$ which can be constructed from \(G_1\) by attaching an edge. We denote the last case by $G = G_1 \cup e$, since the inverse operation of removing an edge is attaching a new edge by connecting its endpoints to the interiors of edges in $G_1$.

\smallskip\noindent\textbf{Case 1: $G = G_1 \cup e$.}

\begin{figure}[ht]
  \includegraphics[height=1.5in,angle=90]{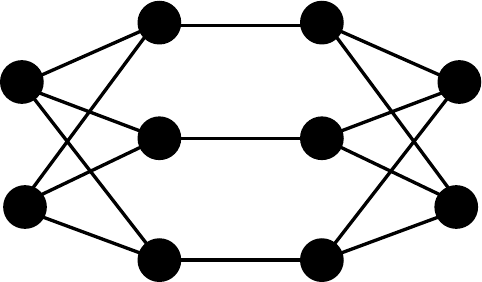}
  \caption{The graph \(G_1\).}\label{fig:g1}
\end{figure}	

Note that $G_1$ consists of two $K_{2,3}$ subgraphs connected by three edges, which we will label $f_1$, $f_2$, and $f_3$ (see Figure \ref{fig:g1}). After relabeling, we may assume that $e$ is not attached to $f_1$ or to an edge adjacent to $f_1$. In all but one case, there is a planar embedding of $G \setminus f_1$, the graph of $G$ with $f_1$ removed. In the remaining case, $e$ is connected to two opposite edges of a four-cycle in one of the copies of $K_{2,3}$. For this graph, the graph $G \setminus f_2$ is planar. Hence in any case, there exists an edge $f_i$ such that
	\[\sys(G)^{-1} \geq \sys(G \setminus f_i)^{-1} \geq \frac{6 - 1 + \chi(\s^2)}{2} = 3.5.\]
This completes the proof in Case 1.

\smallskip\noindent\textbf{Case 2: $G = F_{11}$.}

The graph $F_{11}$, shown on the left of Figure~\ref{fig:c74-5-2'}, contains edges $e_1$ and $e_2$ such that $F_{11} \setminus (e_1 \cup e_2)$ is disconnected. Removing $e_1$ gives rise to a planar graph $F_{11} \setminus e_1$, so Proposition \ref{pro:embeds} implies that
	\[\sys(F_{11})^{-1} \geq \sys(F_{11} \setminus e_1)^{-1} \geq 3.5.\]

\begin{figure}[ht]
  \raisebox{-0.5\height}{\includegraphics[height=1.5in]{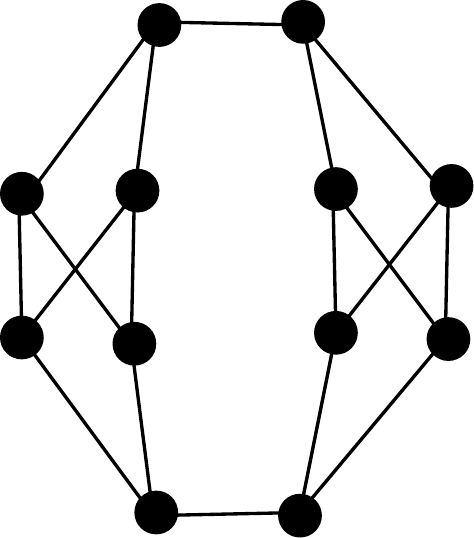}}
  \hspace{0.8cm}
  \raisebox{-0.5\height}{\includegraphics[height=1.8in]{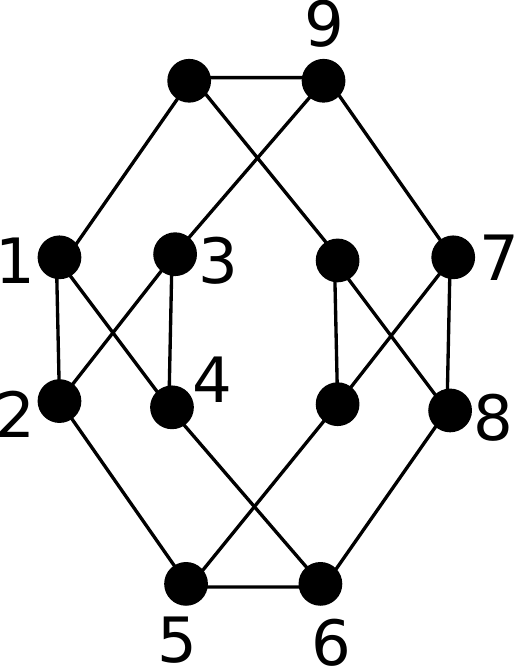}}
  \caption{The graphs \(F_{11}\) and \(F_{12}\).}\label{fig:c74-5-2'}
\end{figure}

\smallskip\noindent\textbf{Case 3: $G = F_{12}$.}

The graph $F_{12}$ is shown on the right of Figure \ref{fig:c74-5-2'}.
 It has $D_4\times \Z_2$ symmetry generated by the horizontal reflection, the vertical reflection, and by a symmetry determined by the properties that the top two vertices are fixed and the bottom two are swapped.

Assume for a moment that the weights are constant along the orbits of the action by $D_4\times \Z_2$. Label the edge weights by $\lambda$, $\mu$, and $\nu$. There are cycles $C_1$ (with vertices \(1,2,3,4\)), $C_2$ (with vertices \(1,2,5,6,4\)), and $C_3$ (with vertices \(3,4,6,8,7,9\)) with total weights $4 \mu$, $2 \lambda + 2 \mu + \nu$, and $4\lambda + 2\mu$, respectively. Since $\sys(F_{12}) \leq \lambda(C_i)$ for all $i$, we obtain the estimate
	\[7 \sys(F_{12}) \leq \lambda(C_1) + 4\lambda(C_2) + 2\lambda(C_3) =  2(8\lambda + 8\mu + 2\nu) = 2.\]
Hence $\sys(F_{12})^{-1} \geq 3.5$. In fact, if we take $\lambda = 1/28$, $\mu = 2/28$, and $\nu = 3/28$, we find that $\lambda(C) \geq 2/7$ for all cycles $C$. Hence $\sys(F_{12})^{-1} = 3.5$.

Finally, we justify the assumption that the edge weights are constant along the orbits. Indeed, we may sum over not only weights of $C_1$, $C_2$, and $C_3$, but over the collection of cycles obtained under the action of $D_4\times \Z_2$. Summing the resulting weights gives rise to estimates as above where $\lambda$, $\mu$, and $\nu$ are replaced by the average weight of the edges in the respective orbits.

\smallskip\noindent\textbf{Case 4: $G = F_{13}$.}

We start by analyzing the symmetry of $F_{13}$. By inspection, we find that $F_{13}$ contains a $9$-cycle whose vertices are precisely those that are part of four $5$-cycles. We call the remaining three vertices {\it tripod vertices} and we call their open unit balls {\it tripods}. Drawing the $9$-cycle as a circle and connecting the tripods, we find that $F_{13}$ may be drawn as in Figure \ref{fig:F13}. 

\begin{figure}[ht]
\includegraphics[width=2in]{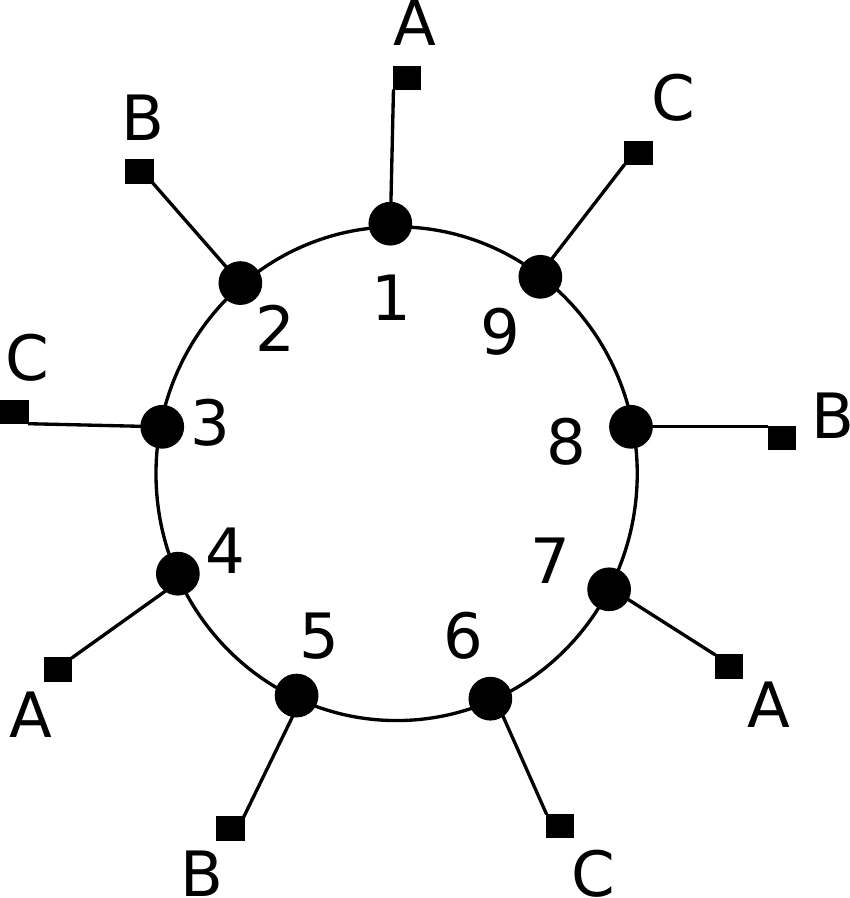}\caption{A drawing of \(F_{13}\) showing the $D_9$ symmetry. In this drawing the squared vertices with the same labels have to be identified.}\label{fig:F13}
\end{figure}

As in the previous case, we can assume that the weights of edges in the same \(D_9\)-orbit are equal. Hence there are only two different weights \(\lambda\) and \(\mu\) with \(9\lambda + 9 \mu=1\).
Moreover there are cycles \(C_1\) and $C_2$ with vertex set $\{1,2,3,4,A\}$ and $\{A,1,2,B,5,4\}$, respectively, have total weight \(2\lambda+3\mu\) and \(4\lambda+2\mu\), respectively. It follows that at least one of the \(C_i\) has weight less than or equal to \(\frac{8}{27}\).
Hence, \(\sys(F_{13})^{-1}\geq \frac{27}{8}=3.375\).
Note that equality holds for \(\mu=\frac{2}{27}\) and \(\lambda=\frac{1}{27}\).

\smallskip\noindent\textbf{Case 5: $G = F_{14}$.}

\begin{figure}[ht]
\includegraphics[width=2in]{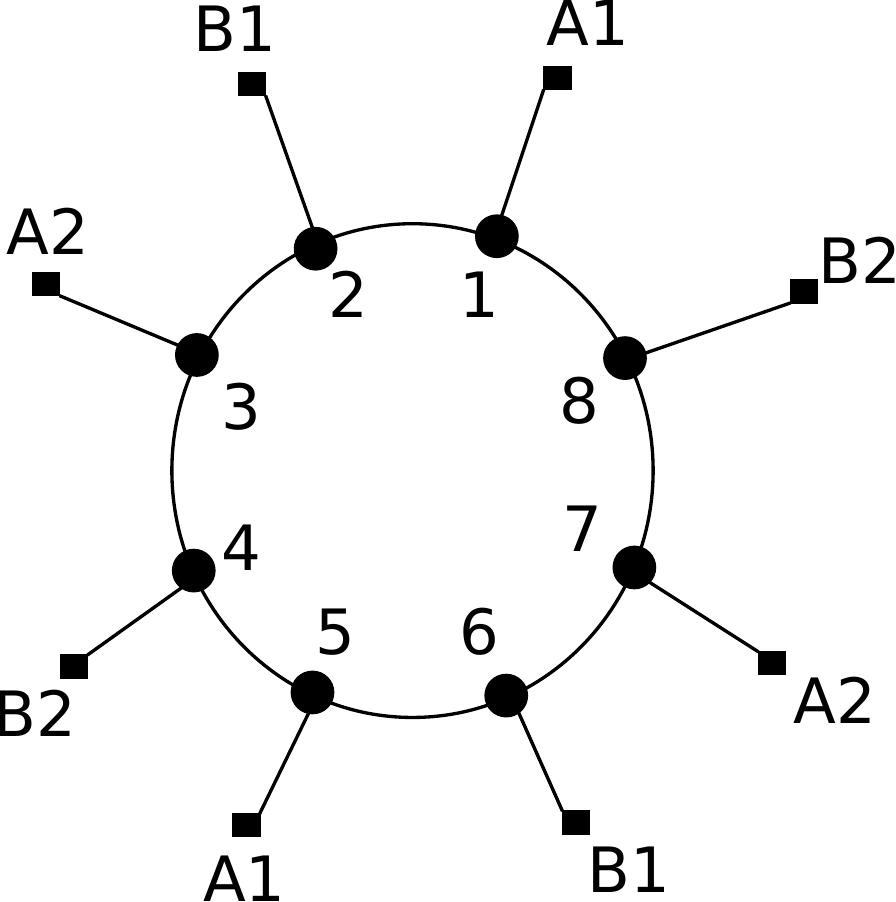}\caption{A drawing of \(F_{14}\) showing the $D_8$ symmetry. In this drawing the squared vertices with the same labels have to be identified. Moreover, between the vertices A1 and A2 there is an edge $H$ which is not drawn and similarly there is an edge $H'$ between B1 and B2.}\label{fig:F14}
\end{figure}

The graph $F_{14}$ contains an $8$-cycle $C$ consisting of precisely those vertices that are contained in exactly three $5$-cycles. The other four vertices are parts of four $5$-cycles. The latter points come in pairs of vertices that are connected by edges, which we call $h$ and $h'$. Let $H$ denote the open ball of radius one about the edge $h$, and likewise for $H'$. Note that  \(H\) consists of those edges which are connected to the vertices A1 and A2 in Figure \ref{fig:F14}, similarly \(H'\) consists of those edges which are connected to B1 and B2. The other edges form a cycle \(C\) whose vertices are labeled with numbers.

The graph $F_{14}$ has dihedral symmetry group $D_8$, corresponding to the symmetries of the $8$-cycle $C$. Note in addition that the subgroup $D_4$ acts on $F_{14}$ while preserving $h$ and $h'$. (To be clear, the $D_4$ subgroup has an index two subgroup that also fixes the endpoints of $h$ and $h'$, while the other elements in $D_4$ swap the endpoints.)

As in the previous cases we can assume that the weights are equal along an orbit of the \(D_8\)-action on \(F_{14}\).
Hence we may assume there are three different weights \(\lambda, \mu,\nu\) with \(2\lambda+8\mu+8\nu=1\).
Moreover, there are cycles $C_1$, $C_2$, and $C_3$ on the vertex sets $\{1,2,3,4,5,A1\}$, $\{1,2,B1,6,5,A1\}$, and $\{1,2,3,A2,A1\}$, respectively, have total weights \(2\mu + 4\nu\), \(4\mu + 2\nu\), and \(\lambda+2\mu+2\nu\), respectively. Summing these total weights shows that at least one of the \(C_i\) has total weight less than or equal to \(\frac{3}{10}\).
Hence \(\sys(F_{14})^{-1}\geq \frac{10}{3}=3.\overline{3}\). Moreover equality holds for \(\lambda=\frac{1}{10}\) and \(\mu=\nu=\frac{1}{20}\).
 \end{proof}

\subsection{Optimal systole estimates for $b = 9$}
\label{sec:cographic_9}

In this section we complete the proof of Theorem \ref{thm:bounds-cogr-matr}. By Lemma \ref{lem:ReductionToCubic-systole}, it suffices to show that a $3$-connected, cubic graph $G$ with Betti number $b = 9$ satisfies $\sys(G)^{-1} \geq 4$.

First, $3$-connectedness implies $G$ has girth at least three. Moreover, if equality holds, then we may apply the small cycle estimate with $h = g = 3$ to obtain
	\[\sys(G)^{-1} \geq 1 + s(6) = 4.\]
Therefore we may assume that $G$ has girth at least four.

Second, suppose for a moment that $G \setminus v$ is not isomorphic to \(F_{13}\) or \(F_{14}\) for any vertex $v$. By $G \setminus v$, we mean the cubic graph obtained by removing $v$ along with the three edges incident to $v$ and then suppressing the three degree two vertices. Lemma \ref{sec:d=7-not-sufficient} implies that $\sys(G\setminus v)^{-1} \geq 3.5$ for all vertices $v$. Combined with the proof of the large girth estimate (Lemma \ref{lem:largegirth}), we have
	\[\sys(G)^{-1}\geq  \frac{8}{7} \min_{v} \sys(G \setminus v)^{-1} \geq 4.\]
Therefore we may assume that removing some vertex from $G$ gives rise to a graph denoted by $G \setminus v$ which is isomorphic to $F_{13}$ or $F_{14}$.

Inverting this operation, and keeping in mind that $G$ has girth at least four, we see that $G$ is built from some $F \in \{F_{13}, F_{14}\}$, together with a choice of three pairwise distinct edges $e_i$ in $F$. Indeed, the $e_i$ are split into two edges, their midpoints $v_i$ are introduced as new vertices, and the $v_i$ are connected to a new vertex $v_0$. 

As a first step, we prove $\sys(G)^{-1} \geq 4$ in cases where $G$ has girth four. This happens precisely when two of the highlighted edges meet at a vertex. The proof uses the small cycle lemma once more.

\begin{lemma}\label{lem:e1e2touch}
If $e_1$ and $e_2$ share a vertex $v$, then $\sys(G)^{-1} \geq 4$.
\end{lemma}

\begin{proof}
Let $v$ denote the shared vertex of $e_1$ and $e_2$, and let $v_1$ and $v_2$ denote the other two vertices. For $i \in \{1,2\}$, let $e_i'$ denote the portion of $e_i$ in $F$ containing $v$ that becomes an edge in $G$ after attaching $v_0$.

{\bf Case 1:} There exists a path in $G \setminus  (e_1' \cup e_2')$ from $v$ to $e_3$ of length at most two. For $i \in \{1,2\}$, let $C_i$ denote the cycle made up of this path, $e_i'$, $f_i$, $f_3$, and one additional edge that connects the path to the midpoint of $e_3$. Note that the $C_i$ has length at most six. Note moreover that $C_1$ becomes a cycle of length at most four upon removing $e_2'$ and $f_2$ from $G$, and likewise for $C_2$ upon removing $e_1'$ and $f_1$. In particular, the $G \setminus (e_i' \cup f_i)$ have girth at most four, are not $F_{13}$ or $F_{14}$, and hence satisfy $\sys(G \setminus (e_i' \cup f_i))^{-1} \geq 3.5$. Applying the proof of the small cycle estimate to the $4$-cycle consisting of $e_1'$, $f_1$, $e_2'$, and $f_2$, we have
	\[\sys(G)^{-1} \geq \frac{1}{2}\of{1 + \sys(G \setminus (e_1' \cup f_1))^{-1} + \sys(G \setminus (e_2' \cup f_2))^{-1}} \geq 4.\]

{\bf Case 2:} For $i \in \{1,2\}$, there exists a path $p_i$ of length at most two connecting $v_i$ to $e_3$. The paths may intersect, but this does not affect the argument that follows. We again apply the small cycle estimate to the $4$-cycle $C$ from Case 1. This time, note that there exists a cycle containing $f_3$, $f_1$, and $p_1$ that has length at most six and becomes a cycle of length at most four upon removing $e_1'$ and $f_2$. There is similarly a cycle containing $f_3$, $f_2$, and $p_2$ that becomes a cycle of length at most four upon removing $e_2'$ and $f_1$. Therefore $\sys(G)^{-1} \geq 4$ by the same argument as in Case 1.

If neither Case 1 nor Case 2 occurs, then $e_3$ has distance at least three from one of the edges, say, $e_1$. In $F_{13}$, there is no such pair of edges. In $F_{14}$, there is a unique pair of edges of distance three, so $G$ is realized by first attaching an edge $f$ to $e_1$ and $e_3$ to obtain a rank 8 graph $H$ and then attaching an edge to $f$ and $e_2$. Note that $H$ has girth six and hence is the Heawood graph. Therefore the proof in this case follows from the next lemma.
\end{proof}

\begin{lemma}\label{lem:Heawood}
If $G$ arises by attaching an edge $e$ to the Heawood graph $H$, then $\sys(G)^{-1} \geq 4$.
\end{lemma}

\begin{proof}
Let $e_1$ and $e_2$ be the two edges in $H$ to which we attach $e$. By inspection of $H$, the distance between $e_1$ and $e_2$ is at most two. Therefore $e_1$ and $e_2$ lie in an arc $p$ of length at most four.

It is a remarkable property of $H$ that every path of length four can be moved by an automorphism to any other such path \cite{conder_morton95}. In particular, up to automorphism, we may assume that $p$ lies on any given cycle. 

Another good property of $H$ is that it can be embedded in the torus $T^2$. Fix a cycle $C$ in $H$ that bounds a hexagon in $T^2$ \cite{heawood1890}. Using the 4-arc transitivity of $H$ discussed in the previous paragraph, we may assume that $G$ arises by attaching $e$ to the boundary of this hexagon. Therefore $G$ can also be embedded in $T^2$, and
	\[\sys(G)^{-1} \geq \frac{9 - 1 + \chi(T^2)}{2} = 4.\]
\end{proof}

The proof of Lemma \ref{lem:Heawood} previews how the rest of the proof of $\sys(G)^{-1} \geq 4$ will go. It involves multiple cases, so we outline the strategy now.

Fix $F \in \{F_{13}, F_{14}\}$, and assume $G$ is built from $F$ by adding a vertex and attaching that vertex to each of three pairwise distinct edges $e_1$, $e_2$, and $e_3$ in $F$. By Lemmas \ref{lem:e1e2touch} and \ref{lem:Heawood}, we may assume that these edges are vertex-disjoint. In particular, since $F_{13}$ and $F_{14}$ have girth five, and since we are not introducing new cycles of length at most four in building $G$, we have already finished the proof of $\sys(G)^{-1} \geq 4$ in the case where $G$ has girth at most four.

To complete the proof, we list the ways in which $G$ can be built according to these rules. In every case, we prove the existence of an embedding into either the torus or the Klein bottle. Since these surfaces have Euler characteristic zero, the estimate $\sys(G)^{-1} \geq 4$ follows from Proposition \ref{pro:embeds}. 

We start with graphs built from $F_{13}$.

\begin{lemma}\label{lem:F13}
If $G$ is a graph built from $F_{13}$ as described above, then $G$ admits an embedding into the Klein bottle and hence satisfies $\sys(G)^{-1} \geq 4$.
\end{lemma}

\begin{proof}
Let $e_1$, $e_2$, and $e_3$ be pairwise vertex-disjoint edges of $F_{13}$. Recall that $G$ is obtained by adding a vertex $v_0$ and by adding an edge from $v_0$ to the midpoint of $e_i$ for $1 \leq i \leq 3$. 

First, suppose that the edges $e_1$, $e_2$, and $e_3$ lie on one of the cycles $C_8$, $C_8'$, or $C_{10}$ shown in Figures \ref{fig:F13-C8}, \ref{fig:F13-genus2maybe}, and \ref{fig:F13-C10}, respectively. In this case, the figures show an embedding of $F_{13}$ into the Klein bottle with the property that the images of $C_8$, \(C_8'\), and $C_{10}$ bound a disc. Hence this embedding of $F_{13}$ extends to an embedding of $G$.

\begin{figure}[ht]
  \raisebox{-0.5\height}{\includegraphics[height=2in]{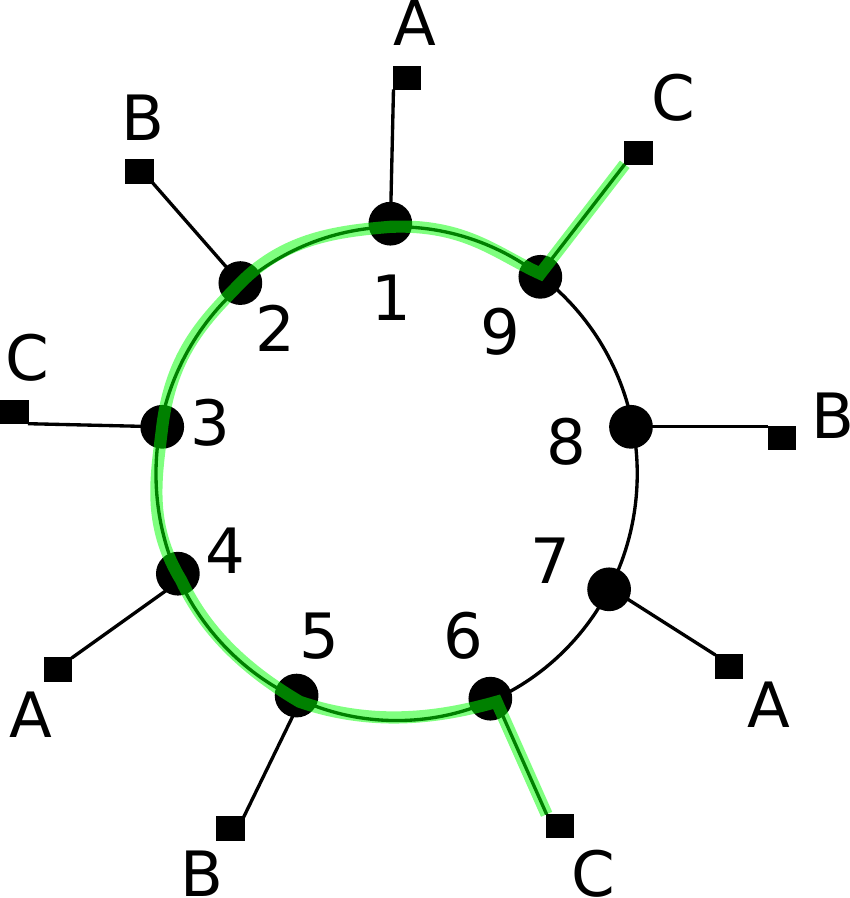}}
  \hspace{0.8cm}
 \raisebox{-0.5\height}{\includegraphics[height=1in]{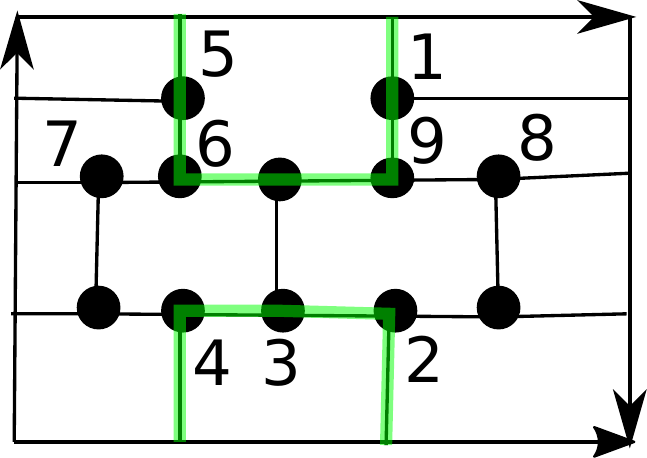}}
  \caption{An embedding of $F_{13}$ into the Klein bottle with the property that the $8$-cycle $C_{8}$ shown has a image bounding a disc.}\label{fig:F13-C8}
\end{figure}

\begin{figure}[ht]
   \raisebox{-0.5\height}{\includegraphics[height=2in]{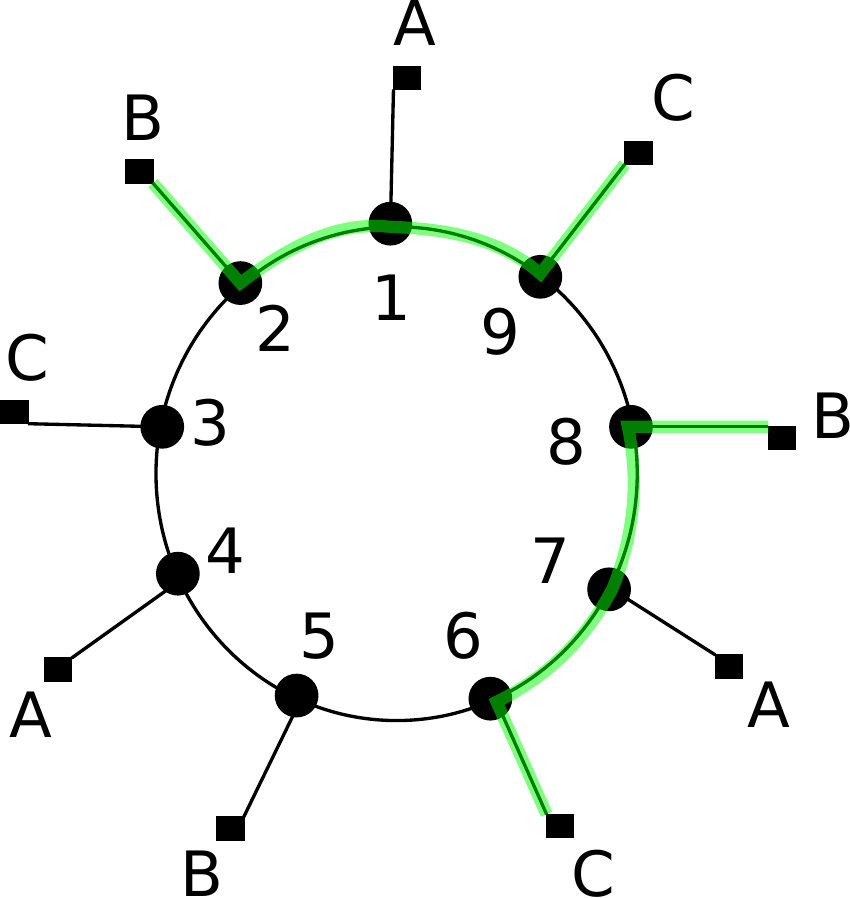}}
\hspace{0.8cm}
   \raisebox{-0.5\height}{\includegraphics[height=1in]{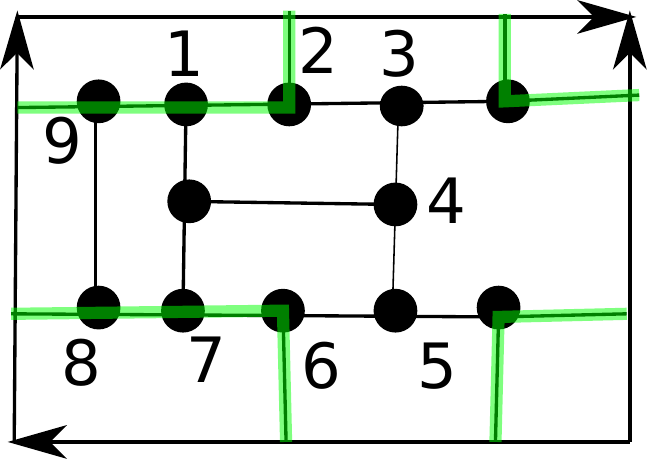}}\caption{The figure shows an embedding of \(F_{13}\) into the Klein bottle such that the cycle \(C_8'\) lies on the boundary of a disc.}\label{fig:F13-genus2maybe}
\end{figure}

\begin{figure}[ht]
   \raisebox{-0.5\height}{\includegraphics[height=2in]{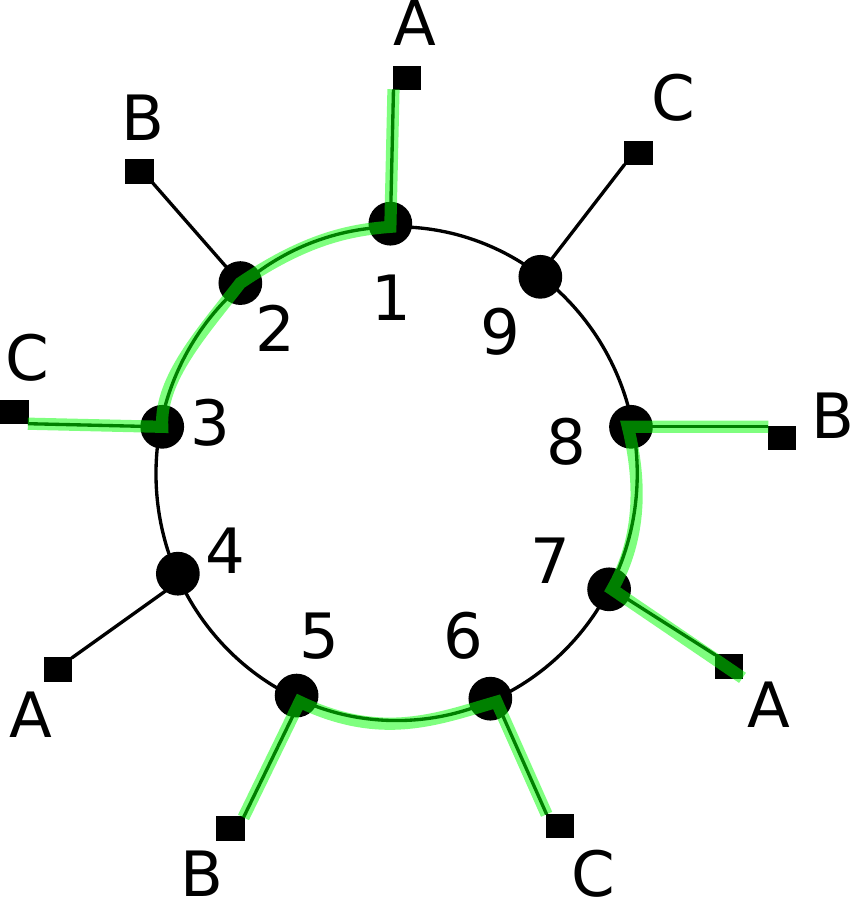}}
  \hspace{0.8cm}
 \raisebox{-0.5\height}{\includegraphics[height=1in]{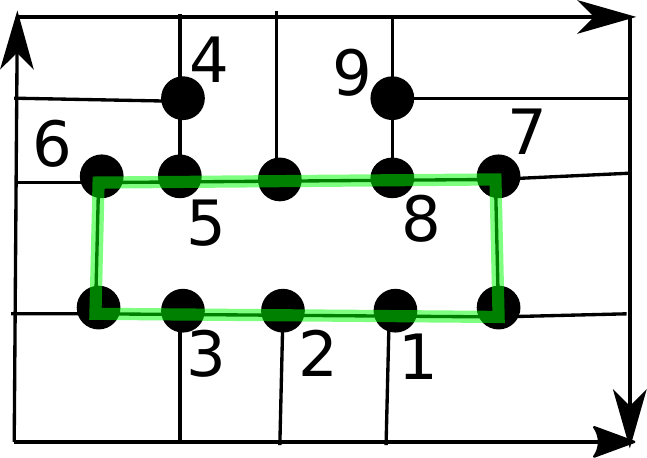}}
  \caption{An embedding of $F_{13}$ into the Klein bottle with the property that the $10$-cycle $C_{10}$ shown has a image bounding a disc.}\label{fig:F13-C10}
\end{figure}

Second, note that we may pre-compose these embeddings with symmetries of $F_{13}$. We claim that, for every triple $\{e_1,e_2,e_3\}$ of pairwise vertex-disjoint edges in $F_{13}$, there exists a symmetry of $F_{13}$ that maps this triple into $C_8$, $C_8'$, or $C_{10}$.

By the above remarks we only have to prove the claim. To arrange the cases, let $C$ denote the $9$-cycle with vertex labels $1,2,\ldots,9$ in the above figures. For any pair of edges $e$ and $e'$ in $F_{13}$, we denote by $d_C(e,e')$ the length of the shortest path in $C$ that connects $e$ and $e'$. Note that $d_C(e_i, e_j) \geq 1$ for all $1 \leq i < j \leq 3$ since these edges are vertex-disjoint.

Case 1: $e_1, e_2, e_3 \not\in C$. There are three possibilities up to symmetries of $F_{13}$, and in all cases, there is a symmetry mapping this triple of edges into $C_{10}$.

Case 2: $e_1, e_2 \not\in C$ and $e_3 \in C$. 
First, if $d_C(e_1, e_2) = 1$ and $e$ is an edge connecting $e_1$ and $e_2$, then there is a symmetry mapping $e$ to the edge with vertices $5,6$ while at the same time mapping $e_3$ onto $C_{10}$. 
Second, if $d_C(e_1, e_2) = 2$, then we can map a connecting path of length two onto the path with vertex labels $6,7$ and $8$ while at the same time mapping $e_3$ onto $C_8'$ or $C_{10}$.
Finally, if neither of these possibilities occurs, the fact that $e_1$ and $e_2$ are vertex-disjoint implies that $d_C(e_1,e_2) = 4$. There are three possible subgraphs $\{e_1,e_2,e_3\}$ in this case up to symmetry, and in each case we can map the subgraph into $C_{10}$.

Case 3: $e_1 \not\in C$ and $e_2, e_3 \in C$.
Assume without loss of generality that $d_C(e_1,e_2) \leq d(e_1,e_3))$.
If $d_C(e_1,e_3) \geq 3$, then there is a symmetry mapping $e_1$ to the edge adjacent to $6$ and $C$ and the triple $\{e_1,e_2,e_3\}$ into $C_8$.
Similarly, if $d_C(e_1, e_3) = 2$, then there is a symmetry mapping $e_1$ to the edge adjacent to $8$ and $B$ or to the edge adjacent to $1$ and $A$ and the triple $\{e_1,e_2,e_3\}$ into $C_{10}$.
Finally, if $d_C(e_1, e_3) = 1$, then we can map \(e_1,e_2,e_3\) to \(C_8'\) by a symmetry of \(F_{13}\). 

Case 4: $e_1,e_2,e_3 \in C$.
We may assume that $d_C(e_1, e_2) \leq d_C(e_2, e_3) \leq d_C(e_1,e_3)$. 
Note that the ordered pair of the first two of these distances is $(1,1)$, $(1,2)$, or $(2,2)$. 
In the first two cases, there is a symmetry mapping the triple into $C_8$. 
In the last two cases, there is a symmetry mapping the triple into $C_{10}$. 
\end{proof}

With the calculation of $\sys(G)$ complete for graphs built from $F_{13}$, it suffices to consider graphs built from $F_{14}$. We do this now.

\begin{lemma}\label{lem:F14}
If $G$ is a graph built from $F_{14}$ as described above, then $G$ embeds into the torus or the Klein bottle and hence satisfies $\sys(G)^{-1} \geq 4$.
\end{lemma}

\begin{proof}
The proof is similar to the case for $F_{13}$. Let $e_1$, $e_2$, and $e_3$ be edges of $F_{14}$ that are pairwise vertex-disjoint. Recall that $G$ is obtained by adding a vertex $v_0$ and by adding an edge from $v_0$ to the midpoint of $e_i$ for $1 \leq i \leq 3$.

Suppose for a moment that all three of the edges lie one one of the cycles $C_8$, $C_9$, or $C_{10}$ shown in Figures \ref{fig:F14-C8}, \ref{fig:F14-C9}, or \ref{fig:F14-C10}, respectively. The figures show an embedding of $F_{14}$ into either the torus or the Klein bottle with the property that the image of $C_8$, $C_{9}$, or $C_{10}$ bounds a disc. It follows that the embedding of $F_{14}$ extends to an embedding of $G$, and the proof is complete in this case.

\begin{figure}[ht]
  \raisebox{-0.5\height}{\includegraphics[height=2in]{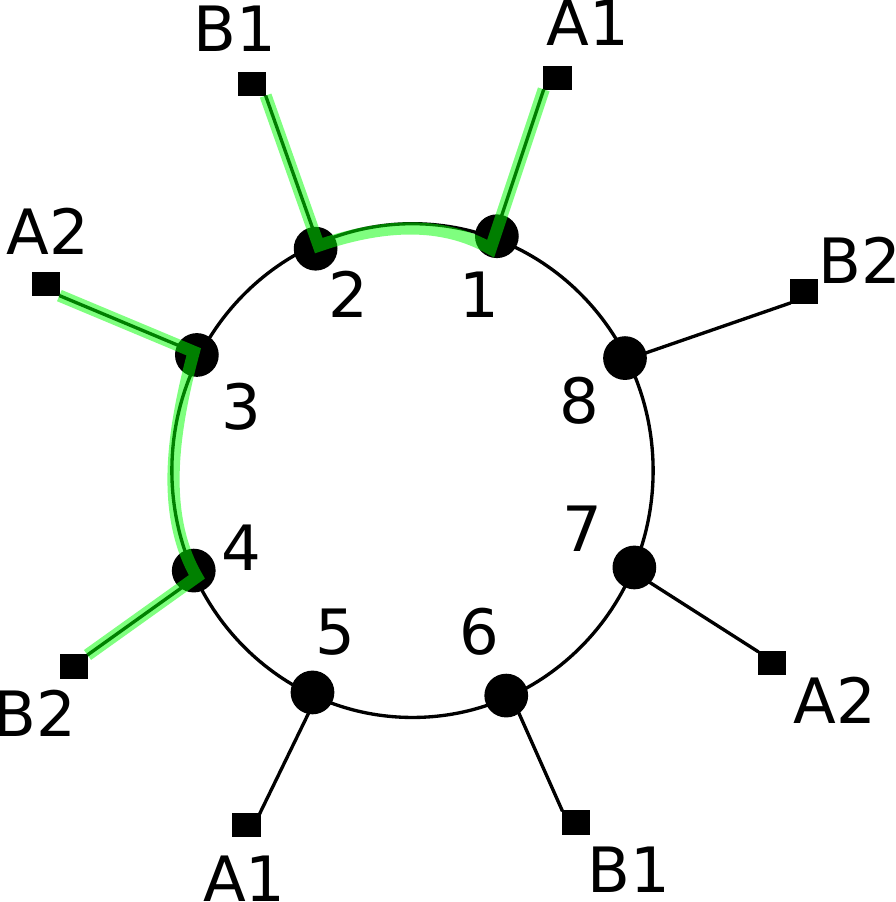}}
  \hspace{0.8cm}
   \raisebox{-0.5\height}{\includegraphics[height=1in]{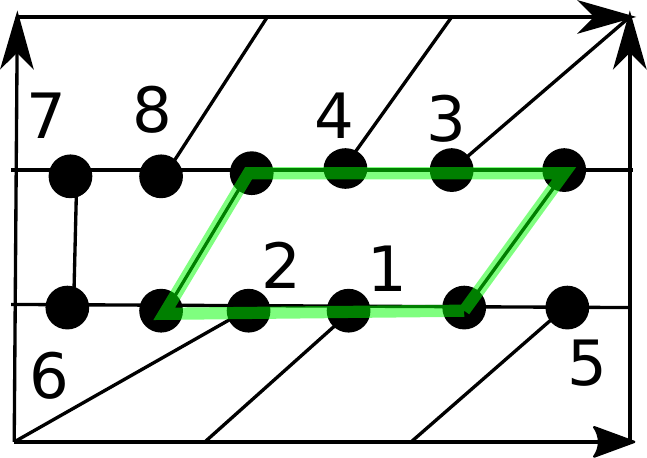}}
\caption{An embedding of $F_{14}$ into the torus with the property that the image of the highlighted $8$-cycle $C_{8}$ bounds a disc.}\label{fig:F14-C8}
\end{figure}

\begin{figure}[ht]
   \raisebox{-0.5\height}{\includegraphics[height=2in]{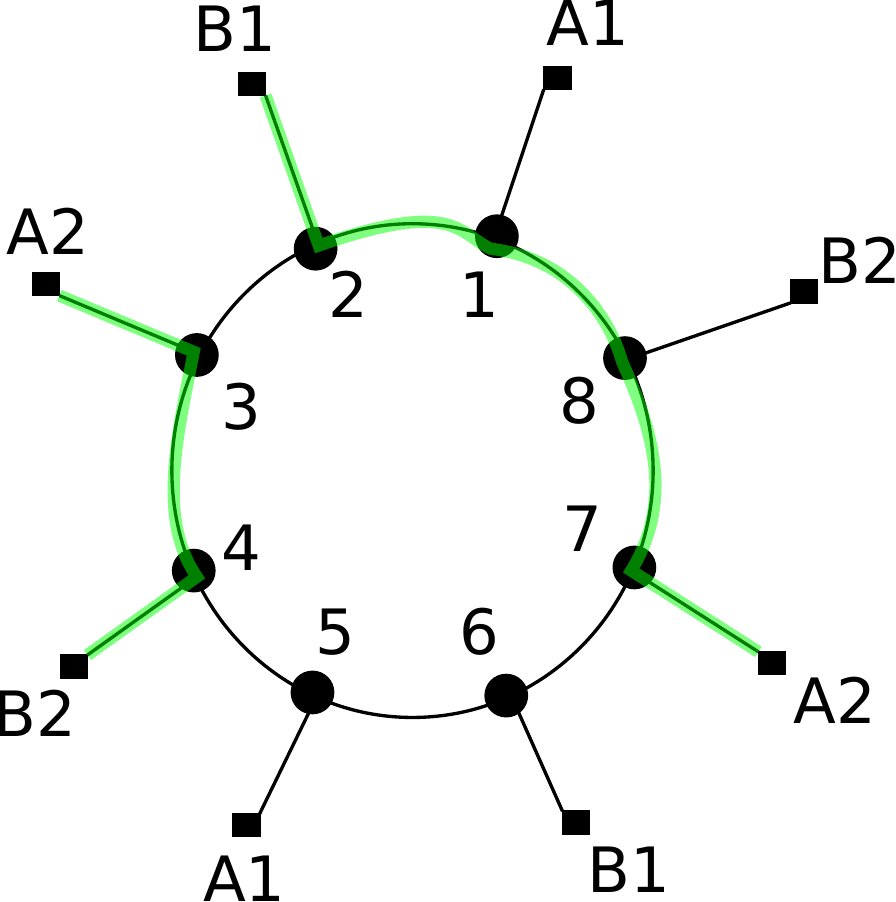}}
  \hspace{0.8cm}
   \raisebox{-0.5\height}{\includegraphics[height=1in]{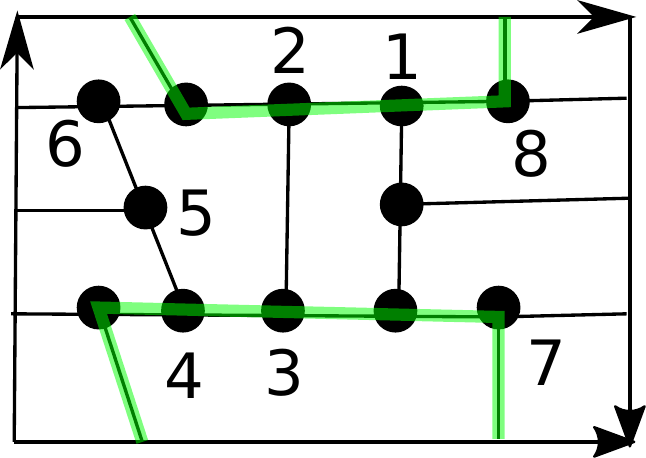}}\caption{An embedding of $F_{14}$ into the Klein bottle with the property that the image of the highlighted $9$-cycle $C_9$ bounds a disc.}\label{fig:F14-C9}
\end{figure}

\begin{figure}[ht]
   \raisebox{-0.5\height}{\includegraphics[height=2in]{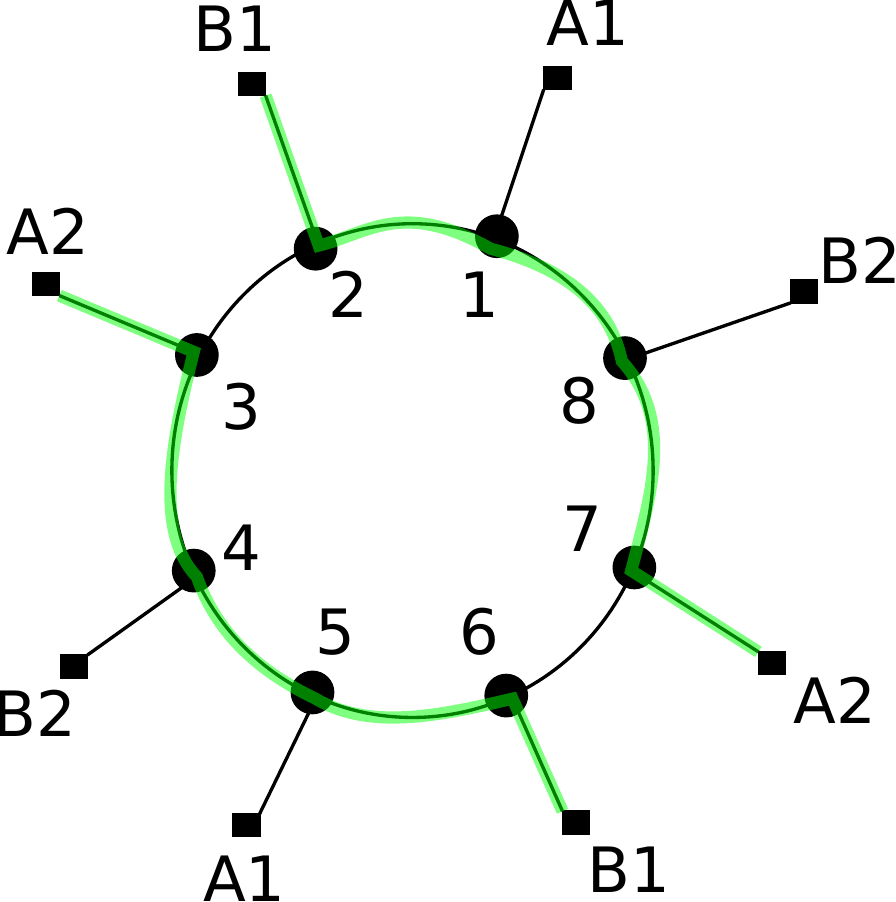}}
  \hspace{0.8cm}
   \raisebox{-0.5\height}{\includegraphics[height=1in]{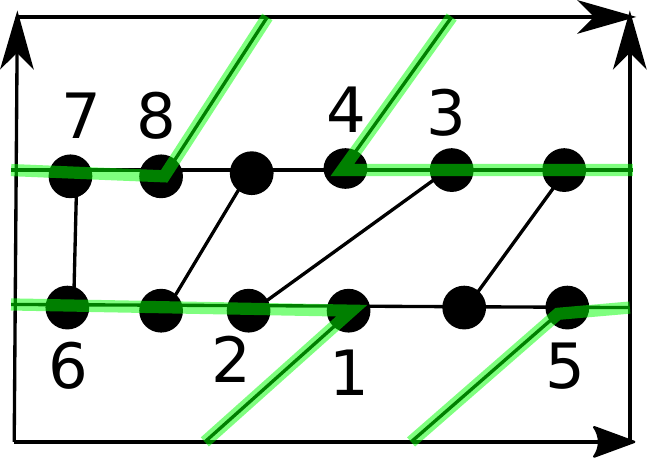}}
\caption{An embedding of $F_{14}$ into the torus with the property that the image of the highlighted $10$-cycle $C_{10}$ bounds a disc.}\label{fig:F14-C10}
\end{figure}

To finish the proof, it suffices to work with the graph $F_{14}$ and to prove that, for any three pairwise vertex-disjoint edges in $F_{14}$, there exists a symmetry of the graph mapping these edges into one of the cycles $C_8$, $C_9$, or $C_{10}$.

We summarize the case-by-case analysis here. Denote by $C$ the $8$-cycle with vertex labels $1,2,\ldots,8$, and denote by $H$ and $H'$ the two connected components of $F_{14} \setminus C$.

Case 1: $e_1, e_2 \in H$. We may assume that $e_1$ and $e_2$ are adjacent to the vertices labeled $4$ and $2$, respectively. Using symmetry once more, there are six possibilities for $e_3$. In two cases, the three edges can be mapped into $C_8$, and in the other four cases, they can be mapped into $C_9$. 

Case 2: $e_1 \in H$ and is vertex-disjoint from $C$. Since the edges are vertex-disjoint, neither $e_2$ nor $e_3$ is in $H$. Moreover by Case 1, we may assume that $e_3 \in C$. Using symmetry, we may assume $e_1$ is on the cycle $C_9$. There are seven possibilities for $e_2$ and $e_3$. In all but one, there is a symmetry mapping the three edges into $C_9$. In the remaining case, $e_2 \in H'$ and is vertex-disjoint from $C$, and we can map them into $C_8$. (Notice in the last case, $G$ is obtained by adding an edge to the Heawood graph, so there is also an embedding into the torus by Lemma \ref{lem:Heawood}.)

Case 3: $e_1 \in H$ and $e_2 \in H'$. By Case 2, we may assume that $e_1$ and $e_2$ share a vertex with $C$, and by Case 1, we may assume that $e_3 \in C$. There are three cases, and we can move each of them into $C_{9}$. (We can also use $C_{10}$ for all three cases here.)

Case 4: $e_1, e_2 \in C$. If $e_3$ is also in $C$, there are two possible graphs, and if $e_3 \not\in C$, then there are five possible graphs. In all seven cases, we can map the triple of edges into $C_{10}$.

Since $H$ and $H'$ are the same up to symmetry, these four cases are exhaustive and the proof is complete.
\end{proof}

\bigskip
\section{Cogirth bounds for regular matroids of small rank}
\label{sec:cogirth}

  In this section we generalize the bounds in Theorem~\ref{thm:bounds-cogr-matr} for cographic representations to all torus representations without finite isotropy groups of even order.
  To do so, we have to introduce some notions from matroid theory.
  Let us recall the quantity which we want to bound.

\begin{definition}\label{def:c}
For a torus representation $\rho:\gT^d\to\SO(V)$ without finite isotropy groups of even order, let
	\[c(\rho) = \min \cod(V^{\gS^1})/\dim(V),\]
where the minimum runs over all circles $\gS^1 \subseteq \gT^d$, and let 
	\[c(d) = \max c(\rho)\]
where the maximum runs over all such representations.
\end{definition}

Upper bounds on $c(d)$ imply the existence of fixed-point sets with large dimension. For example, our calculations in this section show that $c(3) = \frac 1 2$. This implies that, for any $\gT^3$-representation on $V$ without even-order finite isotropy groups, there exists a circle in $\gT^3$ whose fixed-point set has codimension at most $\frac 1 2 \dim V$.

\begin{theorem}\label{thm:cd}
If $\rho:\gT^d\to\SO(V)$ is a torus representation without finite isotropy groups of even order, then $c(\rho)\leq c(d)$ where
        	\[
	\begin{array}{c|c|c|c|c|c|c|c|c|c}
	d 			& 1 & 2 	& 3 & 4 		& 5 		& 6 & 7 		& 8    & 9\\\hline
	 c(d)	& 1 & 2/3 	& 1/2 & 4/9		& 2/5	& 1/3 &  3/10 	& 2/7 & 1/4
	\end{array}
      \]
Moreover the upper bound of $c(d)$ is optimal for all $d \leq 9$.
\end{theorem}

Our proof of Theorem~\ref{thm:cd} uses a deep classification result of Seymour for regular matroids and Theorem~\ref{thm:bounds-cogr-matr}. In Sections \ref{sec:Background} and \ref{sec:Optimization}, we introduce the definitions and results on matroids that we use to prove Theorem \ref{thm:cd} and recast the desired bound in Theorem \ref{thm:cd} in terms of an optimization problem on regular matroids. Given this setup, we prove Theorem \ref{thm:cd} in the sporadic, graphic, cographic, and decomposable cases of Seymour's theorem in Sections \ref{sec:graphic} and \ref{sec:decomposable}. Finally, in Section \ref{sec:6involutions}, we prove Theorem \ref{thm:6involutions}, which is a refinement of Theorem \ref{thm:cd} for $\gT^6$-representations that is required in the proof of Theorem \ref{thm:t9}.

\subsection{Background on matroid theory and Seymour's theorem}
\label{sec:Background}

Finding upper bounds for the codimension of a fixed-point set of a subgroup of a torus in a representation of that torus without finite isotropy groups of even order can be translated to an optimization problem for regular matroids.
In this and the next section, we describe how this translation works.
We start by recalling the basic definitions of matroid theory.
As a general reference for this subject we refer the reader to the book \cite{Oxley_matroids} and the survey articles \cite{welsh_matroids,seymour_survey}.

We start with the definition of a matroid.

\begin{definition}
  A matroid \(M\) is a pair \((E,\mathcal{I})=(E(M),\mathcal{I}(M))\), where \(E\) is a finite set and \(\mathcal{I}\) is a family of subsets of \(E\), such that
  \begin{enumerate}
  \item \(\emptyset\in \mathcal{I}\).
  \item If \(I\in \mathcal{I}\) and \(I'\subset I\), then \(I'\in \mathcal{I}\).
  \item If \(I_1\) and \(I_2\) are in \(\mathcal{I}\) and \(|I_1|<|I_2|\), then there exists an element \(e\in I_2\setminus I_1\) such that \((I_1\cup \{e\})\in \mathcal{I}\).
  \end{enumerate}
  The elements of \(\mathcal{I}\) are called the independent subsets, and the subsets of \(E\) which are not in \(\mathcal{I}\) are called the dependent subsets.
\end{definition}

The two basic examples of matroids are as follows (see \cite[Chapter 1]{Oxley_matroids}).

\begin{example}
  Let \(\mathbb{F}\) be a field, \(V\) an \(\mathbb{F}\)-vector space, and \(E\) a finite collection of vectors in \(V\).
  Denote by \(\mathcal{I}\) the linearly independent subsets of \(E\).

  Then \(M[E]=(E,\mathcal{I})\) is called an \(\mathbb{F}\)-regular or \(\mathbb{F}\)-representable matroid.
  A matroid which is \(\mathbb{F}\)-regular for every field \(\mathbb{F}\) is just called regular.
\end{example}

\begin{example}
  Let \(G\) be a finite directed graph, \(E\) the set of edges of \(G\), and
  \[\mathcal{I}=\{F\in \mathcal{P}(E)\st F \text{ is a forest in } G\}.\]
  Then \(M[G]=(E,\mathcal{I})\) is called a graphic matroid.

  Graphic matroids are \(\mathbb{F}\)-representable for every field \(\mathbb{F}\).
  A representation can be constructed as follows:
  Let \(v_1,\dots,v_n\) be the vertices of \(G\) and \(e_1,\dots,e_m\) be the edges of \(G\). Then let \(f_1,\dots,f_m\in \mathbb{F}^n\) be the vectors with entries
  \begin{equation*}
    f_{ij}=
    \begin{cases}
      1&\text{if } e_i \text{ ends in } v_j,\\
      -1&\text{if } e_i \text{ starts in } v_j,\\
      0& \text{else},
    \end{cases}
  \end{equation*}
  if \(e_i\) is not a loop and \(f_i=0\) if \(e_i\) is a loop.
  Then it can be shown that \(M[G]\) and \(M[f_1,\dots,f_m]\) are isomorphic.
  Hence \(M[G]\) is regular.

  Note that the matrix with columns \(f_1,\dots,f_m\) represents the boundary operator in the cellular chain complex \(C_*(G)\) of \(G\) (viewed as a CW-complex) with respect to the standard basis.
\end{example}

There is the following characterization of regular matroids which goes back to Tutte (see \cite[Theorem 6.6.3]{Oxley_matroids}).

\begin{theorem}
  \label{sec:char-regular-matroid}
  A matroid is regular if and only if it is representable over \(\mathbb{Z}_2\) and $\Q$.
\end{theorem}

As for sets of vectors, there is a rank function for the matroid \(M=(E,\mathcal{I})\). It assigns to a subset \(A\subset E\) the cardinality of a maximal independent subset \(I\subset A\).
This cardinality is independent of the choice  of the maximal independent subset of \(A\).
In particular, the maximal independent subsets of \(E\) all have the same cardinality.
These subsets are called {\it bases} of \(M\).
The rank \(\rank E\) of $E$ is also denoted \(\rank M\).

Moreover, motivated by the situation of a representable matroid, the maximal sets \(H\subset E\) with \(\rank(H)=\rank(M)-1\) are called {\it hyperplanes} of \(M\). They will play a special role in our optimization problem. Similarly, motivated by the situation of a graphic matroid the minimal dependent subsets of a matroid are called {\it circuits}. For graphic matroids these are just the (simple) cycles of the graph.

For every matroid, there is the notion of a {\it dual matroid}.
It is defined as follows (see \cite[Chapter 2]{Oxley_matroids}):

\begin{thm_def}
  Let \(M\) be a matroid on a set \(E\), and let $\mathcal B$ be the set of bases of $M$. The set $\mathcal B^* = \{E \setminus B \st B \in \mathcal B\}$ is a set of bases for a matroid on $E$. This matroid is denoted by $M^*$ and is called the dual of $M$.     In particular, the independent sets in $M^*$ are the subsets of the form $E \setminus A$ such that $A\supseteq B$ for some $B \in \mathcal B$.

\end{thm_def}

Duality is an important notion, and we state some properties we need later. First, by definition, the dual of the dual is the original matroid. 

Second, $A \subseteq E$ is a hyperplane of $M$ if and only if $E \setminus A \subseteq E$ is a circuit of $M^*$ (see \cite[Proposition 2.1.6]{Oxley_matroids}).

Third, the duals of \(\mathbb{F}\)-representable matroids are \(\mathbb{F}\)-representable.

Fourth, the dual of a graphic matroid \(M[G]\) is not graphic in general. It is graphic if and only if \(G\) is a planar graph, and in this case the matroid $M[G]$ is called planar and its dual is the graphic matroid of the dual graph of $G$. The duals of graphic matroids are called {\it cographic}.

Using cellular cohomology a representation over a field \(\mathbb{F}\) for a cographic matroid \(M^*[G]\) can be constructed as follows.
Let \(V=H^1(G;\mathbb{F})=C^1(G;\mathbb{F})/d^*(C^0(G;\mathbb{F}))\)
Then we identify the edges of \(G\) with the images of the elements of the standard basis of \(C^1(G;\mathbb{F})\) in \(H^1(G;\mathbb{F})\).
Using the relation of \(M[G]\) to the boundary operator in the cellular chain complex for \(G\) and the definition of the dual of a matroid, we see that this identification, leads to a representation of \(M^*[G]\) in \(V\). Note in particular that $M^*[G]$ has rank equal to the first Betti number of the graph $G$.

There is a deep classification result for regular matroids due to Seymour \cite{Seymour80}.
To state it, we have to define certain constructions of new matroids from old ones called $k$-sums for $k \in \{1,2,3\}$. Here we only describe these constructions for \(\mathbb{F}_2\)-regular matroids.
For a general discussion see \cite{Oxley_matroids}, \cite{seymour_survey}, or \cite{Seymour80}.

Let \(M_i=(E_i,\mathcal{I}_i)\), \(i=1,2\) be \(\mathbb{F}_2\)-representable matroids, say \(E_i\subset V_i\) for \(\mathbb{F}_2\)-vector spaces \(V_i\).

\begin{enumerate}
\item The {\it $1$-sum} \(M_1\oplus_1 M_2\) is given by \(M[E_1\cup E_2]\), where the \(E_i\) are viewed as collections of vectors in \(V_1\oplus V_2\).

\item Assume that $|E_i| \geq 2$ and that there are non-zero elements \(v_i\in E_i\) for $i \in \{1,2\}$.  The {\it $2$-sum} \(M_1\oplus_2 M_2\) at the \(v_i\) is defined as \(M[E']\) where \(E'\) is image of \(E_1\cup E_2\setminus \{v_1,v_2\}\) in \((V_1\oplus V_2)/\langle v_1-v_2\rangle\)  (see \cite[Proposition 7.1.20]{Oxley_matroids}).

\item Assume that $|E_i| \geq 7$ and that there are two-dimensional subspaces \(W_i\subset V_i\) such that \(W_i\setminus\{0\}\subset E_i\) for $i \in \{1,2\}$, and fix an identification \(W_i\cong \mathbb{F}^2_2\). The {\it $3$-sum} \(M_1\oplus_3M_2\) is defined as \(M[E'']\), where \(E''\) is the image of \((E_1\cup E_2)\setminus(W_1\cup W_2)\) in \((V_1\oplus V_2)/W\) and where \(W\) is the diagonal \(\mathbb{F}_2^2\) in \(W_1\oplus W_2\cong \mathbb{F}_2^2\oplus \mathbb{F}^2_2\) (see \cite[Lemma 9.3.3]{Oxley_matroids}).
\end{enumerate}

Note that the last two constructions depend on the choice of the \(v_i\) and \(W_i\) (and the identifications of the latter), respectively.
Note, moreover, that \(M_1\oplus_k M_2\) for $k \in \{1,2,3\}$ is graphic (or cographic) if and only if \(M_1\) and \(M_2\) are graphic (or cographic, respectively). 
However it is not true that the dual of a \(3\)-sum is the \(3\)-sum of the duals of the summands.

Now we can state Seymour's theorem.

\begin{theorem}[Seymour's classification of regular matroids, \cite{Seymour80}]
\label{sec:backgr-matr-theory}
  Every regular matroid is (at least) one of the following:
  \begin{enumerate}
  \item a graphic matroid, 
  \item a cographic matroid, 
  \item the sporadic matroid $R_{10}$, or
  \item a (non-trivial) $k$-sum of regular matroids for some $k \in \{1,2,3\}$.
  \end{enumerate}
\end{theorem}

Here the sporadic matroid $R_{10}$ can be represented by the following matrix:
\[A_{R_{10}} = \left(\begin{matrix}1&0&0&0&0&-1&1&0&0&1\\0&1&0&0&0& 1&-1&1&0&0 \\0&0&1&0&0& 0&1&-1&1&0 \\0&0&0&1&0& 0&0&1&-1&1\\0&0&0&0&1&1&0&0&1&-1\end{matrix}\right),\]

Note in particular that regular matroids of rank six or larger are graphic, cographic, or decomposable as a $k$-sum. On the other hand, it is not hard to see that all (simple) regular matroids of rank at most four are graphic with one exception, the cographic matroid $M^*[K_{3,3}]$ associated to the Kuratowski graph $K_{3,3}$ (see Example \ref{exa:cigUpTo5}).

\subsection{An optimization problem for matroids}
\label{sec:Optimization}

Having introduced the necessary definitions and results from matroid theory, we can state our optimization problem. We first do this for \(\mathbb{Q}\)-representable matroids. Let
\[H=[h_1\,\dots\,h_{n+d}] \in \mathbb{Z}^{d\times (d+n)}\]
a matrix with \(d\) rows and \(d+n\) columns such that every invertible \(d\times d\) submatrix \(H'\) of \(H\) has odd determinant.

So for any \(d\times d\) submatrix \(H'\) of \(H\) we have 
\begin{equation}
  \label{eq:submatrix}
  \det H'\in \{0 \} \cup (2\mathbb{Z}+1).
\end{equation}

Let, moreover, \(\lambda=(\lambda_1,\dots,\lambda_{n+d})\in [0,1]^{n+d}\) such that
\[\sum_{i=1}^{n+d} \lambda_i =1.\]
Define
\(f_{H,\lambda}:(\mathbb{Q}^d)^*\rightarrow [0,1]\) by
\begin{equation}
  \label{eq:1}
  f_{H,\lambda}(v)=\sum_{\; \langle v, h_i \rangle \neq 0} \lambda_i=1-\sum_{\; \langle v, h_i \rangle = 0}\lambda_i
\end{equation}
and define
\[c(H,\lambda)= \min_{v\in (\mathbb{Q}^d)^*-\{0\}} f_{H,\lambda}(v).\]

Note that this minimum is zero for any choice of $\lambda$ if the columns of \(H\) do not span \(\mathbb{Q}^d\). Therefore, in the following, we assume that \(\rank H = d\). Note, moreover, that we can define a similar optimization problem for matrices \(H\) with coefficients in any field and weight vectors \(\lambda\) as above.

We are interested in finding {\it upper bounds} \(c(d)\) for \(c(H,\lambda)\) which {\it only depend on the dimension} \(d\). We can translate this optimization problem to the language of general matroids.
This goes as follows.

Let \(M\) be the matroid represented by \(H\). By Condition \eqref{eq:submatrix}, a subset of columns of \(H\) are linearly independent if and only if their reductions modulo two are linearly independent. Therefore \(M\) is both $\Q$-representable and $\mathbb{Z}_2$-representable. By Theorem \ref{sec:char-regular-matroid}, it follows that $M$ is a regular matroid.

Next, \(\lambda\) can be interpreted as a weight vector \(\lambda\in \mathbb{R}^{E(M)}\) or equivalently a probability measure on \(E(M)\). We denote this measure also by \(\lambda\). Finding \(c(H,\lambda)\) is equivalent to finding the minimum of
	\[1-\lambda(A)=\lambda(E(M)\setminus A),\]
where \(A\) runs through the hyperplanes of \(M\), that is, maximal subsets $A \subseteq E(M)$ having rank $\rank(M) - 1 = d-1$.

We then define, for a regular matroid \(M\), the invariant
	\[c(M)=c(H)=\sup_\lambda c(H,\lambda)=\sup_\lambda \min_A \lambda(E(M)\setminus A),\]
where $H$ is a matrix representing $M$ and where the supremum is taken over all probability measures \(\lambda\) on \(E(M)\) and \(A\) runs through the hyperplanes of \(M\).

As mentioned before, \(A\subset E(M)\) is a hyperplane of \(M\) if and only if \(E(M)\setminus A\) is a circuit of the dual matroid \(M^*\). Therefore we also have
\[c(M)=\sup_\lambda \min_C \lambda(C),\]
where the supremum is taken over all probability measures \(\lambda\) on \(E(M)\) and \(C\) runs through the circuits of \(M^*\). For this reason, the matroid invariant $c(M)$ is called the cogirth (see \cite{CrenshawOxley-pre}). 

We also note that the cogirth $c(M)$ of a cographic matroid $M = M^*[G]$ associated to a graph $G$ equals the systole
	\[\sys(G) = \sup_\lambda \min_C \lambda(C)\]
        of the graph $G$, where the supremum runs over probability measures $\lambda$ on the set of edges in $G$ and the infimum runs over cycles $C$ in the graph $G$. 
This is because the circuits of a graphic matroid \(M[G]\) are precisely the cycles of the graph \(G\).

\begin{example}
\label{sec:an-optim-probl}
  We are in particular interested in the following situation:
  Let $\rho$ be a faithful representation of $\gT^d$ on $W$ without finite isotropy groups of even order. 
  Let \(H\) be the \emph{weight matrix} of \(W\), that is, the matrix whose columns are the weights of \(W\). The assumption on isotropy groups is equivalent to Condition \eqref{eq:submatrix} on the determinants of submatrices. In particular, $H$ represents a regular matroid $M(\rho)$. 
  Finally, let
  	\[\lambda_i=\dim W_i/\dim W,\]
where \(W_i\) is the weight space of \(W\) corresponding to the \(i\)-th column of \(H\). The $\lambda_i$ may be viewed as normalized multiplicities of the irreducible subrepresentations of $\rho$. 
Given this setup, we have that the geometric quantity
	\[c(\rho) = \min_{\gS^1\subset \gT^d} \cod W^{S^1}/\dim W\]
from Definition \ref{def:c} that we wish to bound is related to our optimization problem by
	\[c(\rho) = c(H,\lambda).\]
In particular, $c(\rho) > 0$ if and only if $W$ is a faithful representation. Therefore, to prove Theorem \ref{thm:cd} on the codimensions of fixed-point sets of these special types of torus representations, it suffices to replace $\rho:\gT^d \to \SO(W)$ by an arbitrary regular matroid $M$ of rank $d$, and to prove $c(\rho)$ is bounded above by a constant $C$, it suffices to prove that $c(H) = \sup_\lambda c(H,\lambda) \leq C$.
\end{example}

Having translated the claim in Theorem \ref{thm:cd} into the language of matroids, we are ready to prove Theorem \ref{thm:cd} in Sections \ref{sec:graphic} and \ref{sec:decomposable}. Finally Section \ref{sec:6involutions} contains the proof of Theorem \ref{thm:6involutions}.

\subsection{The graphic, cographic, and sporadic cases}
\label{sec:graphic}

In this section we prove Theorem~\ref{thm:cd} in the graphic and sporadic cases. 

\begin{proposition}
\label{sec:graphic-1}
  For \(d\geq 1\) we have
  \[c_g(d)\leq \frac{2}{d+1}.\]
  Here \(c_{g}(d)\) is defined as
  \[c_{g}(d)=\sup\{c(M[G])\st  M[G] \text{ is a graphic matroid of rank } d\}.\]
\end{proposition}

This proposition is well known (see \cite{CrenshawOxley-pre}). But for the sake of completeness we also give a proof here which was communicated to us by James Oxley.

\begin{proof}
  Let \(G\) be a connected graph with \(d+1\) vertices and edge set \(E\).
  Then the circuits of \(M^*[G]\) are the minimal cut-sets of \(G\), i.e. the minimal sets \(A\subset E\) such that \(G\setminus A\) is disconnected.
  If \(v\) is a vertex of \(G\) then the set \(B_1(v)\setminus v\) of edges adjacent to \(v\) is a cut-set of \(G\). Therefore for every probability measure \(\lambda\) on \(E\) we have
  \[c(M[G],\lambda)\leq \lambda(B_1(v)\setminus v).\]
  Summing over all \(d+1\) vertices of \(G\) and maximizing over \(\lambda\) now gives the result.
\end{proof}

As we already explained the cocircuits of a cographic matroid \(M^*[G]\) are the circuits of $M[G]$ and hence the cycles of the graph \(G\). Moreover, the rank of the cographic matroid \(M^*[G]\) is the first Betti number of \(G\) (see Section \ref{sec:Background}). Therefore the optimal bound on the cogirth of a weighted cographic matroid is equivalent to the systole bound for graphs discussed in Theorem~\ref{thm:bounds-cogr-matr}.

\begin{proposition}
  \label{sec:sporadic-1}
  For the sporadic matroid  \(R_{10}\) we have \(c(R_{10})= \frac{2}{5}\).
\end{proposition}

\begin{proof}
  Denote by \(e_1,\dots,e_5\) the standard basis of \(\mathbb{F}_2^5\), by \(e_1^*,\dots,e_5^*\) the dual basis of  \((\mathbb{F}_2^5)^*\) and by \(f_1,\dots,f_5\) the last five columns of \(A_{R_{10}}\).
  Let \(\lambda\) be a probability measure on the set of columns of \(A_{R_{10}}\).
  Then we have, for $i \in \{1,\ldots,5\}$,
  \(f_{A_{R_{10}},\lambda}(e_i^*+e_{i+1}^*)=\lambda(e_i)+\lambda(e_{i+1})+\lambda(f_{i-1})+\lambda(f_{i+2}),\)
  where we think of the indices as elements of \(\mathbb{Z}/5\mathbb{Z}\).
  Therefore by averaging over \(i=1,\dots,5\) we get the upper bound.
  It is attained if all columns of \(A_{R_{10}}\) have the same weight.
\end{proof}

Since $R_{10}$ is the only sporadic matroid in Seymour's classification, the computation of $c(R_{10})$ is all we need to do for $c(d)$ in the sporadic case. Hence Theorem \ref{thm:cd} holds in the sporadic case.

\subsection{The decomposable case}
\label{sec:decomposable}

In this section, we prove Theorem \ref{thm:cd} in rank $d$ in the decomposable case of Seymour's theorem, under the assumptions that $d \leq 9$ and that Theorem \ref{thm:cd} holds for ranks less than $d$. The main step is the following (see \cite{ChoChenDing07} for the case when $k = 1$):

\begin{proposition}
  \label{pro:bounds-decomp}
  For \(d\geq 4\) we have
  \begin{equation*}
    c_{de}(d)^{-1}\geq \min \{c(d_1)^{-1}+c(d_2)^{-1}\st  d-2 = d_1+d_2 \mathrm{~and~}d_1,d_2\geq 1\}.
  \end{equation*}
  Here \[c_{de}(d)=\sup\{c(M)\st  M \text{ is a decomposable regular matroid of rank } d\}.\]
  Here by a decomposable matroid we mean a matroid \(M\) such that the following holds:
  \begin{enumerate}
  \item Any two elements of \(E(M)\) are independent.
  \item \(M=M_1\oplus_k M_2\) decomposes as a \(k\)-sum, \(k\in \{1,2,3\}\), such that \(M\) is not isomorphic to \(M_1\) or \(M_2\).
  \end{enumerate}
  
  If in the definition of \(c_{de}(d)\) we restrict to graphic (or cographic) matroids, then we can replace the \(c(d_i)\) on the right hand side of the above inequality by the corresponding bound for graphic (or cographic, respectively) matroids.
\end{proposition}

Before proving this result, we explain how to conclude Theorem \ref{thm:cd} in the decomposable case (assuming the theorem in smaller ranks). 

First, if a decomposable matroid is also graphic, cographic, or sporadic, then Theorem \ref{thm:cd} holds by the proof in these cases. In particular, since all regular matroids of rank at most five are one of these types, we may assume the rank $d \geq 6$.

Second, we may assume inductively that Theorem \ref{thm:cd} holds in ranks less than $d$. Hence for $d = 6$, for example, Proposition \ref{pro:bounds-decomp} implies that
	\[c_{de}(6)^{-1} \geq \min\of{c(1)^{-1} + c(3)^{-1}, c(2)^{-2} + c(2)^{-1}}.\]
Since $c(1) = 1$, $c(2) = 2/3$, and $c(3) = 1/2$ by the inductive hypothesis, we derive that
	\[c_{de}(6)^{-1} \geq 3.\]
This is equivalent to the claimed bound for $d = 6$. The proof for $d \in \{7,8,9\}$ is similarly straightforward in all cases with one exception.

From the above proposition, the only possibility that the bound for \(c_{de}(d)\) does not hold is the case that $d = 7$ and there is a rank-\(7\) regular matroid \(M\) which decomposes as a $3$-sum of a rank-\(3\) matroid and a rank-\(6\) matroid but does not decompose as a $1$- or $2$-sum. Fortunately we can avoid this case. Indeed, since $M$ is not a $1$-sum or $2$-sum, \cite[Section 8.3]{Oxley_matroids} implies that \(M\) is \(3\)-connected, which implies by \cite[Corollary 13.4.6]{Oxley_matroids} that one of the following holds:
        \begin{itemize}
        \item \(M\) is graphic
        \item \(M\) is cographic
        \item \(M\) decomposes as a three-sum of two matroids of ranks at least \(4\).
        \end{itemize}
In all of these cases it follows from the above propositions that the bound \(c(7)\) given in the table hold for \(M\). Hence Theorem \ref{thm:cd} in the decomposable case follows from the result in smaller ranks together with Proposition \ref{pro:bounds-decomp}.

We proceed to the proof of Proposition \ref{pro:bounds-decomp}. We need the following lemma.

\begin{lemma}
  \label{sec:bounds-decomp-matr}
  Let \(M=M_1\oplus_k M_2\) for \(\mathbb{F}_2\)-regular matroids \(M_1\) and \(M_2\) and \(k=1,2,3\).
  Then:
  \begin{enumerate}
  \item  \(\rank(M)\leq\rank(M_1)+\rank(M_2)-k+1\),
  \item \(c(M)^{-1}\geq c(M_1')^{-1}+c(M_2')^{-1}\), where \(M_i'\) is a minor of \(M_i\) with \(\rank M_i'=\rank M_i-k+1\), \(i=1,2\).
  \end{enumerate} 
\end{lemma}
\begin{proof}
  First assume that \(k=1\).
  We use the same notation as in the definition of the one-sum.
  The first claim is obvious from the definition of the one-sum.
  Let \(\lambda\) be a probability measure on \(M\), then by considering vectors \(v\in V_i^*\subset (V_1\oplus V_2)^*\) in the sums (\ref{eq:1}) we find:
  \[c(M)\leq \min (c(M_1)\lambda(E_1),c(M_2)\lambda(E_2))=\min (c(M_1)\lambda(E_1),c(M_2)(1-\lambda(E_1))\]
  The second claim now follows as in the proof of the Small Cycle Estimate (see also the proof of Lemma \ref{sec:d=7-not-sufficient}).

  Next assume \(k=2\).
  We use the same notation as in the definition of the two-sum.
  The first claim is obvious from the definition of the two-sum.
  Let \(\lambda\) be a probability measure on \(M\), then by considering vectors \(v\) in one of the first two summands of  the splitting \[((V_1\oplus V_2)/\langle v_1-v_2 \rangle)^*\cong (V_1/\langle v_1 \rangle)^*\oplus (V_2/\langle v_2\rangle)^*\oplus \mathbb{F}_2^*\] in the sums (\ref{eq:1}) we find:
  \begin{align*}c(M) & \leq \min (c(M_1')\lambda(E_1\setminus \{v_1\}),c(M_2')\lambda(E_2\setminus \{v_2\}))\\ & =\min (c(M_1')\lambda(E_1\setminus \{v_1\}),c(M_2')(1-\lambda(E_1\setminus \{v_1\})).\end{align*}
  Here \(M_i'\) is the matroid with \(M_i'=M[E_i']\) where \(E_i'\) is the image of \(E_i\setminus \{v_i\}\) in \(V_i/\langle v_i\rangle\), \(i=1,2\).
  The second claim now follows as in the previous case.

Last assume \(k=3\).
  We use the same notation as in the definition of the three-sum.
  The first claim is obvious from the definition of the three-sum.
  Let \(\lambda\) be a probability measure on \(M\), then by considering vectors \(v\) in one of the first two summands of  the splitting \[((V_1\oplus V_2)/W)^*\cong (V_1/W_1)^*\oplus (V_2/W_2)^*\oplus (\mathbb{F}_2^2)^*\] in the sums (\ref{eq:1}) we find:
  \begin{align*}c(M)&\leq \min (c(M_1')\lambda(E_1\setminus W_1),c(M_2')\lambda(E_2\setminus W_2))\\ &=\min (c(M_1')\lambda(E_1\setminus W_1),c(M_2')(1-\lambda(E_1\setminus W_1)).\end{align*}
  Here \(M_i'\) is the matroid with \(M_i'=M[E_i']\) where \(E_i'\) is the image of \(E_i\setminus W_i\) in \(V_i/W_i\) for \(i=1,2\).
  The second claim now follows as in the previous case.
\end{proof}

From the above lemma we get
  \begin{equation*}
    c_{de}(d)^{-1}\geq \min \{c(d_1)^{-1}+c(d_2)^{-1}\st  d-2 \leq d_1+d_2, d_1,d_2\geq 1\}.
  \end{equation*}
  The lower bounds for the \(d_i\) follow because we can assume that \(\rank M_i\geq k\) for $i \in \{1,2\}$ if a matroid \(M\) decomposes non-trivially as \(M=M_1\oplus_k M_2\) for some \(k\in\{1,2,3\}\).
  Here by a non-trivial decomposition we mean a decomposition such that \(M\) is not isomorphic to one of the \(M_i\), \(i=1,2\).
   This follows from the fact that we can assume that any two elements of \(E(M)\) are independent and the lower bounds for \(|E(M_i)|\) given in the definitions of the \(k\)-sum, \(k=1,2,3\).

  It is easy to see that \(c(d)\) is a non-increasing function of \(d\).
  Indeed if \(M=M[E]\), \(E\subset V\) is a regular matroid of \(\rank d\), where \(V\) is a vector space and \(E\) a finite multiset.
  Then for a non-zero element \(e\) of \(E\) we can look at \(M_1=M[\bar{E}]\), where \(\bar{E}\) is the image of \(E\) in \(V/\langle e\rangle\).
  Since the preimage of every hyperplane in \(M_1\) is a hyperplane in \(M\), we clearly have
  \[c(M)\leq c(M_1)\leq c(d-1).\]
  Maximizing over \(M\) gives \(c(d)\leq c(d-1)\).
  
Therefore Proposition \ref{pro:bounds-decomp} follows from the above lemma.
  The claim about the graphic and cographic matroids follows because minors of these are graphic and cographic, respectively.

\subsection{Choosing six involutions in $\gT^6$ }
\label{sec:6involutions}

In this section we prove a refinement of Theorem~\ref{thm:cd} which is needed in the proof of Theorem~\ref{thm:t9}.

\begin{theorem}
  \label{thm:6involutions}
 If \(W\) is an almost effective \(\gT^6\)-representation with no even-order finite isotropy groups, then there exist pairwise distinct, non-trivial involutions \(\iota_1 , \ldots , \iota_6 \in \gT^6\) whose fixed-point sets \(W^{\iota_i}\) satisfy
\(k_1 + \dots + k_6 \leq 2n\),
where \(k_i\) is the codimension of \(W^{\iota_i}\) in \(W\) and \(n\) is the dimension of \(W\).
\end{theorem}

To prove this theorem we look at the regular weighted matroid \(M\) of \(W\). By the almost effectiveness of the action it has rank \(6\).
We view it as represented in the \(6\)-dimensional vector space \((\mathbb{Z}_2^6)^*\) over \(\mathbb{Z}_2\).

Recall from Example \ref{sec:an-optim-probl} that an involution $\iota \in \gT^6$ has fixed-point component $W^\iota$ with codimension satisfying $\cod(W^\iota) \leq f_\lambda(v) n$, where $v = \iota \in \Z_2^6$ and $f_\lambda:\Z_2^6 \to \R$ is the function
	\[f_\lambda(v) = \sum_{\inner{e,v} \neq 0} \lambda(e) = \sum_{e \in E - v^\perp} \lambda(e).\]
We claim we can find $6$ codimension-one subspaces $v_i^\perp\subset (\mathbb{Z}_2^6)^*$ such that each $e \in E(M)$ is in at most two $E(M) - v_i^\perp$. Given this claim, it would follow that
	\[\sum_{i=1}^{6} \cod(W^{\iota_i}) = \sum_{i=1}^{6} \sum_{e \in E(M) - v_i^\perp} \lambda(e) n = \sum_e \sum \lambda(e)n,\]
where the unlabeled sum is over $i$-values such that $e \in E(M) - v_i^\perp$. Since $\lambda(e)$ does not depend on $i$ and since there are at most two $i$-values in this sum, the right-hand side is at most $\sum_e 2 \lambda(e) n = 2n$.

To prove the claim, we need to look at all regular matroids of rank $d \leq 6$.
In the following let \(M\) be a regular matroid of rank \(d\) represented over \(\mathbb{Z}_2\) in a \(d\)-dimensional vector space \(V\).
Note that since we are working over \(\mathbb{Z}_2\) the representation of \(M\) in \(V\) is unique up to automorphisms of \(V\).

The first lemma considers the cases of \(M\) being graphic, cographic, or sporadic.

\begin{lemma}\label{lem:6involutions-BuildingBlocks}
If $M$ is graphic of rank \(d\geq 2\), then there are $d+1$ pairwise distinct codimension-one vector subspaces of $V$ such that each element of $E(M)$ is contained in at least $d-1$ of the subspaces.

Similarly, if $M$ is a cographic matroid of rank $d \leq 6$ or the sporadic matroid of rank $d = 5$, then there are $d$ such subspaces such that every element of $E(M)$ is contained in at least $d-2$ of them.
\end{lemma}

\begin{proof}
  The sporadic case follows from an inspection of the arguments in the proof of Proposition~\ref{sec:sporadic-1}.

  So assume that \(M=M[G]\) is graphic with a connected graph \(G\) on \(d+1\) vertices \(v_1,\dots,v_{d+1}\).
  Then we may assume that \(V=\ker \epsilon\) with
  \[\epsilon:\bigoplus_{i=1}^{d+1}\Z_2v_i\rightarrow \Z_2,\quad\quad \sum_{i=1}^{d+1}a_iv_i\mapsto \sum_{i=1}^{d+1} a_i.\]
  Moreover, we may assume that each edge of \(G\) is represented in \(V\) by the sum of its initial and terminal vertices.
  
  If \(d\geq 2\), then the \(d+1\) codimension-one subspaces \(\ker\epsilon\cap \ker\epsilon_j\), \(j=1,\dots,d+1\) with
  \[\epsilon_j:\bigoplus_{i=1}^{d+1}\Z_2v_i\rightarrow \Z_2,\quad\quad \sum_{i=1}^{d+1}a_iv_i\mapsto a_j.\]
  are pairwise distinct and have the desired property.
  
  Next we may assume $M$ is cographic of rank $d \leq 6$, and let $G$ denote the graph such that $M \cong M^*[G]$. 
  It suffices to find \(d\) pairwise distinct (simple) cycles of \(G\) such that each edge of \(G\) is contained in at most two of them.
  
Suppose first that $G$ does not embed into the real projective plane $\RP^2$. By the arguments in Section \ref{sec:cographic_7}, $G$ must contain the graph $G_1$ from Figure \ref{fig:g1} as a subgraph. Hence the six $4$-cycles in the subgraph consisting of two disjoint copies of \(K_{2,3}\) in \(G_1\) have the desired properties.

Assume now that $G$ does embed in $\RP^2$.
We can assume that \(\mathbb{R} P^2-G\) is a union of discs.
Therefore \(\RP^2\) has a structure as a CW-complex with one-skeleton \(G\).
We look at the induced cellular chain complex \(C_*(\mathbb{R} P^2;\mathbb{Q})\).
  Since \(b_1(\mathbb{R} P^2;\mathbb{Q})=0\) and \(b_1(G;\mathbb{Q})=d\), there are \(d\) boundaries of discs which form a basis of \(H_1(G;\mathbb{Q})\). After removing double points we can assume that these boundaries are simple cycles in \(G\). Therefore the claim follows.
\end{proof}

Lemma \ref{lem:6involutions-BuildingBlocks} already proves Theorem  \ref{thm:6involutions} in the graphic, cographic, and sporadic cases. By Seymour's classification, the only other possibility for $M$ is that it decomposes as a $k$-sum for some $k \in \{1,2,3\}$. Before finishing the proof, we first need a recursive statement for $k$-sums.

\begin{lemma}
\label{lem:6involutions-Decomposables}
Let $M \cong M_1 \oplus_k M_2$ decompose as a $k$-sum with $k \in \{1,2,3\}$, and assume that $k$ is minimal. 
Assume that $M_i$ has rank $d_i$ and is represented over $\Z_2$ in $V_i$, \(\dim V_i=d_i\), and assume there exist $n_i$ codimension-one vector subspaces of $V_i$ such that each element of $E(M_i)$ is contained in at least $n_i - 2$ of them. 

For $k \in \{2,3\}$, $M$ has rank $d = d_1 + d_2 - k + 1$ and is represented over $\Z_2$ in a $d$-dimensional quotient space $V$ of $V_1 \oplus V_2$, and there exist $n_1 + n_2 - k$ pairwise distinct codimension-one vector subspaces of $V$ such that every element of $E(M)$ is contained in all but at most two of them. Similarly, the statement holds for $k=1$ with $n_1 + n_2-k$ replaced by $n_1 + n_2$.
\end{lemma}

Notice that the conclusion for $k \in \{2,3\}$ gives one fewer codimension-one subspace than required to prove our claim, but we will overcome this issue thanks to the extra codimension-one subspace obtained in the graphic case in Lemma \ref{lem:6involutions-BuildingBlocks}.

\begin{proof}[Proof of Lemma \ref{lem:6involutions-Decomposables} for $k \in \{1,2\}$]
Let $M_i$ and $V_i$ be as in the lemma. Choose linear forms $f_i:V_1 \to \Z_2$ for $1 \leq i \leq n_1$ whose kernels are the $n_1$ codimension-one subspaces of $V_1$ as in the assumption, and similarly choose $g_i:V_2 \to \Z_2$ for $1 \leq j \leq n_2$.

If $k = 1$, then $V = V_1 \oplus V_2$, and the result follows easily by using the linear forms $(f_i, 0)$ for $1 \leq i \leq n_1$ and $(0,g_j)$ for $1 \leq j \leq n_2$. Note that every $e \in E(M)$ is of the form $(\tilde e, 0)$ or $(0, \tilde e)$, so $e$ is in the kernel of at least $n_1 + n_2 - 2$ of these linear forms, as required. 

Next, assume $k = 2$, and let $v_1 \in E(M_1)$ and $v_2 \in E(M_2)$ be the elements such that $V$ is the quotient of $V_1 \oplus V_2$ by the subspace spanned by $(v_1, v_2)$. Applying the condition on the $f_i$ and $g_j$, we may relabel these linear forms so that
	\begin{itemize}
	\item $f_i(v_1) = 0$ for $1 \leq i \leq n_1 - 2$.
	\item $g_j(v_1) = 0$ for $1 \leq j \leq n_2 - 2$.
	\end{itemize}

Let $\mathcal F$ be the collection of functionals on $V_1 \oplus V_2$ given by $(f_i,0)$ or $(0,g_j)$ for $1 \leq i \leq n_1 - 2$ or $1 \leq j \leq n_2 - 2$, respectively. Our task is to find two additional functionals such that, when added to $\mathcal F$, we get a collection of $n_1 + n_2 - 2$ functionals on $V_1 \oplus V_2$ that descend to $V$ and have the property that every $e \in E(M)$ lies in the kernel of at least $n_1 + n_2 - 4$ of them. This involves three cases.

\bigskip{\bf Case 1:} $f_{n_1-1}(v_1) = 0$ and $f_{n_1}(v_1) = 0$. We add the functionals $(f_{n_1 - 1}, 0)$ and $(f_{n_1}, 0)$ to $\mathcal F$. These functionals clearly descend to functionals on $V$, since for example $(f_i, 0)$ evaluated on $(v_1,0) - (0,v_2)$ gives $(f_i(v_1), 0) - (0,0) = (0,0)$. In addition, given $e \in E(M)$, we evaluate these functionals on $e$ by lifting to a preimage $\tilde e \in V_1 \oplus V_2$, which is of the form $(\tilde e_1, 0)$ or $(0, \tilde e_2)$ for some $\tilde e_i \in E(M_i) - \{v_i\}$. In the first case, the element is in the kernel of at least $n_1 - 2$ of the $n_1$ functionals of the form $(f_i, 0)$ and is in the kernel of all $n_2 - 2$ functionals of the form $(0, g_i)$. In the second case, the element $(0,\tilde e_2)$ is in the kernel of all $n_1$ functionals of the form $(f_i, 0)$ and at least $n_2 - 4$ of the $n_2 - 2$ functionals of the form $(0,g_j)$. (The worst case here looks like $g_1(\tilde e_2) = 1$ and $g_2(\tilde e_2) = 1$, in which case $\tilde e_2$ is in the kernel of $g_3,\ldots,g_{n_2-2}$.) In either case, it follows that $e$ lies in at least $n_1 + n_2 - 4$ kernels, as required.

\bigskip{\bf Case 2:} $f_{n_1 - 1}(v_1) = 0$ and $f_{n_1}(v_1) = 1$. By the previous case, we may assume that $g_{n_2}(v_2) = 1$. In this case, we add the functionals $(f_{n_1 - 1}, 0)$ and $(f_{n_1}, g_{n_2})$ to $\mathcal F$ and argue similarly. Indeed note, in particular, that the latter functional is well defined since it maps $(v_1,0) - (0,v_2)$ to $f_{n_1}(v_1) - g_{n_2}(v_2) = 1 - 1 = 0$. In addition, it is easy to find $n_1 + n_2 - 4$ kernels containing $e \in E(M)$ if $e$ lifts to an element of the form $(\tilde e_1, 0)$, or if it lifts to an element of the form $(0, \tilde e_2)$ such that $g_{n_2}(\tilde e_2) = 0$. In the remaining case, note that the lift $(0, \tilde e_2)$ with $g_{n_2}(\tilde e_2) = 1$ is in the kernel of all $n_1 - 1$ functionals of the $(f_i,0)$ as well as at least $n_2 - 3$ of the $n_2 - 2$ functionals of the form $(0,g_i)$. 

\bigskip{\bf Case 3:} $f_{n_1 - 1}(v_1) = 1$ and $f_{n_1}(v_1) = 1$. By the previous cases, we may assume $g_{n_2-1}(v_2) = 1$ and $g_{n_2}(v_2) = 1$. Adding the functionals $(f_{n_1 - 1}, g_{n_2 - 1})$ and $(f_{n_1}, g_{n_2})$ to $\mathcal F$ and arguing as in the previous case, the claim follows once again.

After possibly permuting the labels of $f_{n_1 - 1}$ and $f_{n_1}$, one of these cases occurs, so the proof is complete.
\end{proof}

\begin{proof}[Proof of Lemma \ref{lem:6involutions-Decomposables} for $k = 3$]
Let $k = 3$, and let $M_i$ and $W_i \subseteq E(M_i) \subseteq V_i$ with $W_i \cong \Z_2^2 - \{0\}$ be as in the lemma. Choose linear forms $f_i:V_1 \to \Z_2$ for $1 \leq i \leq n_1$ whose kernels are the $n_1$ codimension-one subspaces of $V_1$ as in the assumption, and similarly choose $g_i:V_2 \to \Z_2$ for $1 \leq j \leq n_2$.

Consider the restrictions of the $f_i$ to $W_1$. Since $W_1$ consists of three non-zero vectors that sum to zero, and since $f_i$ vanishes on each $v \in W_1$ for at least $n_1 - 2$ values of $i$, we find that $f_i|W_1 = 0$ for at least $n_1 - 3$ values of $i$. After relabeling and arguing similarly with the $g_j$, we may assume
	\begin{itemize}
	\item $f_4 = \ldots = f_{n_1} = 0$ on $W_1$, and
	\item $g_4 = \ldots = g_{n_2} = 0$ on $W_2$.
	\end{itemize}
Moreover, these arguments imply that either $f_3|_{W_1} = 0$ or that the linear forms $\{f_i|_{W_1}\}_{i=1}^3$ are pairwise distinct and hence equal the set of non-zero elements in the dual space of $W_1$. Similar comments hold for the restrictions of $g_1,g_2,g_3$ to $W_2$, and we use this frequently.

Set $\mathcal F = \{(f_i,0)\}_{i=4}^{n_1} \cup \{(0, g_j)\}_{j=4}^{n_2}$. Our task is to find three additional linear functionals on $V_1 \oplus V_2$ such that, together with those in $\mathcal F$ we get $n_1 + n_2 - 3$ linear functionals on $V_1 \oplus V_2$ that descend to maps on the quotient space $V$ and have the property that every $e \in E(M)$ lies in at least $n_1 + n_2 - 5$ of their kernels. We do this in cases.

\bigskip{\bf Case 1:} $g_3 \neq 0$ on $W_2$, and the non-zero restrictions $f_i|_{W_1}$ are pairwise-distinct. To $\mathcal F$, we add for each $1 \leq i \leq 3$ the functional $(f_i, 0)$ if $f_i|_{W_1} = 0$ or the functional $(f_i, g_{j_i})$ otherwise, where $j_i \in \{1,2,3\}$ is the index such that $f_i|_{W_1} = g_{j_i}|_{W_2}$, and where we have fixed an identification of $W_1 \cong \Z_2^2 - \{0\} \cong W_2$. Note that $j_i$ exists by the rigidity discussed above and that the $j_i$ are pairwise distinct. Notice that these functionals vanish on the diagonal subspace by which we take the quotient to get $V$, so we get well defined functionals on $V$. Note moreover that, for $e \in E(M)$ that lift to an element of the form $(\tilde e, 0)$, $e$ is contained in $n_1 - 2$ of the kernels of the form $(f_i, *)$ and all $n_2 - 3$ kernels of the form $(0,g_j)$, for a total of $n_1 + n_2 - 5$, as desired. Similarly lifts of the form $(0, \tilde e)$ lie in all $(n_1 - 3) + z$ of the kernels of functionals of the form $(f_i, 0)$ and additionally in $n_2 - z - 2$ of the kernels of the other functionals, again for a total of $n_1 + n_2 - 5$.

\bigskip{\bf Case 2:} $g_3 \neq 0$ on $W_2$, and after relabeling $f_1 = f_2$ when restricted to $W_1$. Note in this case that $f_3 = 0$ on \(W_1\). This time, the three functionals we add to $\mathcal F$ are $(f_3, 0)$, $(f_2, g_2 + g_3)$, and $(f_1, g_1)$, where we have relabeled the $g_j$ for $j \in \{1,2,3\}$ so that $f_1|_{W_1} = g_1|_{W_2}$. One can again check that these are well defined and the condition on the number of kernels containing each $e \in E(M)$.

\bigskip{\bf Case 3:} $g_3 = 0$ on \(W_2\). By the previous cases, we may assume that $f_3 = 0$ on \(W_1\). The first two functionals we add to $\mathcal F$ are $(f_3, 0)$ and $(0, g_3)$, and we need to add one more. 
If $f_2 = 0$ on \(W_1\), then we add $(f_2, 0)$. 
If $f_1 = f_2$ on $W_1$ and $g_1 = g_2$ on $W_2$, then we add $(f_1 + f_2, g_1 + g_2)$. 
If neither of these occurs, then we may assume that $g_1$, $g_2$, and $g_1 + g_2$ are non-zero on $W_2$. Since moreover we may assume $f_2 \neq 0$, it follows that $f_2|_{W_1} = g|_{W_2}$ for some $g \in \{g_1, g_2, g_1 + g_2\}$. In this case, we add $(f_2, g)$ to $\mathcal F$. In each case, one can again easily check the conditions of well defined and the number of kernels.
\end{proof}

The proof of Theorem~\ref{thm:6involutions} is now completed by the following.

\begin{corollary}
  Let \(M\) be a regular matroid of rank \(6\) represented in a six-dimensional vector space \(V\) over \(\mathbb{Z}_2\). Then there are six pairwise distinct codimension-one subspaces of \(V\) such that each element of \(E(M)\) is contained in at least four of them.
\end{corollary}

\begin{proof}
  We may assume that \(M\) is simple, i.e. that all subsets of \(E(M)\) of cardinality two or less are independent. 
 Moreover, we may assume that $M$ is neither graphic nor cographic, since otherwise  Lemma \ref{lem:6involutions-BuildingBlocks} implies the claim.
By Seymours theorem $M$ is decomposable as a $k$-sum  $M_1 \oplus_k M_2$ for some $k \in \{1,2,3\}$, and we can assume 
\(
d_i=\rank M_i\ge k\) since the sum is non-trivial. We also assume that the sum $k$ is minimal with the above properties. 
  Since \(k\) is minimal, we have \(6=d_1+d_2-k+1\), and therefore \(d_i\leq 5\) for \(i\in\{1,2\}\). 
  Hence the \(M_i\) are graphic, cographic, or have simplifications isomorphic to the sporadic matroid.

If \(k=1\), then the claim follows directly from Lemmas~\ref{lem:6involutions-BuildingBlocks} and \ref{lem:6involutions-Decomposables}. 
Hence assume \(k=2,3\).

If $M_1$ is graphic, we may choose $n_1 = d_1 + 1$ codimension-one subspaces of $V_1$ and $n_2 = d_2$ codimension-one subspaces for $V_2$, as in Lemma \ref{lem:6involutions-BuildingBlocks}. By Lemma \ref{lem:6involutions-Decomposables}, it follows that the vector space $V$ representing $M$ has $6$ pairwise distinct codimension-one subspaces such that every $e \in E(M)$ is contained in at least $4$ of them. 

We may assume that neither $M_1$ nor $M_2$ is graphic. In particular, both ranks $d_i \geq 4$. Since these ranks satisfy $6 = d = d_1 + d_2 - k + 1$, we find that $d_1 = d_2 = 4$ and $k = 3$. In particular, neither $M_i$ is the rank five sporadic matroid. In fact, all rank four matroids are graphic or cographic, so we find that both $M_1$ and $M_2$ are cographic.
Since \(M_1\) and \(M_2\) are both not graphic, we may assume that \(M_i=M^*[G_i]\), \(i=1,2\), where for \(i=1,2\), \(G_i\) is some subdivision of \(K_{3,3}\).
In this case, there is essentially only one way to perform the \(3\)-sum and $M_1 \oplus_k M_2$ is also cographic (see Remark \ref{rem:sums} c)) and the result follows as above.
\end{proof}

\bigskip
\section{From positive curvature to a rational cohomology CROSS}
\label{sec:t4orLess}

This section reviews and refines the results we need for the proof of Theorem \ref{thm:t9}. Section \ref{sec:Preliminaries} contains the statements of the Connectedness and Periodicity Lemmas of the third author, the Four-Periodicity Theorem of the first author, and the $b_3$ lemma from \cite{KWW}. Together with the $\gS^1$-splitting theorem and a lemma to rule out the case of rational type $\s^p \times \HP^k$ for $p \in \{2,3\}$ in \cite{KWW}, these results imply that a fixed-point component $F^f$ of an isometric $\gT^5$-action is a rational cohomology $\s$, $\CP$, or $\HP$. In fact, Nienhaus \cite{Nienhaus-PhD} proved that this also holds for $\gT^4$-actions. He does this with an improvement of the Four-Periodicity Theorem, and we state these results as well, as they are used throughout the paper.

Section \ref{sec:PushingAndPulling} defines periodicity up to a fixed degree and contains proofs of statements that relate to moving periodicity down to, or up from, submanifolds. These are required for the proof of Theorem \ref{thm:t9}.

\subsection{Connectedness lemma and periodicity}
\label{sec:Preliminaries} 
We start with preliminaries. For simplicity, we state that all cohomology groups are taken with rational coefficients.

\begin{theorem}[Connectedness Lemma, \cite{Wilking03}]\label{thm:CL}
Let $M^n$ be a closed, positively curved Riemannian manifold.
\begin{enumerate}
\item If $N^{n-k} \subseteq M^n$ is a fixed-point component of an isometric circle action on $M^n$, then the inclusion is $(n-2k+2)$-connected.
\item If $N_1^{n-k_1}$ and $N_2^{n-k_2}$ are two such submanifolds, and if $k_1 \leq k_2$, then the inclusion $N_1 \cap N_2 \subseteq N_2$ is $(n-k_1-k_2)$-connected.
\end{enumerate}
\end{theorem}

\begin{theorem}[Periodicity Lemma, \cite{Wilking03}]\label{thm:PL}
If $N^{n-k} \subseteq M^n$ is a $(n - k - l)$-connected inclusion of closed, orientable manifolds, then there exists $e \in H^k(M)$ such that the map $H^i(M) \to H^{i+k}(M)$ given by multiplication by $x$ is surjective for $l \leq i < n - l - k$ and injective for $l < i \leq n - l - k$. 
\end{theorem}

The periodicity in Theorem \ref{thm:PL} may be viewed as starting in degree $l$ and going up to degree $n-l$. 

\begin{remark}\label{rem:GroveSearleCodim2}
When the codimension $k = 2$ in Part (1) of Theorem \ref{thm:CL}, then Theorem \ref{thm:PL} applies with $l = 0$. If additionally $M^n$ is oriented, then it follows from the two-periodic cohomology together with the finiteness of $\pi_1(M)$ that $M$ has the rational cohomology ring of $\s^n$ or $\CP^{\frac n 2}$. Although we do not require it here, Grove and Searle proved the stronger result that $M^n$ is diffeomorphic to $\s^n$, $\CP^{\frac n 2}$, $\RP^n$, or a lens space (see \cite[Theorem 1.2]{GroveSearle94}).
\end{remark}

Throughout this article, we will encounter periodicity starting in degree $0$ and ending at some point. We make the following definition:

\begin{definition}\label{def:UpToc}
A manifold $M$ is $k$-periodic in degrees from $b$ to $c$ if there exists $x \in H^k(M)$ with $2k \leq c - b$ such that the multiplication maps $H^i(M) \to H^{i+k}(M)$ by $x$ are surjective for $b \leq i < c - k$ and injective for $b < i \leq c - k$.

If $b = 0$, we say $M$ has $k$-periodic cohomology up to degree $c$, and if additionally $c = \dim M$, then we say that $M$ has $k$-periodic cohomology.

In the special case where $k \in \{1,2,4\}$, $H^1(M) = 0$, and $H^i(M) = 0$ for $0 < i < k$, we say that $M$ is a cohomology $\s$, $\CP$, or $\HP$ from degree $b$ to degree $c$.
\end{definition}

We note that $M$ is a rational sphere up to degree $c$ if and only if $M$ is rationally $(c-1)$-connected.

The passage from positive curvature to periodicity is via the Connectedness and Periodicity Lemmas. The next step from periodicity to rational cohomology spheres and projective spaces requires multiple additional tools. The first, in the simply connected case, is the Four-Periodicity Theorem of the first author, together with an improvement that removes the simply connected assumption and relaxes the codimension assumption in the case where the normal bundle to the submanifold admits a complex structure due to Nienhaus.

\begin{theorem}[Four-Periodicity Theorem, \cite{Kennard13,Nienhaus-PhD}]\label{thm:4PT}
Let $N^{n-k} \subseteq M^n$ be a $(n-k)$-connected inclusion of closed, orientable manifolds.
	\begin{enumerate}
	\item If $k \leq \tfrac n 3$, then $M$ has four-periodic rational cohomology.
	\item If $k \leq \tfrac n 2$ and if the normal bundle to $N$ admits a complex structure, then $M$ has four-periodic rational cohomology.
	\end{enumerate}
\end{theorem}

Note that one gets such highly connected inclusions from the Connectedness and Periodicity Lemmas in the presence of {\it transversely intersecting} submanifolds of {\it small codimension}. If one can arrange these submanifolds to be fixed-point components of circles, then one can moreover get by with the weaker assumption in (2).

For dimensions $n \not\equiv 2\bmod 4$, four-periodicity together with Poincar\'e duality and the assumption $H^1(M) = 0$ is sufficient to conclude that $M$ is a rational cohomology $\s^n$, $\CP^{\frac n 2}$, $\HP^{\frac n 4}$, or $\s^3 \times \HP^{\frac{n-3}{4}}$.  In the dimensions $n \equiv 2\bmod 4$, one has the model spaces $\s^n$, $\CP^{\frac n 2}$, and $\s^2 \times \HP^{\frac{n-2}{4}}$, as well as manifolds with non-vanishing third Betti number. These examples are ruled out by the $b_3$ lemma in \cite{KWW}. As for the remaining exotic examples of $\s^p \times \HP^{\frac{n-p}{4}}$ with $p \in \{2,3\}$, these are ruled out in \cite{KWW} for fixed-point components of isometric $\gT^5$-actions. Nienhaus proved the same conclusion but using only a $\gT^4$-action (see \cite{Nienhaus-PhD}):

\begin{theorem}\label{thm:t4}
If $F^f$ is a fixed-point component of an effective, isometric $\gT^4$-action on a connected, closed, orientable, positively curved Riemannian manifold $M^n$, then $F^f$ is a rational cohomology $\s^f$, $\CP^{\frac f 2}$, or $\HP^{\frac f 4}$.
\end{theorem}

In the remainder of this section, we provide useful tools for pushing down periodicity to, and pulling it up from, submanifolds.

\subsection{Pushing and pulling periodicity}
\label{sec:PushingAndPulling}

We start with pushing down, which follows for purely topological reasons. 

\begin{theorem}[Bredon]\label{thm:Bredon}
If $M$ is a smooth, closed, orientable manifold with singly generated, $g$-periodic rational cohomology, then for any fixed-point component $F$ of a smooth circle action on $M$, $F$ also has singly generated, $g'$-periodic cohomology for some $g' \leq g$ with $g' \equiv g \bmod 2$. 
\end{theorem}

For example, the fixed-point components of a circle action on a rational cohomology $\HP^n$ are rational cohomology $\CP^f$ and $\HP^f$ for some $f \leq n$. 

Note, in addition, that positive curvature might imply that the inclusion $F \subseteq M$ is highly connected by the Connectedness Lemma. If, in particular, the connectedness is at least $g$, then $H^i(N;\Q) \cong H^i(M;\Q) = 0$ for $0 < i < g$, and so $g' = g$ in Bredon's theorem.

Now we move to the issue of pulling up periodicity. In applications, Theorem \ref{thm:t4} will give rise to a submanifold that is a rational cohomology $\s$, $\CP$, or $\HP$. The Connectedness Lemma will also give us highly connected inclusions of $F$ and possibly other submanifolds into $M$. The following lemmas are used frequently. 

\begin{lemma}\label{lem:uptoc}
Let $N^m \subseteq M^n$ be a $c$-connected inclusion of closed, orientable manifolds for some integer $c$. If $0 < k < c$, then $N$ has $k$-periodic cohomology up to $c$ if and only if $M$ has $k$-periodic rational cohomology up to $c$.
\end{lemma}

\begin{proof}
We consider the restriction maps $H^i(M) \to H^i(N)$, which are isomorphisms for $i < c$ and which is an injective for $i = c$. Let $\tilde x \in H^k(M)$ be a generator that maps to a generator $x \in H^k(N)$. By the naturality of cup products, the composition of the map $H^{i-k}(M) \to H^i(M)$ given by multiplication by $\tilde x$ followed by restriction is the same as restriction followed by the map $H^{i-k}(N) \to H^i(N)$ given by multiplication by $x$. For $i \leq c$, the restriction map is an isomorphism in degree $i-k$ and at least an injection in degree $i$, so multiplication by $\tilde x$ is injective if and only if multiplication by $x$ is injective. Similar remarks hold for the surjectivity statements, so the conclusion follows.
\end{proof}

By Poincar\'e duality, one can recover the whole cohomology ring from roughly the first half. The precise statements we use in this paper are as follows.

\begin{lemma}\label{lem:UpToHalfPD}
Assume $M^n$ is a closed, oriented manifold.
	\begin{enumerate}
	\item If $M^n$ is a rational cohomology $\s$, $\CP$, or $\HP$ up to $\ceil{\tfrac{n+g}{2}}$, where $g$ is $1$, $2$, or $4$, respectively, then $g$ divides $n$ and $M$ is a rational $\s^n$, $\CP^{\frac n 2}$, or $\HP^{\frac n 4}$.
	\item If $N$ is a rational $\s$, $\CP$, or $\HP$ and has dimension $\dim N \geq \tfrac{n+g}{2}$, where $g$ is $1$, $2$, or $4$, respectively, and if there is a $\ceil{\tfrac{n-1}{2}}$-connected inclusion $N \subseteq M^n$, then $M$ is a rational $\s$, $\CP$, or $\HP$.
	\end{enumerate}
\end{lemma}

We remark that (2) is used in the proof of \cite[Lemma 6.2]{Wilking03}, but we prove it here for completeness. 

We also note that the assumption in (2) implies that $M$ is a rational $\s$, $\CP$, or $\HP$ up to degree $\ceil{\frac{n-1}{2}}$ by Lemma \ref{lem:uptoc}, so (2) is an improvement of (1) in the situation where periodicity up to $c$ comes from a $c$-connected inclusion of a submanifold. We also note that the dimension assumption on $N$ in (2) is required since, for example, there is an $8$-connected inclusion $\s^8 \to \CaP$.

We also remark that there is no similar improvement of (1) for $\ceil{\tfrac{m-1}{2}}$-connected inclusions $N^m \subseteq M$ where this time $M$ is a rational $\s$, $\CP$, or $\HP$. Indeed, one can homotope the composition of the projection $\CP^m \times \s^{2m} \to \CP^m$ and the inclusion $\CP^m \subseteq \CP^l$ for all sufficiently large $l$ to obtain a $(2m)$-connected inclusion of the $(4m)$-manifold $\CP^m \times \s^{2m}$ into $\CP^l$ for all sufficiently large $l$. Similar examples exist where the ambient space is $\s^l$ or $\HP^l$.

\begin{proof}
The proof of (1) uses Poincar\'e duality. In the case of the sphere, the assumption implies that $H^i(M^n) = 0$ for $i \leq \ceil{\tfrac{n+1}{2}} - 1 = \ceil{\tfrac{n-1}{2}}$. Hence all Betti numbers vanish, and $M$ is a rational sphere.

For the case of $\CP$, we note that $n$ is even since otherwise $2$-periodicity up to degree $\frac{n+3}{2}$ and Poincar\'e duality implies
	\[H^{\frac{n-3}{2}}(M) \cong H^{\frac{n+1}{2}}(M) \cong H^{\frac{n-1}{2}}(M).\]
The degree of the first and last of these groups differ by one, so we have a contradiction to the assumption on the topology of $M$. Assuming then that $n$ is even, the periodicity goes up to degree $\tfrac n 2 + 1$ and hence the powers of a generator $x \in H^2(M)$ are non-zero in degree $\tfrac n 2$ if $n \equiv 0 \bmod 4$ and in degree $\tfrac{n+2}{2}$ if $n \equiv 2 \bmod 4$. In either case, this is sufficient to conclude that the top cohomology group is generated by $x^{\frac n 2}$ by Poincar\'e duality. We also have that the Betti numbers match those of $\CP$ by assumption and by Poincar\'e duality, so $M$ is a rational cohomology $\CP$ in this case.

The case of $\HP$ is similar to the $\CP$ case. If $n \equiv 1,7 \bmod 8$, then the  isomorphisms
	\[H^{\frac{n+1}{2} - 4}(M) \cong H^{\frac{n+1}{2}}(M) \cong H^{\frac{n-1}{2}}(M)\]
lead to a contradiction to the assumed topology of $M$. Similar contradictions are reached if $n \equiv 2 \bmod 4$ using the isomorphisms
	\[H^{\frac{n+2}{2} - 4}(M) \cong H^{\frac{n+2}{2}}(M) \cong H^{\frac{n-2}{2}}(M)\]
and if $n \equiv 3,5 \bmod 8$ using the isomorphisms
	\[H^{\frac{n+3}{2} - 4}(M) \cong H^{\frac{n+3}{2}}(M) \cong H^{\frac{n-3}{2}}(M)\]
For $n \equiv 0 \bmod 4$, the argument proceeds as in the $\CP$ case, noting that the periodicity up to $c \geq \frac{n+4}{2}$ is crucial when $n \equiv 4 \bmod 8$ to see that a power of a generator of $H^4(M)$ generates $H^{\frac{n+4}{2}}(M)$.

To prove (2), fix an element $x \in H^g(N)$ that induces $g$-periodicity. Let $\tilde x \in H^g(M)$ denote its preimage under the map induced by inclusion. Consider the map
	\[H^{j - g}(M) \to H^j(M)\]
induced by multiplication by $x$. For $g < j \leq \ceil{\tfrac{n+g}{2}}$, this map is injective since the corresponding map
	\[H^{j-g}(N) \to H^j(N)\]
given by multiplication by $x$ is injective in these degrees by the assumption on $N$ and since the map $H^{j-g}(M) \to H^{j-g}(N)$ induced by inclusion is an isomorphism in these degrees. 

We next claim that this map is surjective in degrees $g \leq j < \ceil{\tfrac{n+g}{2}}$. Given this claim, we have that $M$ is a rational cohomology $\s$, $\CP$, or $\HP$ in degree up to $\ceil{\frac{n+g}{2}}$, so (1) implies that $M$ is a rational $\s$, $\CP$, or $\HP$, which proves (2).

To prove the surjectivity claim, we first note that an argument similar to the injectivity proof carries through for $j \leq \ceil{\tfrac{n-1}{2}}$, since in these degrees the map $H^j(M) \to H^j(N)$ is injective. We may therefore assume that $\ceil{\tfrac{n-1}{2}} < j < \ceil{\tfrac{n+g}{2}}$. Hence that we are in one of the following cases:
	\begin{enumerate}
	\item $n \equiv 1 \bmod 2$ and $g \in \{2,4\}$,
	\item $n \equiv 2 \bmod 4$ and $g = 4$, or
	\item $n \equiv 0 \bmod 4$, $g = 4$, and $j = \tfrac{n+2}{2}$.
	\end{enumerate}

The first two cases do not actually occur. To see this, set $m = \ceil{\tfrac{n+1}{2g}}$ and note one on hand that $H^{gm}(M) \neq 0$ since the element $\tilde x^m \in H^{gm}(M)$ restricts to the generator $x^m$ of $H^{gm}(N)$, which is non-zero by the assumptions on the dimension and the rational cohomology of $N$. On the other hand, we have 
	\[H^{gm}(M) \cong H^{n - gm}(M) \cong H^{n - gm}(N) = 0\]
by Poincar\'e duality, the fact that $N\subseteq M$ is $\ceil{\frac{n-1}{2}}$-connected, and the conditions on $n$ and $g$, which imply that $n \not\equiv 0 \bmod g$.

In the third case, the map $H^{j-g}(M) \to H^j(M)$ is trivially surjective if we can show that the latter group is zero, and indeed this is the case since
	\[H^j(M) \cong H^{n-j}(M) = H^{n-j}(N) = 0,\]
where in the last step we used the fact that $n - j = \tfrac{n-2}{2}$ is odd. This completes the proof of the surjectivity claim and hence of (2) as explained above.
\end{proof}

\begin{lemma}\label{lem:k1plus2k2}
Let $M^n$ be a closed, orientable, positively curved Riemannian manifold. Suppose that $N^{n-k}$ is a fixed-point component of an isometric circle action on $M$. Assume one of the following:
	\begin{enumerate}
	\item $M$ is a rational $\s$ up to degree $k + 1$.
	\item $M$ is a rational $\s^3 \times \HP$ up to degree $k + 1$, $H^{k+1}(M) = 0$, and $k \equiv 0 \bmod 4$.
	\item $M$ is a rational $\CP$, $\HP$, or $\s^p \times \HP$ for some $p \in \{2,3\}$ up to degree $k + 2$.
	\end{enumerate}
Then $M$ is a rational $\s$, $\CP$, $\HP$, $\s^2 \times \HP$, or $\s^3 \times \HP$ up to degree $n - 2k + 3$. If additionally $k \leq \tfrac{n+5}{4}$, then $M^n$ is a rational cohomology $\s^n$, $\CP^{\frac n 2}$, $\HP^{\frac n 4}$, or $\s^p \times \HP^{\frac{n-p}{4}}$ for some $p \in \{2,3\}$.
\end{lemma}

Note that, for $k = 2$, the stronger conclusion that $M$ is diffeomorphic to $\s^n$ or $\CP^{\frac n 2}$ holds by a theorem of Grove and Searle (see Remark \ref{rem:GroveSearleCodim2}). In any case, it follows that the rational cohomology type in the assumption matches that in the conclusion.

\begin{proof}
Fix $e \in H^{k}(M)$ as in the Periodicity Lemma, so that $e$ induces periodicity in degrees from $k-2$ to $n-(k-2)$.

First, in the case where $M$ is a rational $\s$ up to degree $k+1$, we have $H^i(M) = 0$ for $1 \leq i \leq k$. In particular, $e = 0$ and hence the injectivity of the multiplication map $H^i(M) \to H^{i+k}(M)$ for $i \leq n - 2k + 2$ implies that $H^i(M) = 0$ for all $0 < i < n - 2k+3$. This is equivalent to the conclusion that $M$ is a rational sphere up to degree $n - 2k + 3$.

Second, assume we we are in the $\CP$, $\HP$, or $\s^p \times \HP$ for some $p \in \{2,3\}$, and let $x \in H^g(M)$ be a generator, where $g = 2$ in the $\CP$ case and $g = 4$ in the other cases. Note that $k \geq g$, since otherwise $k = 2$ and hence that $M$ is a $\CP$ in all degrees by Grove and Searle's theorem (see Remark \ref{rem:GroveSearleCodim2}). Since multiplication by $x$ is surjective into degree $k$, we obtain a factorization of the form $e = xy$ for some $y \in H^{k-g}(M)$. 

Consider the maps
	\[H^i(M) \stackrel{x}{\longrightarrow} H^{i+g}(M) \stackrel{y}{\longrightarrow} H^{i+k}(M) \stackrel{x}{\longrightarrow} H^{i+k+g}(M)\]
given by multiplication by $x$, $y$, and $x$, respectively. The composition of the first two maps equals $e$, and likewise for the composition of the last two. In particular, for 
	\[k - 2 \leq i \leq (n - 2k + 2) - g,\] 
the composition of the first two maps is surjective and the composition of the last two is injective. Putting these claims together, we find that the middle map is an isomorphism and hence that the first map is surjective in this range. Similarly, on the indices 
	\[k - 2 < i \leq (n - 2k + 2),\]
we obtain injectivity for the composition of the first two maps and hence for the first map. In particular, the injectivity holds for $i \leq (n - 2k + 3) - g$, so combining with the surjectivity statement shows that $M$ is a rational $\CP$, $\HP$, or $\s^2 \times \HP$ up to degree $n - 2k + 3$.

Third, assume we are in the $\s^3 \times \HP$ case up to degree $k+1$, and let $x \in H^4(M)$ be a generator. The proof for $g = 4$ above almost shows that multiplication by $x$ induces periodicity. We only need show in addition that the maps
	\[H^{k-3}(M) \to H^{k+1}(M) \hspace{.3in} \mathrm{and}\hspace{.3in} H^{k-2}(M) \to H^{k+2}(M)\]
are surjective and injective, respectively. The first of these claims follows trivially by the assumption $H^{k+1}(M) = 0$, and the second follows trivially by the assumption that $k \equiv 0 \bmod 4$, which implies that $H^{k-2}(M) = 0$.

Finally, we prove the last claim under the assumption $k \leq \tfrac{n+5}{4}$. Since $n - 2k + 3 \geq \tfrac{n+1}{2}$, the Betti numbers agree with the appropriate model in degrees up to $\floor{\tfrac n 2}$ and hence in all degrees by Poincar\'e duality. In particular, that we are already done in the sphere case.

For the remaining cases, let $g \in \{2,4\}$ as above. We have shown that the maps $H^i(M) \to H^{i+g}(M)$ induced by multiplication by $x$ are injective for all 
	\[0 < i \leq n - 2k + 2.\]
We claim next that these maps are also injective for all
	\[n - 2k + 3 \leq i \leq n - g.\]
If we can do this, then the injectivity part of $g$-periodicity would be established, and the surjectivity part would follow since we already know the Betti numbers.

Fix an $i$ in this range, and suppose $y \in H^i(M)$ is non-zero. Using Poincar\'e duality, choose $z \in H^{n-i}(M)$ such that $yz \neq 0$. Note that the degree of $z$ satisfies
	\[n - i \leq 2k - 3 \leq n - 2k + 2\]
by the assumed upper bound on $k$, so periodicity up to degree $n - 2k + 3$ implies that $z = x z'$ for some $z' \in H^{n - i - g}$. But now $0 \neq yz = (xy)z'$, so $xy \neq 0$. This completes the proof of the injectivity claim and hence the proof of (2).
\end{proof}

\bigskip
\section{Proof of Theorem \ref{thm:t9}}
\label{sec:t9}

In this section, we prove Theorem \ref{thm:t9}. We are given a closed, oriented, positively curved Riemannian manifold $M^n$, and we assume $\gT^d$ acts effectively by isometries on $M^n$ with a fixed point and the property that all isotropy groups in a neighborhood of the fixed point have an odd number of components. The theorem is implied by the following claims:
	\begin{enumerate}
	\item If $d = 9$, then $M$ is a rational cohomology \(\s^n,\CP^{\frac n 2}\), or \(\HP^{\frac n 4}\).
	\item If $d = 6$, then $M$ is a rational cohomology \(\s^n,\CP^{\frac n 2}\), or \(\HP^{\frac n 4}\) up to degree $n - 4\floor{\tfrac n 6} + 2$. 
	\end{enumerate}

We recall that $M$ has $g$-periodic rational cohomology up to degree $c$ if there is an element $x \in H^g(M;\Q)$ such that the map $H^i(M;\Q) \to H^{i+g}(M;\Q)$ induced by multiplication by $x$ is surjective for $0 \leq i < c - g$ and injective for $0 < i \leq c - g$. The conclusions are equivalent to showing, for some $g \in \{1,2,4\}$, that $H^i(M;\Q) = 0$ for $0 < i < g$ and that $H^*(M;\Q)$ is $g$-periodic up to degree $n$ or $n - 4\floor{\tfrac n 6} + 2$.

\subsection{Proof of (1) given (2)}

Considering the subaction by any $\gT^6 \subseteq \gT^9$, we get from Conclusion (2) that $M^n$ is a rational cohomology $\s^n$, $\CP^{\frac n 2}$, or $\HP^{\frac n 4}$ up to degree $n - 4\floor{\frac n 6} + 2$ and hence up to degree $\floor{\tfrac n 4} + 2$. At the same time, Theorem \ref{thm:cd} implies the existence of a fixed-point component $N \subseteq M$ of a circle in $\gT^9$ with codimension $\cod N \leq \floor{\tfrac n 4}$. By Lemma \ref{lem:k1plus2k2}, it follows that $M$ is a rational cohomology $\s^n$, $\CP^{\frac n 2}$, or $\HP^{\frac n 4}$ in all degrees.

\subsection{Proof of (2)}

We have a $\gT^6$-action with a fixed point $x \in M^n$ and the property that all isotropy groups have an odd number of components. Theorem \ref{thm:6involutions} implies that there exist non-trivial, pairwise distinct involutions $\iota_1,\ldots,\iota_6 \in \gT^6$ whose fixed-point components $N_i^{n-k_i} = M^{\iota_i}_x$ satisfy
	\[k_1 + \ldots + k_6 \leq 2n.\]
Without loss of generality, we may assume $k_1 \leq \ldots \leq k_6$. We begin the proof with two reductions.

\bigskip
\noindent
{\bf Claim 1:} The theorem holds if $N_1$ is a rational $\s$, $\CP$, or $\HP$.

\begin{proof}[Proof of Claim 1]
The inclusion $N_1 \subseteq M$ is $(n - 2k_1 + 2)$-connected by the Connectedness Lemma, so $M$ is a rational $\s$, $\CP$, or $\HP$ up to degree $n - 2k_1 + 2$. Since $k_1$ is even and satisfies $k_1 \leq \tfrac n 3$, we find that $M$ is a rational $\s$, $\CP$, or $\HP$ up to degree $n - 4 \floor{\frac n 6} + 2$.

It suffices to show $n \equiv 0 \bmod 2$ in the $\CP$ case and $n\equiv 0 \bmod 4$ in the $\HP$ case. First note that $n$ is odd only if $\dim N_1$ is odd, which holds only if $N_1$ is a rational sphere, so we only need to rule out the possibility that $n = 4m+2$ and $N_1$ is a rational $\HP$. If this were the case, then a generator $x \in H^4(M)$ would satisfy $x^{\frac{k_1+2}{4}} \neq 0$ since this element would restrict non-trivially in $H^{k_1+2}(N_1) \cong \Q$. On the other hand, the Periodicity Lemma gives rise to an injection $H^{k_1+2}(M) \to H^{2k_1 + 2}(M)$ given by multiplication by some $e \in H^{k_1}(M)$, which would be isomorphic to $H^{k_1}(N_1) = 0$, a contradiction.
\end{proof}

\bigskip
\noindent
{\bf Claim 2:} The theorem holds if $N_i \cap N_j$ is a rational $\s$, $\CP$, or $\HP$ for some $1 \leq i < j \leq 3$.

\begin{proof}[Proof of Claim 2]
First suppose $k_1 < \tfrac n 3$. The inclusion $N_i \cap N_j \subseteq N_j$ is $c$-connected with
	\[c \geq n - k_i - k_j \geq k_1 + 1,\]
where the first inequality comes from the Connectedness Lemma and the second uses the estimates $k_1 + \ldots + k_6 \leq 2n$ and $k_1 < \tfrac n 3$. Note also that $n - k_i - k_j$ is even if $n$ is even, so $c \geq k_1 + 2$ in this case. Lemma \ref{lem:uptoc} implies that $N_j$ is a rational $\s$, $\CP$ or $\HP$ up to $c$. Similarly since $N_j \subseteq M$ is $c'$-connected with
	\[c' \geq n - 2k_j + 2 \geq k_1 + 2,\]
$M$ too is a rational $\s$, $\CP$, or $\HP$ up to degree $k_1+1$ if $n$ is odd dimensions and $k_1+2$ if $n$ is even. Lemma \ref{lem:k1plus2k2} now implies that $M$ is a rational $\s$, $\CP$, or $\HP$ up to degree $n - 2k_1 + 3$. Finally, since $N_1 \subseteq M$ is $(n - 2k_1 + 2)$-connected, Lemma \ref{lem:uptoc} implies that $N_1$ is a rational $\s$, $\CP$, or $\HP$ up to degree $n - 2k_1 + 2$ and hence up to $\tfrac 1 2 \dim(N_1) + 2$. By Poincar\'e duality, $N_1$ is a rational $\s$, $\CP$, or $\HP$ (see Lemma \ref{lem:UpToHalfPD}), so the theorem follows by Claim 1.

Now suppose that $k_1 = \tfrac n 3$. Since the $k_i$ are increasing and sum to at most $2n$, we have $k_i = \tfrac n 3$ for all $1 \leq i \leq 6$. In particular, $n = 6m$ and $\dim N_j = 4m$ for some integer $m$ since the codimensions $k_i$ are even. In addition, if we can show that $N_j$ is a rational $\s$, $\CP$, or $\HP$, then we may swap the roles of $\iota_1$ and $\iota_j$ if $j > 1$ and conclude the theorem by Claim 1.

If $N_i$ and $N_j$ intersect transversely, then the Four-Periodicity Theorem implies that $N_j$ has four-periodic rational cohomology. Since $\dim N_j = 4m$, Poincar\'e duality implies that $N_j$ is a rational $\s$, $\CP$, or $\HP$, as needed.

If instead $N_i$ and $N_j$ do not intersect transversely, then the dimension of the intersection satisfies
	\[\dim(N_i \cap N_j) \geq (n - k_i - k_j) + 2 = 2m +2,\]
and the inclusion of the intersection into $N_j$ is $(2m)$-connected. Since $2m = \tfrac 1 2 \dim(N_j)$, Lemma \ref{lem:UpToHalfPD} implies that $N_j$ is a rational $\s$, $\CP$, or $\HP$, so the proof is complete.
\end{proof}

By Claim 2, the second part of the next lemma finishes the proof of the theorem in the case where the three involutions in $\gT^6$ whose fixed-point components have minimal codimension are linearly independent. The first part of the lemma is used to prove the second part, and it will be used later in the linearly dependent case.

\begin{lemma}\label{lem:LinearlyIndependent}
Let $\sigma_1,\sigma_2,\sigma_3 \in \gT^6$ be linearly independent involutions, where $\gT^6$ acts isometrically without finite isotropy groups of even order near a fixed point $x \in M$. Assume that the fixed-point components $P_i^{n - l_i} = M^{\sigma_i}_x$ satisfy $l_1 + l_2 + l_3 \leq n$. 
\begin{enumerate}
\item The intersection $P_1 \cap P_2 \cap P_3$ is a rational cohomology $\s^f$, $\CP^{\frac f 2}$, or $\HP^{\frac f 4}$.
\item If $\sigma_i = \iota_i$ for $i \in \{1,2,3\}$, where the $\iota_i$ are chosen as above, then $P_i \cap P_j = N_i \cap N_j$ is a rational cohomology $\s$, $\CP$, or $\HP$ for some $1 \leq i < j \leq 3$.
\end{enumerate}
\end{lemma}

\begin{proof}
To prove this lemma, set
	\[F^f = P_1 \cap P_2 \cap P_3,\]
and let $\gT^d$ denote the identity component 
of the kernel of the $\gT^6$-action on $F$. By assumption $\sigma_i\in \gT^d$ and thus $d\ge 3$.
Moreover, \(\gT^d\) acts without finite isotropy groups of even order near \(F\).

We consider the restriction 
\[\rho:\mathbb{Z}_2^3 \to \gO(\nu_x F)\]
of the isotropy representation to the subgroup of involutions $\Z_2^3 = \langle \sigma_1,\sigma_2,\sigma_3\rangle$. If \(d=3\), then not all seven non-trivial irreducible representations of $\Z_2^3$ occur as subrepresentations of $\rho$.
If \(d>3\), then the first part of the lemma holds by Theorem~\ref{thm:t4}.
Our first claim is that if  the conditions of (2) hold and \(d>3\), then also not all seven non-trivial irreducible representations of \(\Z_2^3\) occur as subrepresentations of \(\rho\).

First, we prove the claim by contradiction for $d = 4$. If, for some $i < j$, the intersection $P_i \cap P_j$ were fixed by a three-torus, then Lemma~\ref{lem:6involutions-BuildingBlocks} would imply the existence of two involutions for which the codimensions of the fixed-point sets add 
up to at most $n-\dim(P_i\cap P_j)$.  By assumption the same must hold for  $\sigma_i$ and $\sigma_j$. This is only possible if $P_i$ and $P_j$ intersect transversely, but then not all seven irreducible representations of $\Z_2^3$ can show up. If, instead, the intersection $P_i \cap P_j$ is not fixed by a three-torus for all $i < j$, then the number of non-trivial subrepresentations of $\Z_2^d\subset \gT^d$ in $P_i\cap P_j$ is at least two. In particular, the isotropy representation of $\gT^d$ has at least $7 + 3$ irreducible subrepresentations.  This is the maximum number possible. If we let $W\subset (\Z_2^4)^*$ denote the set of weights of these subrepresentations 
then the complement of $W\cup \{0\}$ in $(\Z_2^{4})^*$ is given by 
five vectors in general linear position. 
The permutation group $\gS_5$ on these five vectors 
induce linear isomorphisms of $(\Z_2^4)^*$ leaving $W$ invariant.
One of these vectors must be the element $\tau$ whose 
kernel is given by $\Z_2^3\subset \gT^3$.
Up to a an isomorphism the situation is now completely determined.
We can  choose a basis of $\Z_2^4$ such that 
 $W$ is given by the elements of weights $\le 2$ and 
 $\tau$ is given by the element of weight $4$.
 The weights of the representation of $\Z_2^3$ are obtained by projecting 
 $\pr\colon (\Z_2^4)^*\to (\Z_2^4)^*/\tau$.
 It is now easy to see that there are indeed three elements 
 in $\pr(W)$ that have two preimages. However these three elements 
 are linearly dependent, so we again have a contradiction. This proves the claim for $d = 4$.  	

Finally, we prove the claim for $d \geq 5$. By a similar argument, we find that the number of non-trivial irreducible subrepresentations of $\Z_2^d\subset \gT^d$ in $P_i\cap P_j$ is at least three for all $i < j$. In particular, the isotropy representation of $\gT^d$ has at least $7 + 2(3) = 13$ irreducible subrepresentations. Using our classification from Section~\ref{sec:classification}, this implies either  that $d = 6$ or that $d = 5$ and the representation of $\gT^5$ is graphic. Applying Theorem~\ref{thm:6involutions} and Lemma~\ref{lem:6involutions-BuildingBlocks},  we can find three involutions $\bar{\sigma}_1,\bar{\sigma}_2,\bar{\sigma}_3\in \gT^d$ such that the codimensions of the fixed-point set add up to at most $n-f$. By the choice of $\sigma_i$ in (2) the same holds for $\sigma_i$. Since the intersection of $P_1, P_2,P_3$ has codimension $n-f$, equality must hold and the $P_i$ and $P_j$ intersect pairwise perpendicularly a contradiction to the fact that all $7$ non-trivial irreducible representations of $\Z_2^3$ show up.

With the claim proven, we come to the main part of the proof. We label the multiplicities of the weights $e_i$ by $m_i$ for $1 \leq i \leq 3$, the weights $e_i + e_j$ by $m_{ij}$ for $1 \leq i < j \leq 3$, and the weight $e_1 + e_2 + e_3$ by $m_{123}$. By the above claim, we may assume that at least one of these multiplicities is zero. Note also that the multiplicities are even since \(\rho\) is the restriction of a torus representation. We summarize these definitions as follows: 

\[
	\begin{array}{c|ccccccccc}
	~  & m_1 & m_2 & m_3	&~ &m_{23}&m_{13}&m_{12}&~&m_{123}\\\hline
	\sigma_1 & 1 & 0 & 0		&~	& 0 & 1 & 1	&~& 1 \\
	\sigma_2 & 0 & 1 & 0		&~	& 1 & 0 & 1	&~& 1 \\
	\sigma_3 & 0 & 0 & 1		&~	& 1 & 1 & 0	&~& 1 \\
	\end{array}
\]
Note that the condition $l_1 + l_2 + l_3 \leq n$ implies the following estimate on the multiplicities of the weights:
	\[\sum_i m_i + 2 \sum_{i<j} m_{ij} + 3 m_{123} \leq n.\]
We first prove the lemma in two special cases.

\bigskip
\noindent
{\bf Case 1:} $\rho$ has only three irreducible subrepresentations. 

\medskip
\noindent
The weights of \(\rho\) form a basis of \((\mathbb{Z}_2^3)^*\), so we can replace $\sigma_1$, $\sigma_2$, and $\sigma_3$ by a dual basis, if necessary, so that $m_{ij} = 0$ for all $i < j$ and $m_{123} = 0$. In particular, all $m_i > 0$. 

We look at the transverse intersection of $P_1$ and $P_2$, which satisfies $\dim(P_1 \cap P_2) \geq \frac{1}{2} \dim P_2$. Additionally, $P_1$ and $P_2$ are fixed-point components of circles by the assumption on the number of components of isotropy groups. Hence the Connectedness-Lemma \ref{thm:CL} and the Four-Periodicity-Theorem \ref{thm:4PT} imply that \(P_1\cap P_2\) has four-periodic rational cohomology, and Lemma \ref{lem:NoSxHP} implies that $P_1 \cap P_2$ is a rational cohomology $\s$, $\CP$, or $\HP$. In each of these cases, Bredon's theorem implies that $F^f$ is also a rational $\s$, $\CP$, or $\HP$, so both parts of the lemma follow in this case. 

\bigskip

\noindent{\bf Case 2:} There exists $h \in \{1,2,3\}$ such that exactly two of the three multiplicities $m_h$, $m_{ij}$, and $m_{123}$ are non-zero, where $\{h,i,j\} = \{1,2,3\}$.

\medskip
\noindent
To prove the lemma in Case 2, we look at the fixed-point component $P_{ij} = M^{\sigma_i\sigma_j}_x$ of the product of the involutions $\sigma_i$ and $\sigma_j$. We also consider $P_{ij} \cap P_h$, $P_{ij} \cap P_i$, and $P_{ij} \cap P_{hi}$, which as submanifolds of $P_{ij}$ have codimensions $m_h + m_{123}$, $m_{ij} + m_{123}$, and $m_h + m_{ij}$, respectively. By the assumption in Case 2, some pair of these submanifolds of $P_{ij}$  intersect transversely and have positive codimension in $P_{ij}$. Moreover, the intersection equals $F^f = M^{\gT^3}_x$, $F^f$ arises as the fixed-point component of a circle by the assumption on the number of components of isotropy groups, and the smaller of the two codimensions is at most $\max(m_{ij}, m_{123})$, which satisfies
	\[\max(m_{ij}, m_{123}) \leq \sum m_{rs} + 2m_{123} \leq n - \sum m_r - \sum m_{rs} - m_{123} = f.\]
In particular, $F^f = M^{\gT^3}_x$ has four-periodic rational cohomology by the Four-Periodicity Theorem (Theorem \ref{thm:4PT}). By Lemma \ref{lem:NoSxHP} below, \(F^f\) is a rational cohomology $\s^f$, $\CP^{\frac f 2}$, or $\HP^{\frac f 4}$, which proves the first part of the lemma in Case 2.

We proceed to the proof of the second part of the lemma under the assumptions of Case 2. First, if $m_h = 0$, then $M^{\gT^3}_x = P_i \cap P_j$ and we are already done.

Second, assume that $m_{123} = 0$. 
If $m_h \leq m_{ij}$, then the inclusion $F^f \subseteq P_i \cap P_j$ is $f$-connected with $f \geq \tfrac 1 2 \dim(P_i \cap P_j)$, so the Four-Periodicity Theorem and Lemma \ref{lem:NoSxHP} imply that $P_i \cap P_j$ is a rational $\s$, $\CP$, or $\HP$. If instead $m_{ij} < m_h$, then we similarly have that $F^f \subset P_{ij} \cap P_h$ is $f$-connected with $f \geq \tfrac 1 2 \dim(P_{ij} \cap P_h)$ and hence that $P_{ij} \cap P_h$ is a rational $\s$, $\CP$, or $\HP$. Now the minimality assumption of the second statement of the lemma gives us $\cod(P_h) \leq \cod(P_{ij})$, so the Connectedness Lemma implies that the inclusion $P_{ij} \cap  P_h \subseteq P_{ij}$ is $c$-connected with
	\[c \geq n - \sum m_r - 2(m_{hi} + m_{hj}) \geq 2m_{ij} \geq m_{ij} + 2.\]
By Lemma \ref{lem:uptoc}, $P_{ij}$ is a rational $\s$, $\CP$, or $\HP$ up to degree $m_{ij} + 2$. Since $P_i \cap P_j \subseteq P_{ij}$ has codimension $m_{ij}$, Lemma \ref{lem:k1plus2k2} implies that $P_{ij}$ is a rational $\s$, $\CP$, or $\HP$ up to degree $\dim(P_{ij}) - 2m_{ij} + 3$. Now the Connectedness Lemma implies that the inclusion $P_i \cap P_j \subseteq P_{ij}$ is $(\dim P_{ij} - 2 m_{ij} + 2)$-connected, so $P_i \cap P_j$ is a rational $\s$, $\CP$, or $\HP$ up to degree $\dim P_{ij} - 2m_{ij} + 2$ by Lemma \ref{lem:uptoc}. Finally, since
	\[\dim P_{ij} - 2m_{ij} + 2 = \dim(P_i \cap P_j) - m_{ij} + 2 \geq \tfrac 1 2 \dim(P_i \cap P_j) + 2,\]
Lemma \ref{lem:UpToHalfPD} implies that $P_i \cap P_j$ is a rational $\s$, $\CP$, or $\HP$.

Finally, assume that $m_{ij} = 0$. Note that $F^f = P_{ij} \cap P_h$ in this case, so the inclusion $F^f \subseteq P_{ij}$ is $c_1$-connected with
	\[c_1 \geq n - \sum m_r - 2(m_{hi} + m_{hj}) - m_{123} \geq 2m_{123} \geq m_{123} + 2.\]
Since $P_i \cap P_j \subseteq P_{ij}$ has codimension $m_{123}$, we argue as in the previous paragraph to conclude that $P_{ij}$ and hence $P_i \cap P_j$ is a rational $\s$, $\CP$, or $\HP$ up to degree $c_2$ with 
	\[c_2 \geq \dim(P_{ij}) - 2m_{123} + 2 \geq \frac 1 2 \dim(P_i \cap P_j) + 2\]
and hence in all degrees by Lemma \ref{lem:UpToHalfPD}. This concludes the proof of the lemma in Case 2.	

\bigskip

We now finish the proof of the lemma assuming that neither Case 1 nor Case 2 occurs. In particular, we may assume $m_{123} \neq 0$. Additionally we may assume that $m_h$ and $m_{ij}$ are both zero or both non-zero for all $h \in \{1,2,3\}$ where $\{h,i,j\} = \{1,2,3\}$. Since there are at least four non-zero weights, we get non-zero values for at least two values of $h$. Finally since the representation has no finite isotropy groups with even order, we may assume (after permuting the $\sigma_i$) that $m_1 = m_{23} = 0$ and that the other five multiplicities are nonzero.

Now we look inside the fixed-point component $P_{123} = M^{\sigma_1\sigma_2\sigma_3}_x$. Since $m_1 = 0$, intersecting with $P_2$ and $P_3$ give submanifolds of $P_{123}$ that intersect transversely and that have codimension $m_{12}$ and $m_{13}$, respectively. Since $\dim(M^{\gT^3}_x) \geq \sum m_{rs} + 2m_{123}$, the Four-Periodicity Theorem and Lemma \ref{lem:NoSxHP} imply that $M^{\gT^3}_x$ is a rational $\s$, $\CP$, or $\HP$. Since $m_1 = 0$, $M^{\gT^3}_x = P_2 \cap P_3$, so the proof of the lemma is complete. 
\end{proof}

By the above discussion we have proved the theorem in the case that the involutions \(\iota_1,\iota_2,\iota_3\) with minimal codimensional fixed-point sets are linearly independent. We assume therefore that \(\iota_1,\iota_2,\iota_3\) are linearly dependent, which is equivalent to the property that their product is the identity element. We assume moreover that $N_4 = M^{\iota_4}_x$ has codimension strictly larger than $k_3$, since otherwise we could swap the roles of $\iota_3$ and $\iota_4$ and proceed as in the linearly independent case because the $\iota_i$ are pairwise distinct.

Linear dependence implies that $\iota_1$, $\iota_2$, and $\iota_3$ are contained in a two-torus. Hence there is a four-torus \(\gT^4\) that acts without even order finite isotropy groups and is complementary to the two-torus. We may assume the $\gT^4$-action on \(N_1\cap N_2\) is almost effective, since otherwise $N_1 \cap N_2$ is fixed by a $\gT^3$-action, which would imply that we could replace $\iota_3$ by another involution that is linearly independent from $\iota_1$ and $\iota_2$ and whose fixed-point set has codimension at most $\tfrac 1 2 \cod(N_1 \cap N_2) < \tfrac 2 3 \cod(N_1 \cap N_2) \leq k_3$.
In particular the two-torus also acts without finite isotropy groups of even order on \(M\).

By Theorem~\ref{thm:cd}, there exists a non-trivial involution \(\iota_4' \in \gT^4\) with \(\cod M^{\iota_4'} \leq \frac{4}{9}n\). Replacing $\iota_4$ by $\iota_4'$, if necessary, we may assume that $N_4^{n-k_4}$ satisfies 
	\[k_4 \leq \tfrac 4 9 n.\]
Since $\iota_1$, $\iota_2$, and $\iota_4$ are linearly independent, and since the codimensions of their fixed-point sets satisfy
	\[k_1 + k_2 + k_4 \leq \tfrac 1 2(k_1 + \ldots + k_6) \leq n,\]
the first part of Lemma \ref{lem:LinearlyIndependent} implies that $N_1 \cap N_2 \cap N_4$ is a rational cohomology $\s$, $\CP$, or $\HP$. By Claim 2 at the beginning of the proof of the theorem, it suffices to show that $N_1 \cap N_2$ is a rational cohomology $\s$, $\CP$, or $\HP$.

To do this, we need more notation. Since $\iota_1$, $\iota_2$, and $\iota_3$ are linearly dependent, the two-fold intersections $N_i \cap N_j$ for $1 \leq i < j \leq 3$ coincide. Moreover, the codimensions
	\[m'_i = \cod(N_1 \cap N_2 \subseteq N_i)\]
for $i \in \{1,2,3\}$ satisfy \(m'_1\geq m'_2\geq m'_3\) and are the multiplicities of the three irreducible subrepresentations of the \(\mathbb{Z}^2_2\)-representation on the normal space to \(N_1\cap N_2\). Note in particular that
	\[m'_3 \leq \frac 1 3 \sum m'_i = \frac 1 3 k,\]
where we define 
	\[k = \cod(N_1 \cap N_2) = m'_1 + m'_2 + m'_3.\]
With this notation, we prove the following claims:

\bigskip
\noindent
{\bf Claim 3:} If $m'_3 \in \{0,2\}$, then $N_1 \cap N_2$ is a rational $\s$, $\CP$, or $\HP$.

\begin{proof}[Proof of Claim 3]
If $m'_3 = 2$, then $N_1 \cap N_2$ is a codimension-two fixed-point component of a circle action on $N_3$, so the claim holds by Remark \ref{rem:GroveSearleCodim2}. If $m'_3 = 0$, then $N_3$ is the transverse intersection of $N_1$ and $N_2$, and we have $2\cod(N_1 \cap N_2) = k_1 + k_2 + k_3 \leq n$. Hence the Connectedness-Lemma \ref{thm:CL}, the Four-Periodicity-Theorem \ref{thm:4PT} and Lemma \ref{lem:NoSxHP} imply the claim.
\end{proof}

\bigskip
\noindent
{\bf Claim 4:} If $m'_3 \geq 4$, then $N_1 \cap N_2$ is a rational $\s$, $\CP$, or $\HP$ up to degree $m'_3 + 1$ if $n$ is odd and up to degree $m'_3+2$ if $n$ is even.

\begin{proof}[Proof of Claim 4]
To begin, we need the following estimate:
	\[n - k - k_4 \geq m'_3 + 1.\]
To prove it, we use the equality $2k = \sum_{i=1}^3 k_i$ and the inequality $k_4 \leq k_5 \leq k_6$ to estimate
	\[n - k - k_4 \geq n - \frac 1 2 \sum_{i=1}^3 k_i - \frac 1 3 \sum_{i=4}^6 k_i.\]
Next, we use the inequalities $k_i < k_{i+3}$ for $i \in \{1,2,3\}$ to obtain
	\[n - k - k_4 \geq n - \frac 1 3 \sum_{i=1}^3 k_i - \frac 1 2 \sum_{i=4}^6 k_i + 1.\]
Finally, we use the inequalities $\sum_{i=1}^6 k_i \leq 2n$ and $k_i \geq k_1 \geq 2m'_3$ for $i \in \{1,2,3\}$ to obtain the desired estimate:
	\[n - k - k_4 \geq \frac 1 6 \sum_{i=1}^3 k_i + 1 \geq m'_3 + 1.\]
We now finish the proof of the claim.

\bigskip
\noindent
{\bf Case 1:} $k_4 \leq k$.

The inclusion $N_1 \cap N_2 \cap N_4 \subseteq N_1 \cap N_2$ is $c$-connected with
	\[ c \geq n - k - k_4 \geq m'_3 + 1.\]
Moreover $c \geq m'_3 + 2$ if $n$ is even since $k$, $k_4$, and $m'_3$ are even. Hence $N_1 \cap N_2$ is a rational $\s$, $\CP$, or $\HP$ up to degree $m'_3+1$ and up to $m'_3+2$ if $n$ is even.

\bigskip
\noindent
{\bf Case 2:} $k \leq k_4$.

This time, we prove that the three inclusions $N_1 \cap N_2 \cap N_4 \subseteq N_4$, $N_4 \subseteq M$, and $N_1 \cap N_2 \subseteq M$ are $(m'_3+1)$-connected if $n$ is odd and $(m'_3+2)$-connected if $n$ is even. This suffices to pull the $\s$, $\CP$, or $\HP$ cohomological type in degrees up to $m'_3+1$ in odd dimensions or $m'_3+2$ in even dimensions up to $N_4$ and $M$ and then push it back down to $N_1 \cap N_2$. We step through the required estimates one at a time.

The first inclusion is $c_1$-connected with
	\[c_1 \geq n - k - k_4,\]
so this is at least $m'_3+1$ in general and $m'_3 + 2$ if $n$ is even just as it was in Case 1. The second inclusion is $c_2$-connected with
	\[c_2 \geq n - 2k_4 + 2
		\geq n - \frac 3 2 \of{\frac 4 9 n} - \frac 1 2 \of{\frac 1 3 \sum_{i=4}^6 k_i} + 2,\]
where we have partially estimated $k_4$ using the upper bound $\tfrac 4 9 n$ and partially using the upper bound $\tfrac 1 3 \sum_{i=4}^6 k_i$. Using again that $3k_1 + \sum_{i=4}^6 k_i \leq \sum_{i=1}^6 k_i \leq 2n$ and $k_1 \geq 2m'_3$, this implies
	\[c_2 \geq \frac 1 3 n - \frac 1 6 \of{2n - 3k_1} + 2 \geq m'_3 + 2,\]
as required. Finally, the third inclusion $N_1 \cap N_2 \subseteq M$ has codimension $k$, which is at most $k_4 = \cod(N_4 \subseteq M)$, so this inclusion is also $(m'_3+2)$-connected by the same estimate as for $N_4$. This concludes the proof of the claim.
\end{proof}

We now finish the proof of the theorem given Claims 3 and 4. By Claim 2 at the beginning of the proof, it suffices to show that $N_1 \cap N_2$ is a rational $\s$, $\CP$, or $\HP$. In the situation of Claim 3, this holds immediately, so we may assume we are in the situation of Claim 4. 

The inclusion $N_1 \cap N_2 \subseteq N_3$ is $c$-connected with
	\[c \geq \dim(N_3) - 2m'_3 + 2 = (n-k) - m'_3 + 2 \geq 2m'_3 + 2,\]
where in the last step we used that $3m'_3 \leq \sum m'_i = k$ and that $2k = k_1 + k_2 + k_3 \leq n$. In addition, we have
	\[4m'_3 \leq 2k_1 \leq n - k_3 = \dim(N_3),\]
so Lemma \ref{lem:k1plus2k2} implies that $N_3$ is a rational $\s$, $\CP$, or $\HP$. By Bredon's theorem, the same conclusion holds for $N_1 \cap N_2$. This completes the proof of the theorem.

\subsection{Upgrading four-periodicity to standard cohomology}

The following lemma uses the assumption on isotropy groups, together with the fact that fixed-point components of $\gT^4$-actions are rational cohomology $\s$, $\CP$, or $\HP$ to upgrade four-periodicity in the proof of Theorem \ref{thm:t9}.

\begin{lemma}\label{lem:NoSxHP}
Let $M^n$ be a closed, oriented, positively curved Riemannian manifold. Assume that $\gT^6$ acts isometrically on $M$, has a fixed point $x$, and has the property that all isotropy groups near $x$ have an odd number of components. For any subtorus $\gT^d \subseteq \gT^6$, if the \(\gT^d\)-fixed-point component $F^f$ at $x$ has $f \leq 6$ or four-periodic rational cohomology, then $F^f$ is a rational cohomology $\s^f$, $\CP^{\frac f 2}$, or $\HP^{\frac f 4}$.
\end{lemma}

\begin{proof}
Let $F^f$ be the fixed-point component at $x$ of a subtorus $\gT^d$ with $0 \leq d \leq 6$. We may assume that $\gT^d$ is the kernel of the induced $\gT^6$-action on $F^f$.

If $d \geq 4$, then $F^f$ is a rational cohomology $\s^f$, $\CP^{\frac f 2}$, or $\HP^{\frac f 4}$ by Theorem \ref{thm:t4}. Assuming then that $d \leq 3$, we can apply Theorem \ref{thm:cd} to choose subtori 
	\[\gT^d \subseteq \gT^{d+1} \subseteq \ldots \subseteq \gT^4\]
such that the corresponding fixed-point components
	\[F^f = F_d \supseteq F_{d+1} \supseteq \ldots \supseteq F_4\]
have the property that
	\[\cod(F_{i+1} \subseteq F_i) \leq c(3) \dim F_i = \tfrac 1 2 \dim F_i.\]
By making our choices to have minimal codimension at each step, we may assume without loss of generality that the induced $\gT^6$-action on $F_i$ has kernel isomorphic to $\gT^i$. Since, in addition, the $\gT^6$-action has a fixed-point, the induced action by $\gT^2 \cong \gT^6/\gT^4$ on $F_4$ is effective, so we have $\dim(F_4) \geq 4$. 

Moreover, we may assume that $f \geq 8$, since otherwise $F_4 \subseteq F$ has codimension two and hence the result follows from Remark \ref{rem:GroveSearleCodim2}. By four-periodicity and Poincar\'e duality, it suffices to show that $b_3(F) = 0$ in even dimensions and that $F$ is not $\s^p \times \HP$ for some $p \in \{2,3\}$.

The Connectedness Lemma implies that the inclusions $F_{i+1} \subseteq F_i$ are $2$-connected in general and $3$-connected if $n$ is odd. Recall for the last statement that $n \equiv \dim F_i \bmod 2$ for all $i$ since the codimensions are even. In particular, in odd dimensions, we conclude that $H^1(F) \cong H^1(F_4) = 0$ and $H^3(F) \cong H^3(F_4) = 0$ and hence that $F$ is a rational sphere. In even dimensions, we conclude that the squaring map $H^2(F) \to H^4(F)$ is injective. In particular, the cohomology in even degrees cannot look like $\s^2 \times \HP^{\frac{f-2}{4}}$, so this excludes both the possibility that $F^f$ is a rational $\s^2 \times \HP^{\frac{f-2}{4}}$ and the possibility that $b_3(F) \neq 0$. Hence $F^f$ is a rational $\s^f$, $\CP^{\frac f 2}$, or $\HP^{\frac f 4}$.
\end{proof}

\subsection{Alternate proof of (1)}

Theorem \ref{thm:t9} for a $\gT^9$-action admits a short, direct proof that does not require the more involved result for $\gT^6$-actions. We present it here.

We are given a closed, oriented, positively curved Riemannian manifold $M^n$ with an isometric $\gT^9$-action such that there is a fixed point $x \in M$ and no finite isotropy groups of even order near $x$. By repeated applications of Theorem \ref{thm:cd}, we obtain a sequence of fixed-point components
	\[F_3 \subseteq F_2 \subseteq F_1 \subseteq F_0 = M\]
the property that
	\[k_{i+1} = \cod\of{F_{i+1} \subseteq F_{i}} \leq c(9-i)\dim F_i\]
for $0 \leq i \leq 2$, where $c(9) = \tfrac 1 4$, $c(8) = \tfrac 2 7$, and $c(7) = \tfrac{3}{10}$ as in Theorem \ref{thm:cd}. We may assume moreover that each $F_i$ is chosen to have maximal dimension in $F_{i-1}$ so that, in particular, the kernel of the induced $\gT^9$-action on $F_3$ equals a subtorus $\gT^3$.

We apply the $\gS^1$-Splitting Theorem to the isotropy representation of $\gT^3$ on the normal space to $F_3$. This gives rise to a circle $\gS^1 \subseteq \gT^3$ such that the induced action by $\gT^3/\gS^1 \cong \gT^2$ is effective and splits as a product action on the fixed-point component $N = M^{\gS^1}_x$. In particular, $F_3$ is the transverse intersection of two fixed-point components inside $N$. 

By maximality of $\dim F_1$, we have $\dim F_1 \geq \dim N$. On the other hand, we have
	\[\dim F_3 \geq \tfrac{7}{10} \dim F_2 \geq \tfrac 1 2 \dim F_1,\]
so $\dim F_3 \geq \tfrac 1 2 \dim N$. This is sufficient for the Four-Periodicity Theorem, and we find that $F_3$ has four-periodic rational cohomology. By Lemma \ref{lem:NoSxHP}, we find moreover that $F_3$ is a rational cohomology $\s$, $\CP$, or $\HP$.

We now use the Connectedness Lemma to pull this cohomological information up to $M$. First, the inclusion $F_3 \subseteq F_2$ is $c$-connected with 
	\[c \geq \dim F_2 - 2k_3 + 2 \geq \max(k_3,k_2) + 2.\]
Hence Lemmas \ref{lem:uptoc} and \ref{lem:k1plus2k2} imply that $F_2$ is a rational $\s$, $\CP$, or $\HP$ up to degree $k_2 + 2$. 

Repeating this argument for the inclusion $F_2 \subseteq F_1$, which is $c'$-connected with
	\[c' \geq \dim F_1 - 2k_2 + 2 \geq \max(k_2,k_1) + 2,\]
we find that $F_1$ is a rational $\s$, $\CP$, or $\HP$ up to degree $k_1 + 2$.

Finally since the inclusion $F_1 \subseteq M$ is $c''$-connected with
	\[c'' \geq n - 2k_1 + 2 \geq \max\of{k_1+2, \tfrac n 2 + 2},\]
we find that $M$ is a rational $\s$, $\CP$, or $\HP$ up to degree $\tfrac n 2 + 2$. By Lemma \ref{lem:UpToHalfPD}, this is sufficient by Poincar\'e duality to conclude that $M$ is a rational cohomology $\s$, $\CP$, or $\HP$ in all degrees.


\end{document}